\documentclass[reqno,12pt,letterpaper]{amsart}
\usepackage{amsmath,amssymb,amsthm,graphicx,mathrsfs,url}
\usepackage[usenames,dvipsnames]{color}
\usepackage[colorlinks=true,linkcolor=Red,citecolor=Green]{hyperref}
\usepackage{amsxtra}
\usepackage{sidecap}
\usepackage[small]{caption}

\DeclareMathOperator{\const}{const}

\def\?[#1]{\textbf{[#1]}\marginpar{\Large{\textbf{??}}}}

\let\epsilon=\varepsilon 

\newcommand{\RR}{{\mathbb R}}
\newcommand{\TT}{{\mathbb T}}
\newcommand{\NN}{{\mathbb N}}
\newcommand{\CC}{{\mathbb C}}
\newcommand{\HH}{{\mathbb H}}
\newcommand{\ZZ}{{\mathbb Z}}
\newcommand{\SP}{{\mathbb S}}
\newcommand{\CI}{{{\mathcal C}^\infty}}
\newcommand{\bCI}{{\bar{\mathcal C}^\infty}}
\newcommand{\dCI}{{\dot{\mathcal C}^\infty}}
\newcommand{\CIc}{{{\mathcal C}^\infty_{\rm{c}}}}

\newtheorem{thm}{Theorem}
\newtheorem{prop}{Proposition}
\newtheorem{defi}[prop]{Definition}
\newtheorem{conj}{Conjecture}

\numberwithin{equation}{section}

\DeclareMathOperator{\Res}{Res}
\DeclareMathOperator{\Spec}{Spec}
\DeclareMathOperator{\comp}{comp}

\let\Im=\Imag
\DeclareMathOperator{\loc}{loc}

\DeclareMathOperator{\PSL}{PSL}
\DeclareMathOperator{\rank}{rank}

\let\Re=\Real

\DeclareMathOperator{\supp}{supp}
\DeclareMathOperator{\spt}{supp}
\DeclareMathOperator{\vol}{vol}
\DeclareMathOperator{\WF}{WF}
\DeclareMathOperator{\tr}{tr}

\def\indic{\operatorname{1\hskip-2.75pt\relax l}}

\title[Scattering resonances]{Mathematical study of scattering resonances}
\author{Maciej Zworski}
\email{zworski@math.berkeley.edu}
\address{Department of Mathematics, University of California,
Berkeley, CA 94720, USA}

\begin{document}


\maketitle

\tableofcontents

\section{Introduction}

\begin{figure}
\begin{center}
\includegraphics[width=4.1in]{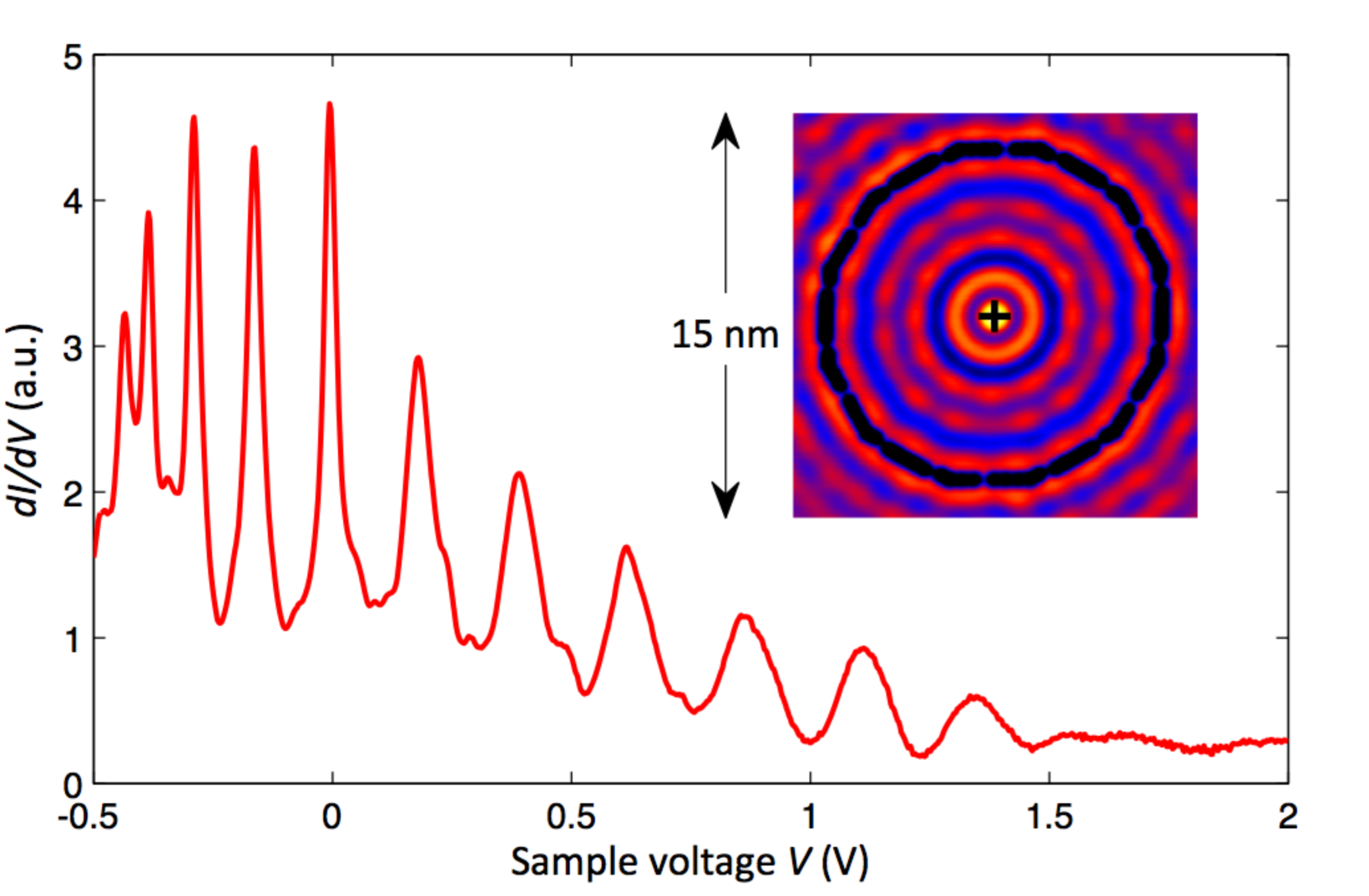}
\caption{\label{f:hm} A scanning tunneling microscope (STM) 
spectrum is a plot of $ dI/dV $ ($ I $
  being the current)  as a function of
bias voltage $ V$. According to the basic theory of STM, this reflects the sample density of
states as a function of energy with respect to the Fermi energy (at $ V = 0$). This spectrum
shows the series of surface state electron {\em resonances} in inside a circular quantum
corral on Cu(111) -- see Figure \ref{f:core} for a visualization of the relation 
between the peaks and the complex resonance poles. The bulk bands contribute to a gradually varying background in this
spectrum. The setpoint was $ V_0 = 1 {\rm V} $ and $ I_0 = 10 {\rm nA}
$  and the modulation voltage was
$ V_{\rm rms} = 4 {\rm mV} $. Inset: a low-bias topograph of the corral
studied ($ 17 \times 17 {\rm nm}^2$, $ V = 10 \rm{mV} $,
$ I = 1 {\rm nA}$). The corral is made from 84 CO molecules adsorbed to Cu(111) and has an 
average radius of $ 69.28 $ {\AA} -- the date and images are from 
the Manoharan Lab at Stanford \cite{Moo}. The large amplitude in the center of the 
top graph is a reflection of the sharp peak seen in the spectrum at $ V = 0$.
For another quantum corall, in the shape of the Bunimovich stadium and made out 
of iron atoms, see Figure~\ref{f:galk}.}
\end{center}
\end{figure}

Scattering resonances appear in many branches of mathematics,
physics and engineering. They generalize eigenvalues or bound
states to systems in which energy can scatter to infinity.
A typical state has then a rate of oscillation (just as a bound
state does) and a rate of decay. Although this notion is 
intrinsically dynamical, an elegant mathematical formulation 
comes from considering meromorphic continuations of Green's functions
or scattering matrices. The poles of these meromorphic continuations
capture the physical information by identifying the rate of 
oscillations with the real part of a pole and the rate of decay
with its imaginary part. The resonant state, which is the corresponding
wave function, then appears in the residue of the meromorphically 
continued operator. An example from pure mathematics is given 
by the zeros of the Riemann zeta function: they are  
the resonances of the Laplacian on the modular surface. 
The Riemann hypothesis then states that the decay rates for the 
modular surface are all either $ 0 $ or $ \frac 14$.
A standard example from physics is given by shape resonances
created when the interaction region is separated from free space
by a potential barrier. The decay rate is then exponentially 
small in a way depending on the width of the barrier.

In this article we survey some foundational and some recent 
aspects of the subject selected using the perspective and experience
of the author. Proofs of many results can be found in the
monograph written in collaboration with Semyon Dyatlov 
\cite{res} to which we provide frequent references.

What we call scattering resonances appear under different names in different
fields: in quantum scattering theory they are called {\em quantum resonances}
or {\em resonance poles}. In obstacle scattering, or more generally 
wave scattering, they go by {\em scattering poles}. In general relativity
the corresponding complex modes of gravitational waves are known 
as {\em quasi-normal modes}. 
The closely related poles of 
power spectra of correlations in chaotic dynamics are called 
{\em Pollicott--Ruelle} resonances. We will discuss mathematical results
related to each of these settings stressing unifying features. 

The survey is organized as follows:
\begin{itemize}

\vspace{-0.15in}

\item in the introduction we provide a basic physical motivation 
from quantum mechanics, discuss the case of the wave equation in 
one dimension, intuitions behind semiclassical study of resonances
and some examples from modern science and engineering;

\item in \S \ref{pot3} we present scattering by bounded compactly supported
potentials in three dimensions and prove meromorphic continuation of the resolvent (Green function), an upper
bound on the counting function, existence of resonance free regions
and expansions of waves; we also explain the method 
of complex scaling in the elementary
setting of one dimension; finally we discuss recent progress and open problems
in potential scattering;

\item we devote \S \ref{srr} to a survey of recent results organized
around topics introduced in the special setting of \S \ref{pot3}: meromorphic
continuation for asymptotically hyperbolic spaces, fractal upper bounds
in physical and geometric settings, resonance free strips in 
chaotic scattering and resonance expansions; we also provide some
references to recent progress in some of the topics not covered in this survey;

\item \S \ref{dsPR} surveys the use of microlocal/scattering theory methods 
in the study of chaotic dynamical systems. Their introduction by 
Faure--Sj\"ostrand \cite{fa-sj} and Tsujii \cite{Ts}
led to rapid progress which included a microlocal proof of Smale's conjecture
about dynamical zeta functions \cite{zeta}, first proved
shortly before by Giulietti--Liverani--Pollicott \cite{glp}. We review
this and other
results, again related to upper bounds, resonance free strips and
resonances expansions.

\end{itemize}

\subsection{Motivation from quantum mechanics}

\begin{figure}[htbp]
\centering
\includegraphics[width=12.5cm]{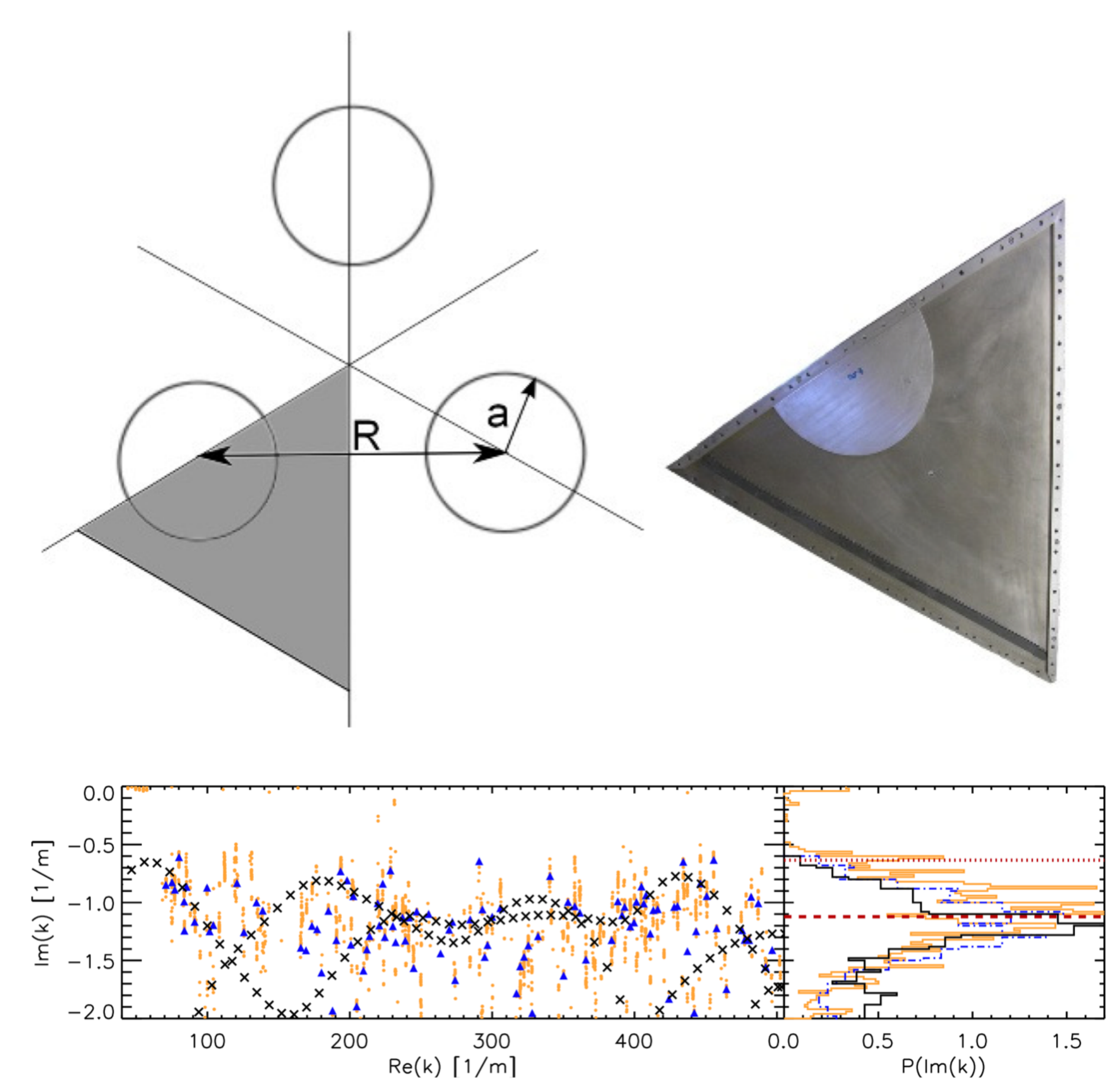} 
\caption{The experimental set-up for the study of the fractal Weyl law \cite{Potz}
and of  the 
distribution of resonance widths \cite{Bark}. On the left panel 
below the resonances for the aspect ration $5.5$ in the complex $k$-plane are shown as well as the distribution of the imaginary parts of $k$ in the right panel. The orange clouds correspond to all resonances resulting in a good reconstruction as well as the orange histogram. Note that many orange dots might overlay each other. Blue triangles and blue dashed-dotted histogram describe the set belonging to the best reconstruction. Black crosses are the numerically calculated poles and the solid black histogram the corresponding distribution. The dotted red line in the right panel is $ P(1/2)$ , the red dashed line half of the classical decay rate $ P(1)/2$. Here $ P ( s ) $ is
the topological pressure associated to the unstable Jacobian.}
\label{f:kuhl}
\end{figure}

In quantum mechanics a particle is described by a wave function 
$ \psi $ which is normalized in $ L^2 $, $ \| \psi \|_{L^2 } = 1 $.
The probability of finding the particle in a region $ \Omega $ is
the given by the integral of $ |\psi( x ) |^2 $ over $ \Omega $. 
A pure state is typically an eigenstate of a Hamiltonian $ P$ 
and hence the evolved state is given by $ \psi ( t) := e^{- i t P} \psi 
= e^{ - it E_0 } \psi $ where $ P \psi = E_0 \psi $. In 
particular the probability density does not change when the state
is propagated.

\begin{figure}[htbp]
\centering
\includegraphics[width=4.2in]{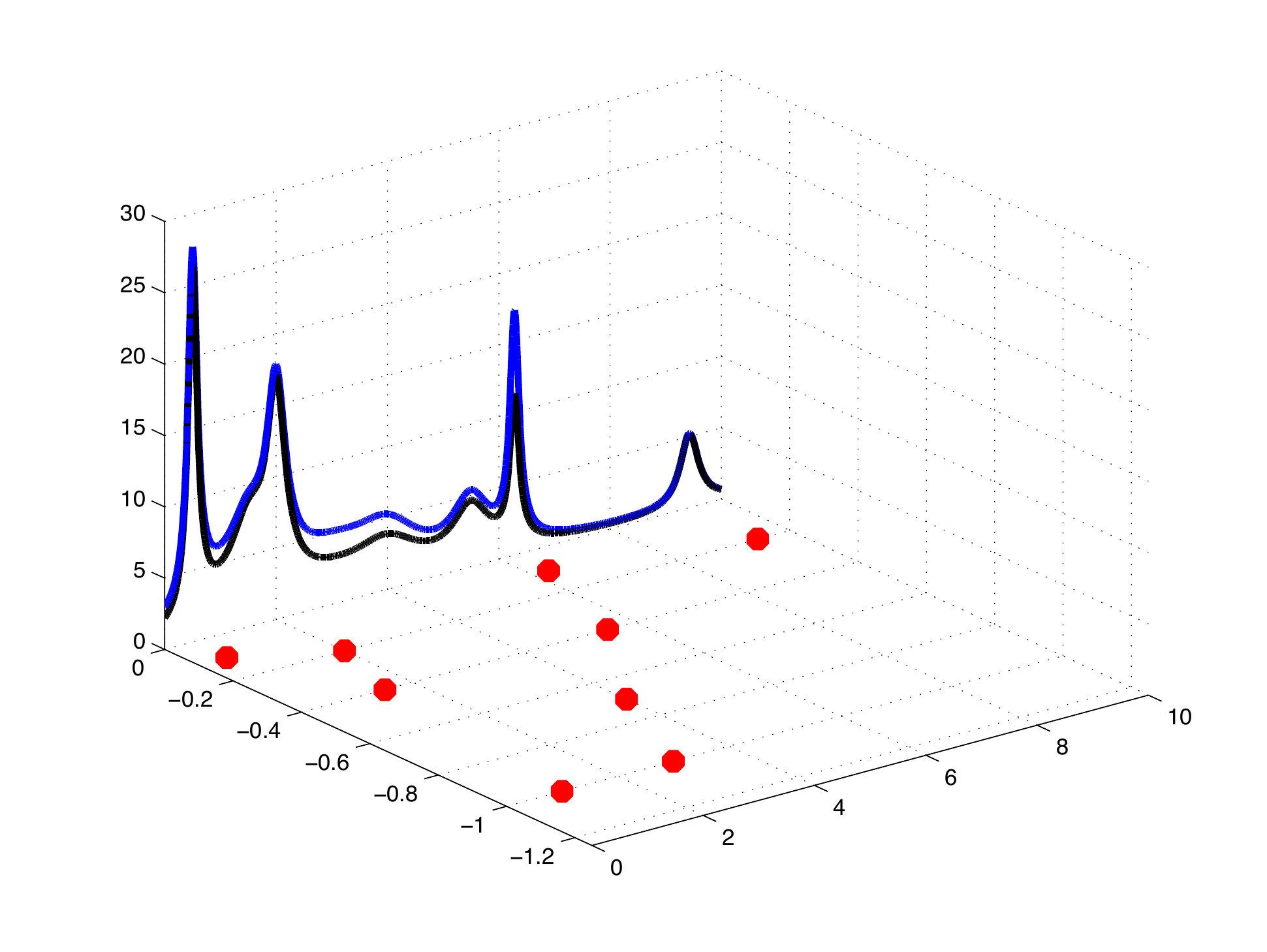}
\caption{If $ U ( t ) $ is a propagator and  $ f,  g$ are states then
the correlation (the evolution of $ f $ measured by $ g $) is given by
$ \rho_{f,g} ( t) = \langle U ( t )  f , g \rangle $.
For instance the evolution
could come from a flow $ \varphi_t : M \to M $ on a compact manifold,
$ U (t ) f ( x ) =  f (\varphi_{-t} (x) ) $.
The power spectrum of correlations is given by
$ \widehat \rho_{ f ,g } ( \lambda ) := \int_0^\infty \rho_{f , g} ( t )      
e^{ i \lambda t } dt  $.
Resonances are the poles of $ \widehat \rho_{ f , g  } ( \lambda ) $
and are 
independent of $ f $ and  $g $.
The figure shows a schematic correspondence between the power
spectrum for different states and these poles:
the real part corresponds to the location of
a peak in the power spectrum and the imaginary to its width;
the $ x $ axis is the
real part (frequency $\lambda $), $ y $ axis the imaginary (rate of decay),
$ z $ axis $ | \hat \rho_{f,g} ( \lambda )|$. Unlike the
power spectrum which depends on $ f $ and $ g $,
the poles depend only on the system. }
\label{f:core}
\end{figure}

An example could be given by the Bloch electron in a molecular corral shown in
Figure \ref{f:hm}. However the same 
figure also shows that the measured states have non-zero ``widths'' -- 
the peak is not a delta function at $ E_0 $ -- 
and hence can be more accurately modeled by resonant states.
Scattering or quantum resonances are given as complex numbers $ E_0                      - i \Gamma/2$
and the following standard argument of the physics literature explains the meaning
of the real and imaginary parts: a time dependent pure resonant
state propagates according to
$  \psi(t) = e^{- itE_0 -t \Gamma  /2 }  \psi (0) $
so that the probability of survival beyond time $t$ is
$ p(t) = { |\psi(t)|^ 2}/{ | \psi(0)|^2 } = e^{- \Gamma t } $.
This explains why the convention for the imaginary part of a resonance
is $ \Gamma/2$.  Here we neglected the issue that $  \psi(0) \notin                     L^2 $ which is remedied
by taking the probabilities over a bounded interaction region.
In the energy representation the wave function is given
\[ \begin{split} \varphi (E) & := {\mathcal F}^{-1} \psi(E) 
=
\frac 1
{ \sqrt{2 \pi }} \int_0^\infty
e^{-itE_0 - t\Gamma/2 } \psi ( 0 ) dt
=
\frac{1}  {\sqrt{ 2\pi} i}
\frac{ \psi ( 0) }
{ E_0 - i \Gamma/2 - E} , \end{split} \]
which means that the probability density of the resonance with energy
E is proportional to the square of the absolute value of the right hand
side.
Consequently the probability density is
\begin{equation}
\label{eq:BW}
\frac{ 1 } { 2 \pi } \frac{ \Gamma }
{ (E -  E_0)^2 + (\Gamma /2)^2} dE , \end{equation}
and this Lorentzian is the famous Breit-Wigner distribution.
This derivation is of course non-rigorous and there are many 
mathematical issues:

\begin{itemize}

\item the state $ \psi ( 0 ) $ is not in $ L^2 $ -- it has physical 
meaning only in the ``interaction region"; that is justified 
differently in Euclidean and non-Euclidean scattering --  see \S\S  \ref{merc3},\ref{rae} 
and \S \ref{vasy} respectively.

\item it is not clear why the evolution of a physical state should have an 
expansion in terms of resonant states; that is justified using 
meromorphic continuation and asymptotic control of Green's function and
-- see \S \ref{expan};

\item one needs to justify the passage from the Fourier transform
$ t \mapsto E $ to the probability distribution \eqref{eq:BW};
that is done using the {\em scattering matrix} or {\em spectral measure}
-- see the end of \S \ref{expan}.

\end{itemize}

Returning to physical motivation, one should stress that in
practice there are many deviations from the simple formula \eqref{eq:BW},
especially at high energies and in the presence of overlapping
resonances. In Figure \ref{f:hm} we see clear Lorentzian peaks
and individual resonances can
be recovered. In the experiment shown in Figure \ref{f:kuhl} the resonances overlap
and the peaks in scattering data do not have the simple
interpretation using \eqref{eq:BW}. The density laws (Weyl laws) for
counting of resonant states are more complicated and richer
as they involve both energy
and rates of decay.
Even the leading term can be affected
by dynamical properties of the system -- see \S\S \ref{upper} and \ref{Weyl}.

\begin{figure}[htbp]
\begin{center}
\includegraphics[scale=0.3]{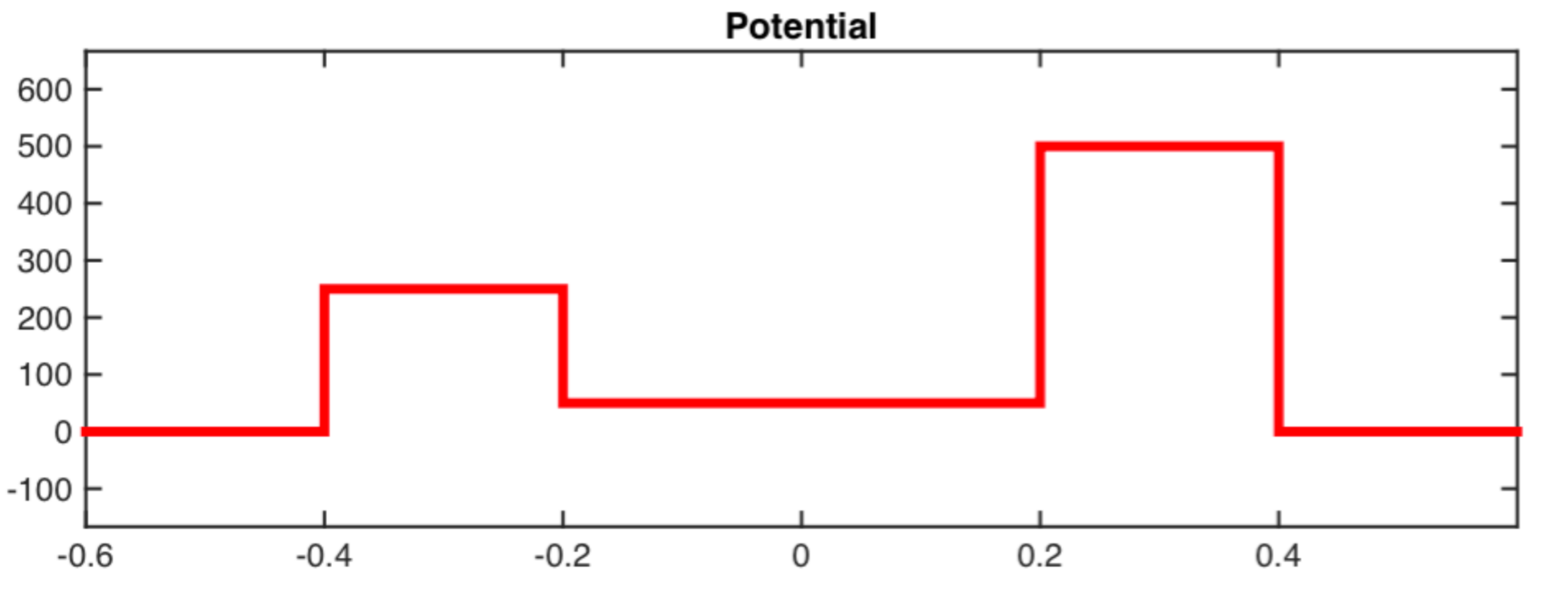}
\end{center}
\caption{\label{f:potin} A simple one dimensional potential with 
resonances shown in Figure \ref{f:polin}. Same potential is used
to see propagation, trapping and tunneling of a wave in Figure \ref{f:wavin}.}
\end{figure}

A more abstract version of Figure \ref{f:hm} is shown in 
Figure \ref{f:core} where correlations of evolved states
and unevolved states are shown with peaks corresponding to 
poles of their Fourier transforms (power spectra).

\subsection{Scattering of waves in one dimension}

\begin{figure}[htbp]
\begin{center}
\includegraphics[scale=0.3]{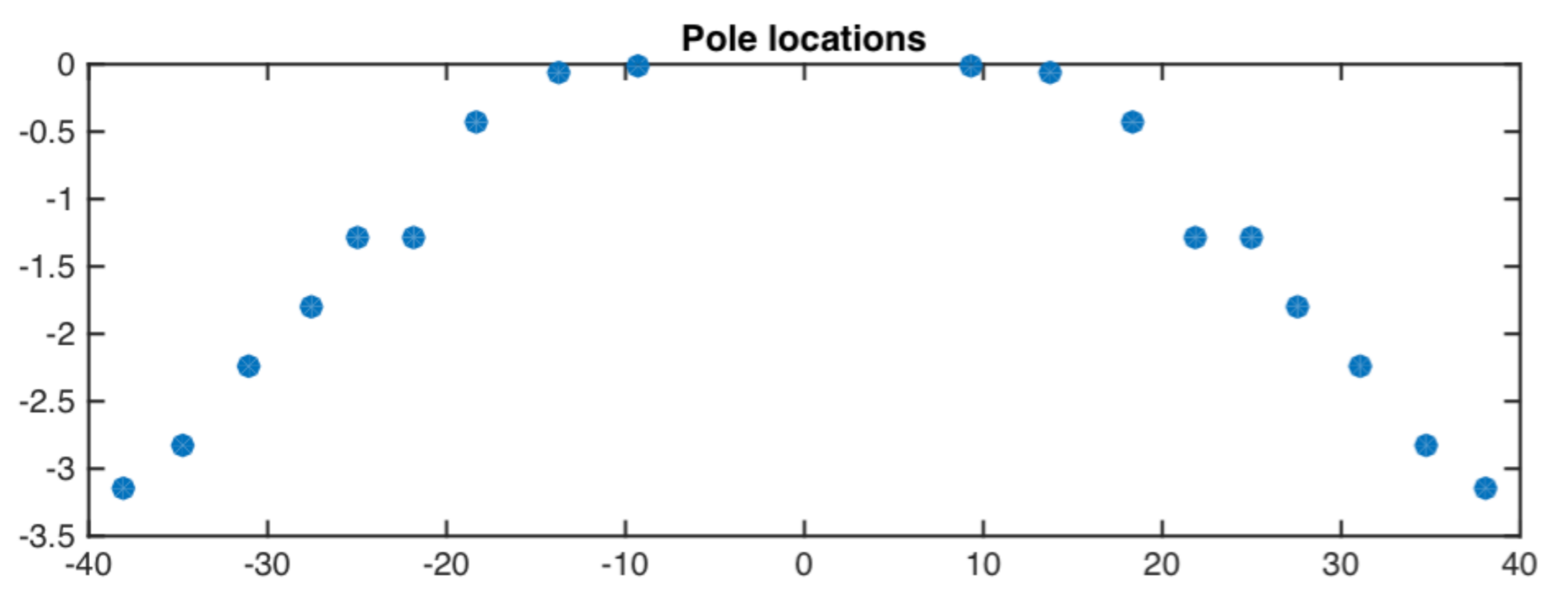}
\caption{\label{f:polin} Scattering poles of the potential shown in 
Figure \ref{f:potin}. They are computed very accurately using a 
Matlab code \cite{BiZ}.}
\end{center}
\end{figure}

A simple mathematical example
is given by scattering by compactly supported potentials
in dimension one, see Figure \ref{f:potin} for an example of such potential.
Scattering resonances are the rates of oscillations 
and decay of the solution of the wave equation and Figure \ref{f:wavin}
shows such a solution. 

\begin{figure}[htbp]
\begin{center}
\includegraphics[scale=0.7]{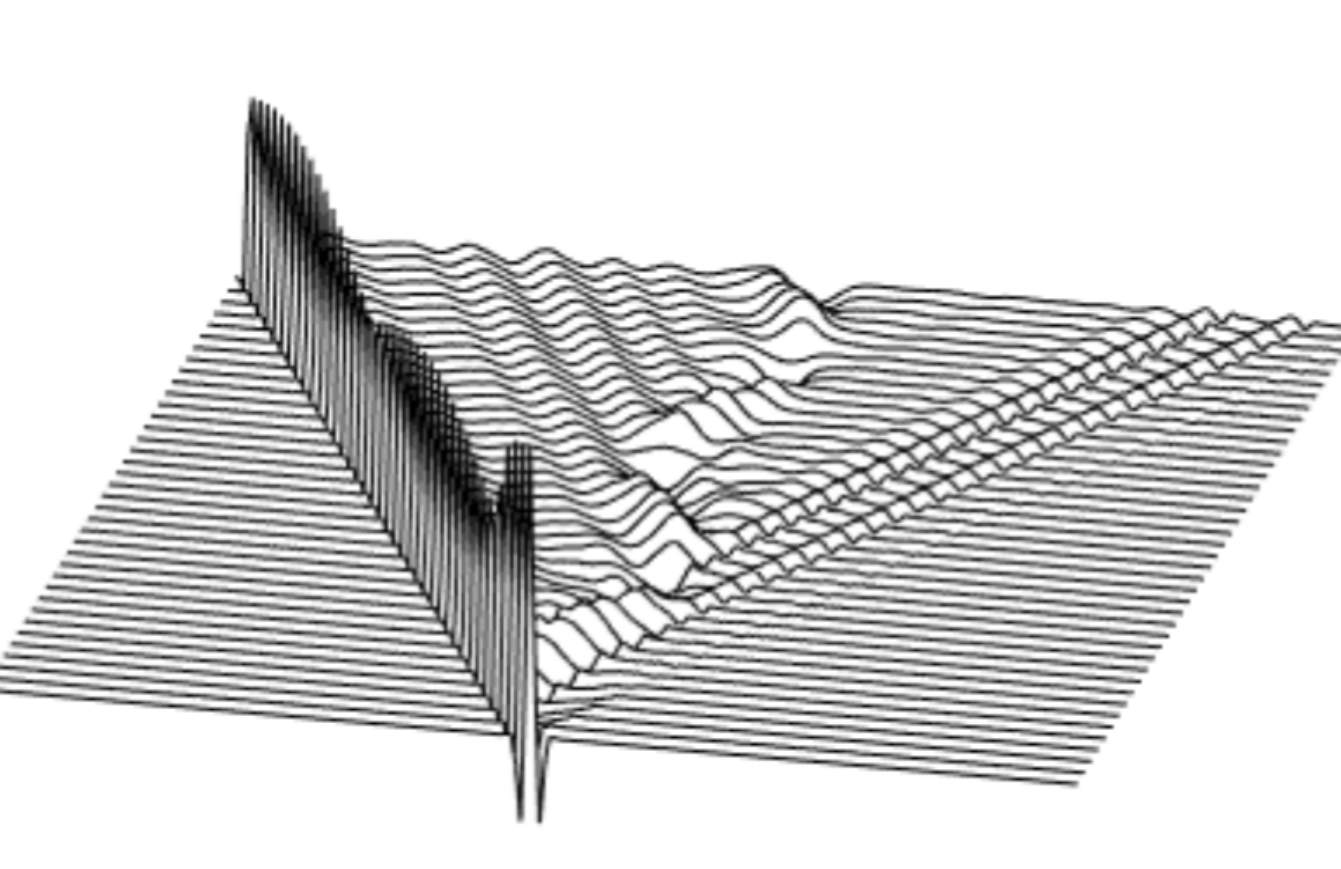}
\caption{\label{f:wavin} A solution of the wave equation
$ \partial_t^2 u - \partial_x^2 u + V u = 0 $ where $ V $ is 
shown in Figure \ref{f:potin}. The initial data is localized
near $ 0 $ and the time variable $ t $ points away from the viewer.}
\end{center}
\end{figure}

We see the main wave escape and some
trapped waves bounce in the well created by the potential and leak out. Figure \ref{f:ringin} shows the values
of the solution at one point. Roughly speaking if $ u ( t , x )$ is
a solution of the wave equation $ ( \partial_t^2 - \partial_x^2 + V( x ) ) u = 0 $
with localized initial data then 
\begin{equation}
\label{eq:uoftx}  u ( t, x ) \sim \sum_{ \Im \lambda_j > - A } e^{ - i \lambda_j t } u_j ( x ) + 
\mathcal O_K  ( e^{ - t A } ) , \ \ x \in K \Subset \RR , \end{equation}
where $ \lambda_j $ are complex numbers with $ \Im \lambda_j < 0$. They are
independent of the initial data and are precisely the scattering resonances -- 
see \S \ref{expan} for the precise statement. The
expansion \eqref{eq:uoftx} is formally related to the Breit--Wigner distribution \eqref{eq:BW} via the Fourier transform -- see Figure \ref{f:core}.

\begin{figure}[htbp]
\begin{center}
\includegraphics[scale=0.3]{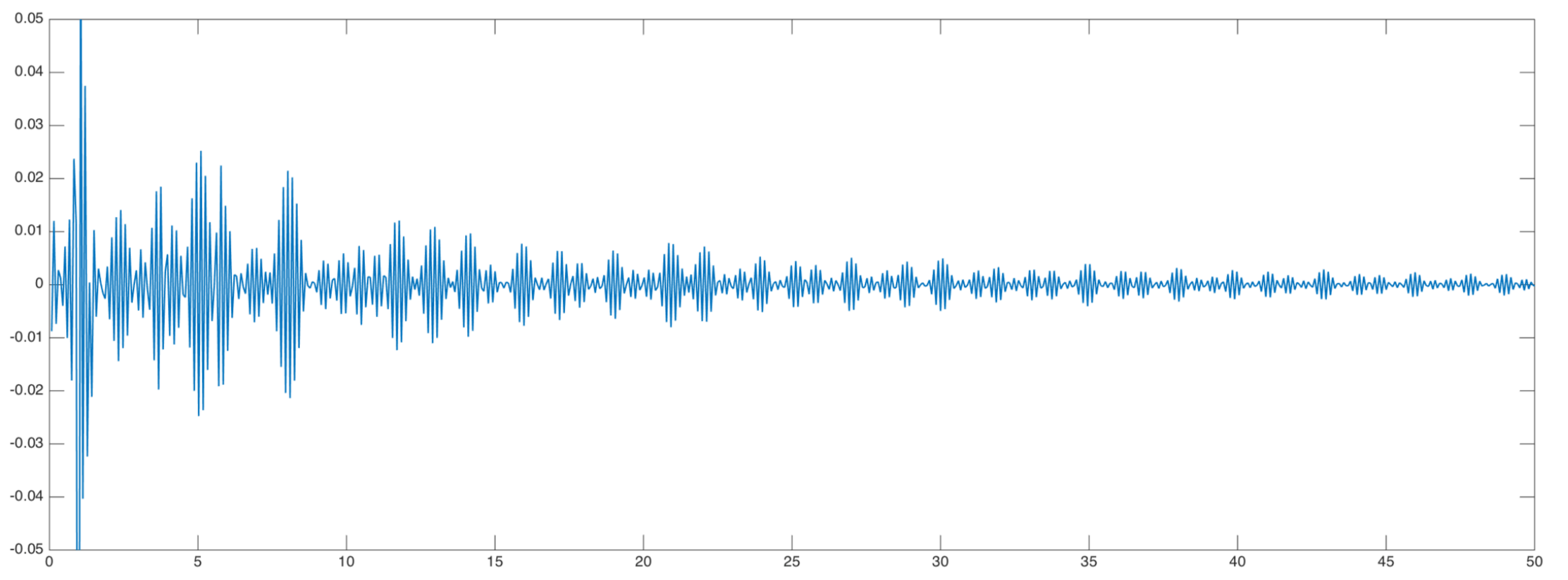}
\caption {\label{f:ringin} The plot of $ u ( 0 , t ) $  showing
oscillations and decay of the solution from Figure \ref{f:wavin} 
in the interaction region. The principal decay rate of the wave is 
determined by the resonance(s) closest to the real axis. For that we need
to justify the expansion \eqref{eq:uoftx} -- see Theorem \ref{t:5}.}
\end{center}
\end{figure}

Harmonic inversion methods, the first being the celebrated
Prony algorithm \cite{Prony}, 
can then be used to extract scattering resonances 
from solutions of the wave equation (as shown in Figure~\ref{f:ringin}; see for instance 
\cite{strauss}). The resulting complex numbers, that is the resonances
for the potential in Figure \ref{f:potin} are 
shown in Figure \ref{f:polin}.

\subsection{Resonances in the semiclassical limit}
\label{ressl}

For some very special systems resonances can be computed explicitely.
One famous example is the Eckart barrier: $ - \partial_x^2 + \cosh^{-2} x $.
It falls into the general class of P\"oschel--Teller potentials 
which can also be used to compute resonances of hyperbolic spaces or
hyperbolic cylinders -- see \cite{borth}.
Another example is given by scattering in the exterior of a spherical obstacle.
In this case scattering resonances
are zeros of Hankel functions which in odd dimensions are in fact zeros of
explicit polynomials  -- 
see \cite{Stef} and Figure \ref{f:sphere}.

\begin{figure}[htbp] 
\begin{center}
\includegraphics[width=5in]{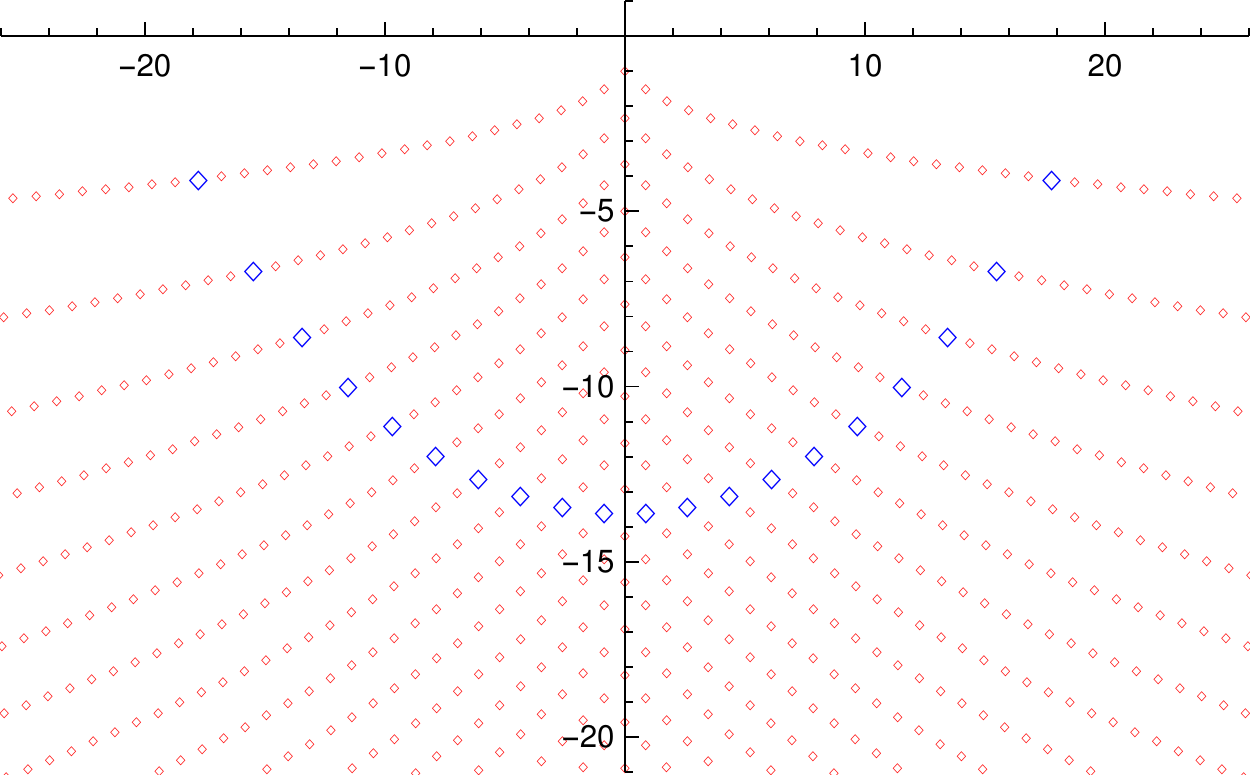}
\caption{\label{f:sphere}
  Resonances for the sphere in three dimensions. For each spherical 
momentum $ \ell $ they are given by solutions of  $ H^{(2)}_{\ell + \frac12} ( \lambda ) = 0 $ where $ H^{(2) }_\nu $ is the Hankel function of the second
kind and order $ \nu $. Each zero appears as a resonance of multiplicity 
$ 2 \ell +1 $: the resonances with $ \ell = 20 $ are highlighted.}
\end{center}
\end{figure} 

In general however it is impossible to obtain an explicit 
description of individual resonances. Hence we need to 
consider their properties and their distribution 
in asymptotic regimes. For instance in the case of obstacle scattering
that could mean the high energy limit. In the case of the sphere 
in Figure \ref{f:sphere} that corresponds to letting the angular momentum
$ \ell \to +\infty$. For a general obstacle that means considering 
resonances as $ |\lambda | \to + \infty $ and $ |\Im \lambda | \ll | \lambda |$.

The high energy limit is an example of a {\em semiclassical} limit.
To describe it we consider resonances of the 
Dirichlet realization of $ - h^2\Delta + V $ on $ \RR^n \setminus
\mathcal O $  in bounded subsets of $ \CC $ as $ h \to 0 $. When $ V \equiv 0 $ 
that corresponds to the high energy limit for obstacle problems and when 
$ \mathcal O = \emptyset $ to Schr\"odinger operators. 

\begin{figure}[htbp] 
\begin{center}
\hspace{0.5in}\includegraphics[width=4.5in]{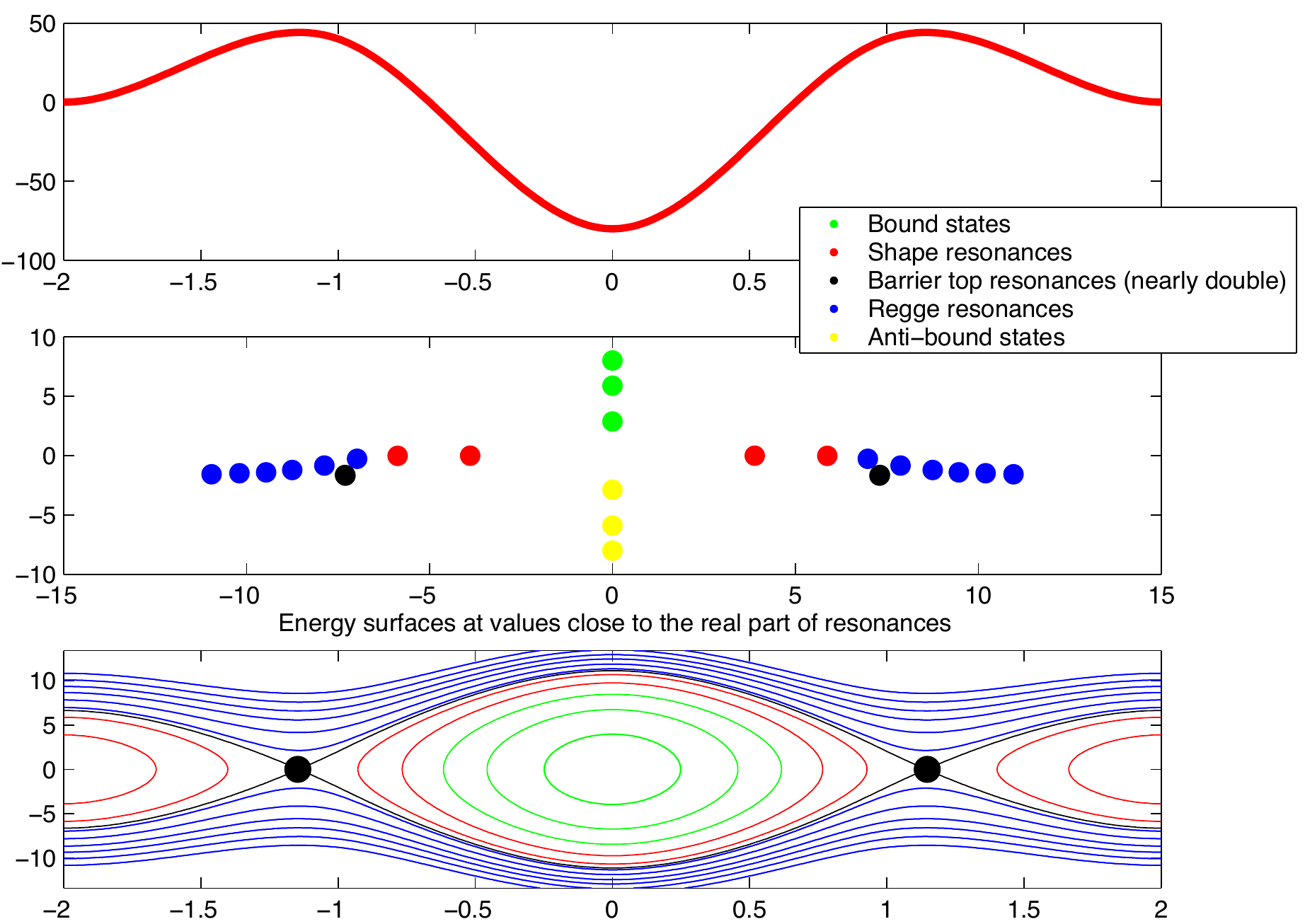}
\caption{\label{f:1}
  Resonances corresponding to different dynamical 
phenomena as poles of continuation of the resolvent 
$ \lambda \to ( - h^2 \partial_x^2 + V ( x) - \lambda^2 )^{-1} $ from 
$ \Im \lambda > 0 $ to $ \Im \lambda \leq 0 $ -- see \S \ref{merc3}.
The bound states are generated by negative level sets
of $ \xi^2 + V ( x )$ satisfying Bohr--Sommerfeld quantization conditions.
Bounded positive level sets of $ \xi^2 + V ( x ) $ can also satisfy 
the quantization conditions but they cannot produce bound state -- tunnelling
to the unbounded components of these level sets is responsible for 
resonances with exponentially small ($ \sim e^{ - S /h } $) imaginary parts/width.
The unstable trapped points corresponding to maxima of the potential
produce resonances which are at distance $ h $ of the real axis. }
\end{center}
\end{figure} 

In the case of semiclassical Schr\"odinger operators, 
the properties of the classical energy surface $ \xi^2 + V ( x ) = E $  
can be used to study 
resonances close to $ E \in \RR $.
Figure \ref{f:1} shows some of the principles in dimension
one. The classical energy surfaces are in this case integral curves of
the Hamilton flow, $ \dot x = 2 \xi $, $ \dot \xi = 
- V'( x)$ (Newton equations) and the properties of this flow
determine location of resonances when $ h $ is small. The
level sets which are totally trapped and satisfy quantization conditions
correspond to bound states. The level sets which have a trapped
component but also an unbounded component, give rise to resonances
with exponentially small decay rates, $ e^{-S/h} $. That make
sense since the corresponding quantum state take an exponentially
long time to tunnel through the barrier (see \cite{He-Sj0} for 
a general theory and \cite{Marr},\cite{GrMa} for some recent results and references). 
\begin{figure}[ht]
\begin{center}
\includegraphics[width=7.5cm]{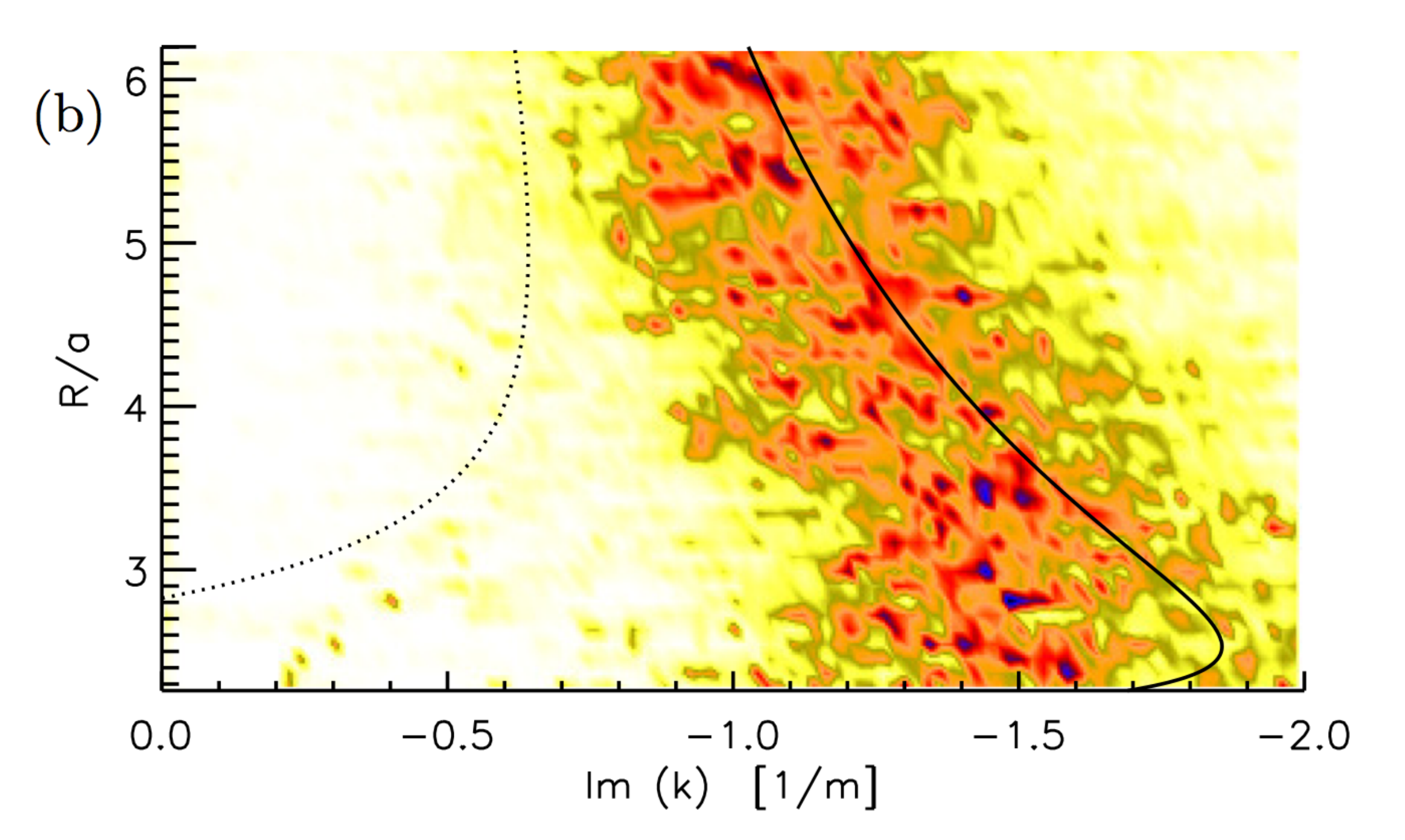} 
\includegraphics[width=7.5cm]{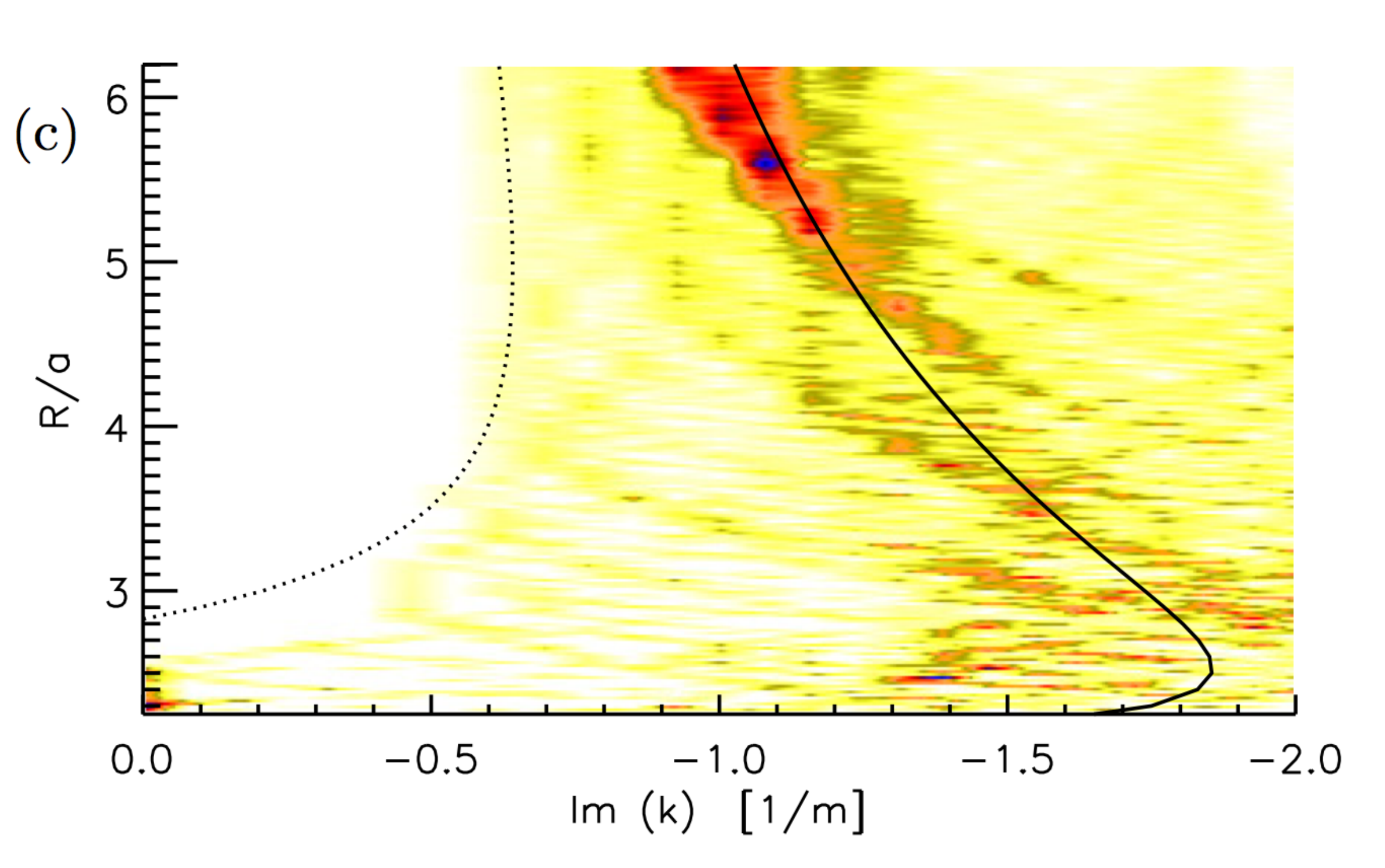}
\end{center}
\caption{Distribution of resonance width (decay rates) for a
three disc scatterer \cite{Bark} in a symmetry reduced microwave
model shown in Figure \ref{f:kuhl}: the plots show the density of
values of resonance width, $ \Im k $, for a range of
rest frequencies ($ 40 {\rm m}^{-1} \leq \Re k \leq 500 {\rm m}^{-1} $) against
the aspect ration of a three disc system (see Figure \ref{f:kuhl}). The 
experimental
data (left) and numerical  data (right) are compared. The dotted line
shows the topological pressure at $ \frac12 $ and the solid line $\frac12 $
of the classical decay rate as  functions of
the aspect ratio. The former predicts the resonance free region and
the latter, the concentration of decay rates.}
\label{f:class}
\end{figure}
Then we have energy 
levels which contain fixed points: that is an example of
{\em normally hyperbolic} trapping which produces resonances
at distance $ h$ from the real axis -- see \S\ref{resfree} 
for a general discussion and \cite{Ra} for precise analysis in one dimension 
and more references. Interestingly the same form of trapping 
occurs in the setting of rotating black holes, see \S\ref{expan} and
\cite{physrev}. Finally the unbounded 
trajectories would not produce resonances close (that is 
with imaginary parts tending to $ 0$ with $ h \to 0 $) to the real axis
if our potential were real analytic. But in the compactly supported, 
non smooth,
case the singularities at the boundary of the support produce
resonances at distance $ M h \log \frac 1 h $ where $ M$ depends on
the regularity of the potential at the boundary of the support.
(We call them Regge resonances as they were investigated in \cite{Re}
but should not be confused with closely related {\em Regge poles} \cite{Gr}.)
The anti-bound states are too deep in the lower half plane to have 
dynamical interpretation. They seem to be almost symmetric to bound
states and in fact when the potential is positive near the 
boundary of its support they are exponentially 
close to the reflection of bound states \cite{DG}.

\subsection{Other examples from physics and engineering}

We present here a few recent examples of scattering resonances appearing 
in physical systems. 

Figure \ref{f:hm} shows resonance peaks for 
a scanning tunneling microscope experiment
where a circular quantum corral of CO molecules 
is constructed -- see \cite{Mo} and references given 
there. The resonances are very close to eigenvalues 
of the Dirichlet Laplacian (rescaled by $ \hbar^2/m_{\rm{eff}} $
where $ m_{\rm{eff}} $ is the effective mass of the Bloch
electron). Mathematical results explaining existence of
resonances created by a barrier (here formed by a corral of CO  molecules)
can be found in \cite{NaStZw}.

\begin{figure}
\begin{center}
\includegraphics[width=2.15in]{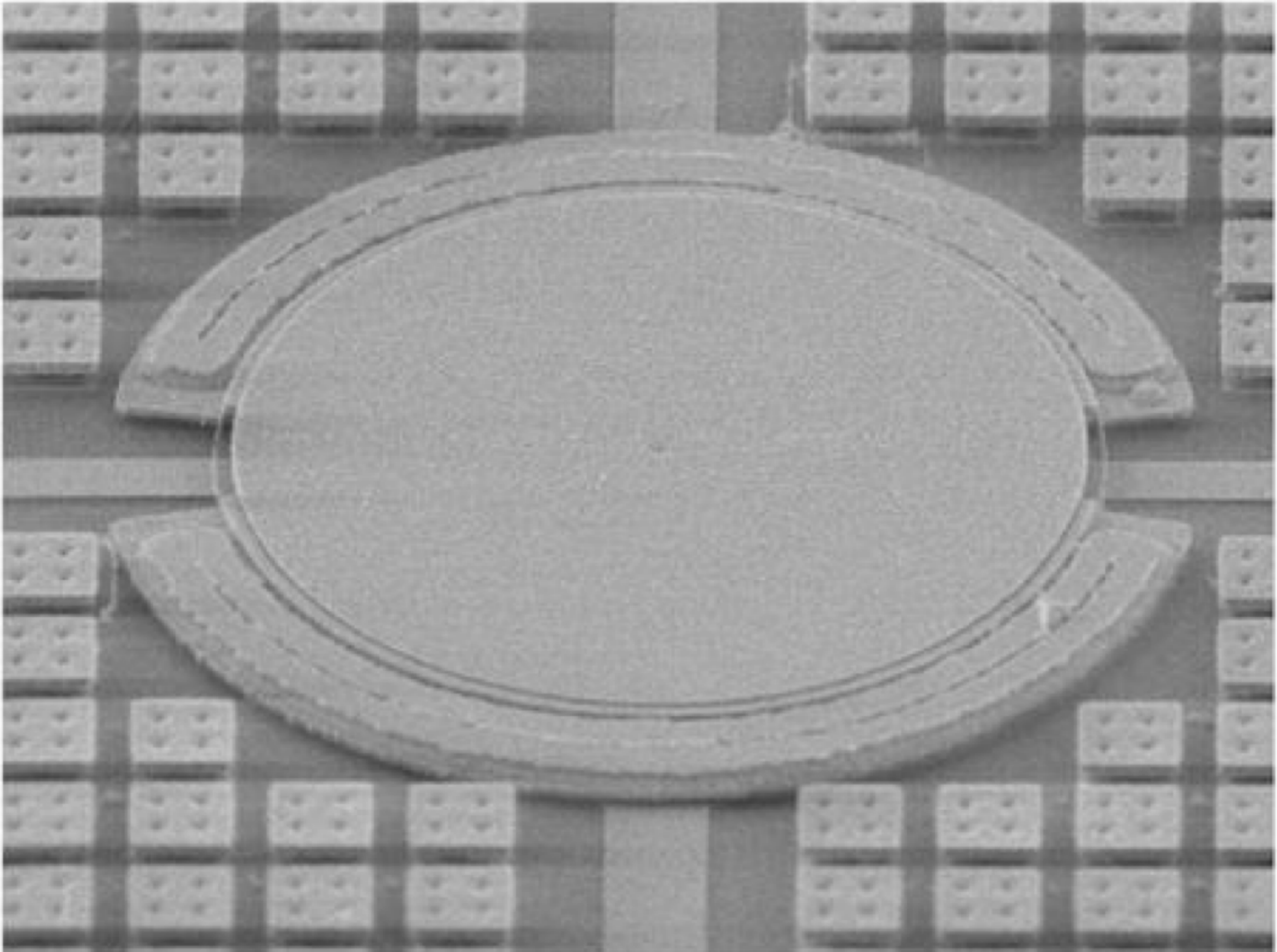}  \hspace{2cm} 
\includegraphics[width=1.22in]{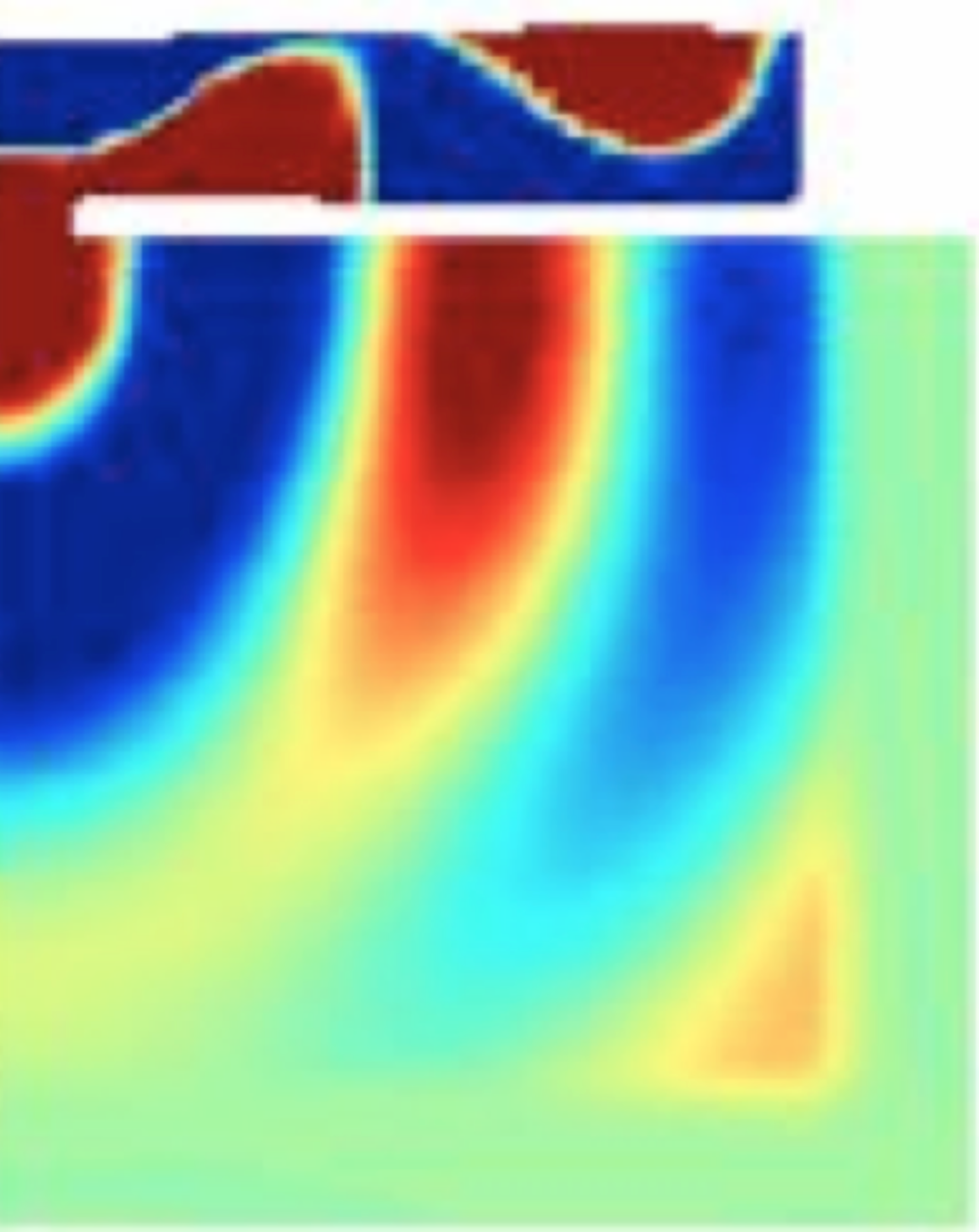}
\caption{\label{f:3}  A MEMS device on the left has resonances
investigated using the complex scaling/perfectly matched layer
methods \cite{BG} -- see \S\ref{rae}. A numerically constructed resonant state 
for the elasticity operator used to model the system -- 
see Definition~\ref{d:2} for a simpler case -- is
shown on the right.}
\end{center}
\end{figure}

Figure \ref{f:kuhl} shows an experimental set-up for 
microwave cavities used to study scattering resonances
for chaotic systems. Density of resonance was investigated in this 
setting in 
\cite{Potz} 
and that is related to semiclassical upper bounds in 
\S \ref{Weyl}. In \cite{Bark} dependence
of resonance free strips on dynamical quantities was confirmed
experimentally and \S \ref{resfree} contains related mathematical results
and references. The experimental and numerical findings in \cite{Bark}
are presented in Figure \ref{f:class}.

Figure \ref{f:3} shows a MEMS (the acronym for the {\em microelectromechanical systems})
resonator. The numerical calculations \cite{BG}
in that case are based on the complex scaling technique,
presented in a model case in \S \ref{rae},
adapted 
to the finite element methods. In that field it is known as the method of 
{\em perfectly matched layers} \cite{ber}. 

\begin{figure}
\includegraphics[width=5.75cm]{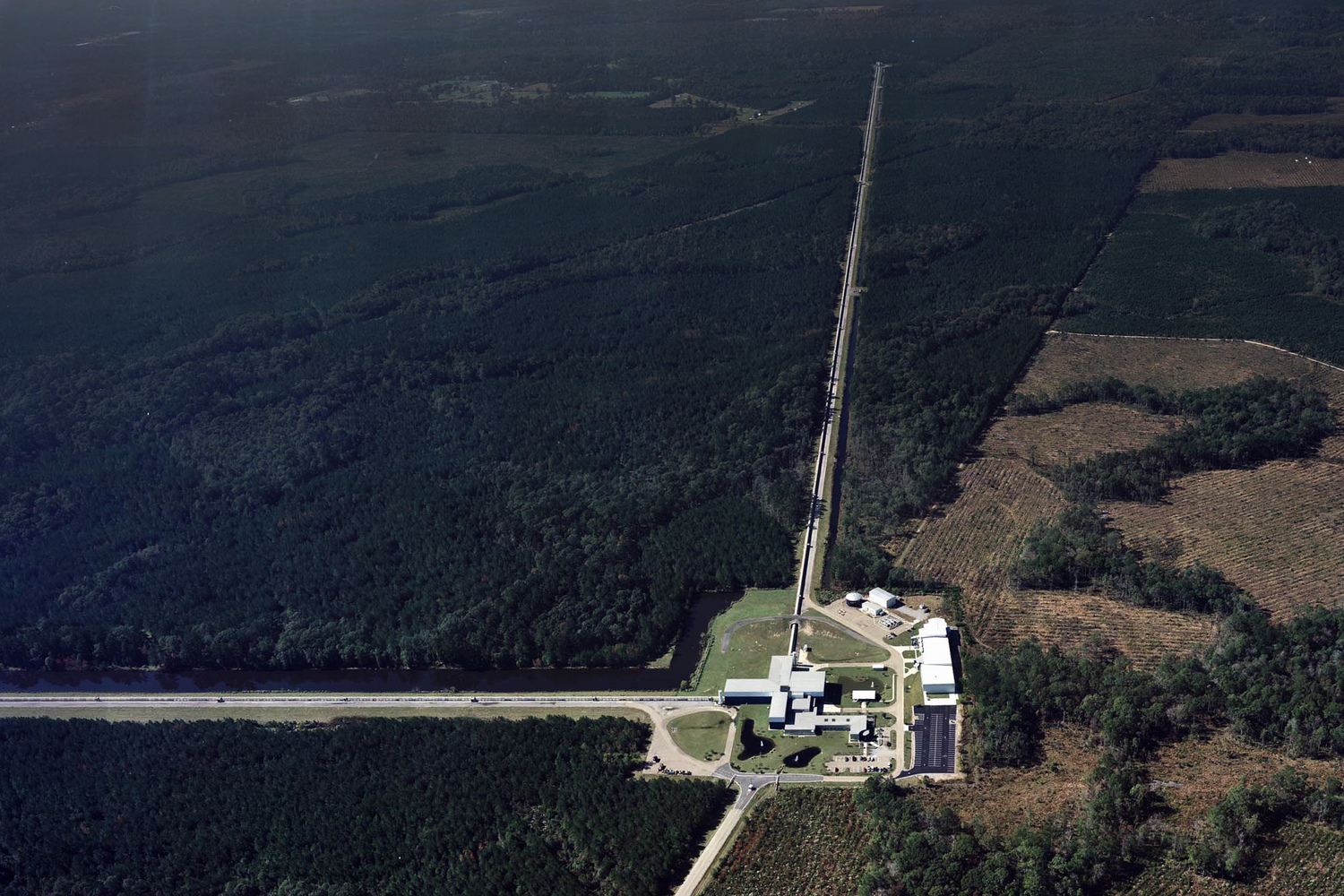}\quad
\includegraphics[width=7cm]{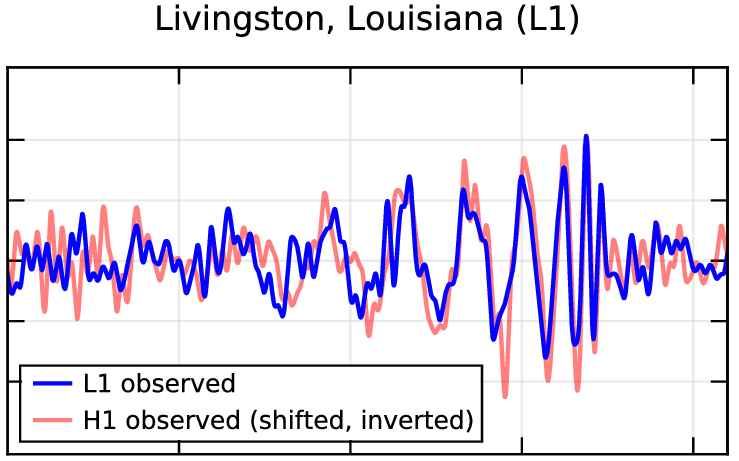}
\caption{Left: an aerial view of the LIGO laboratory in
Livingston, Louisiana, US. Right: the gravitational wave signal
observed on September 14, 2015 simultaneously by LIGO Livingston (blue)
and LIGO Hanford (red); see~\cite{LIGO-PRL}. Quasi-normal modes, as 
scattering resonances are called in the context of gravitational waves, 
are supposed to appear in the ``ringdown phase" but have not been 
observed yet. 
The picture is obtained from the LIGO Open Science
Center, \url{https://losc.ligo.org}, a service of LIGO Laboratory and the LIGO Scientific Collaboration. LIGO is funded by the U.S. National Science Foundation.}
\label{f:ligo}
\end{figure}

Figure~\ref{f:ligo} shows the profile of gravitational waves recently
detected by the Laser Interferometer Gravitational-Wave Observatory
(LIGO) and originating from a binary black hole merger.
Resonances for such waves are known by the name
of \emph{quasi-normal modes} in physics literature
and are the characteristic frequencies of the waves emitted
during the ringdown phase of the merger, when the resulting
single black hole settles down to its stationary state~--
see for instance~\cite{KokkotasSchmidt,zeeman,physrev}
and \S \ref{resrel}.

Our last example is a proposal of using Pollicott--Ruelle 
resonances in climate study by Chekroun et al \cite{cheka}. These are the resonances
appearing in expansions of correlations of chaotic flows and for the
new mathematical developments in their study see
\S \ref{dsPR}. The idea in \cite{cheka} is to use these resonances to encode
information about low-frequency variability  
of turbulent flows in the atmosphere and oceans. 
The spectral gap -- defined as the distance between the Ruelle--Pollicott resonances and the unitarity axis -- is used to see the roughness of parameter dependences
in different models -- see Figure \ref{f:PNAS}. The authors claim that 
``links between model sensitivity and the decay of correlation properties are not limited to this particular model and could hold much more generally".

\begin{figure}
\includegraphics[width=5.75cm]{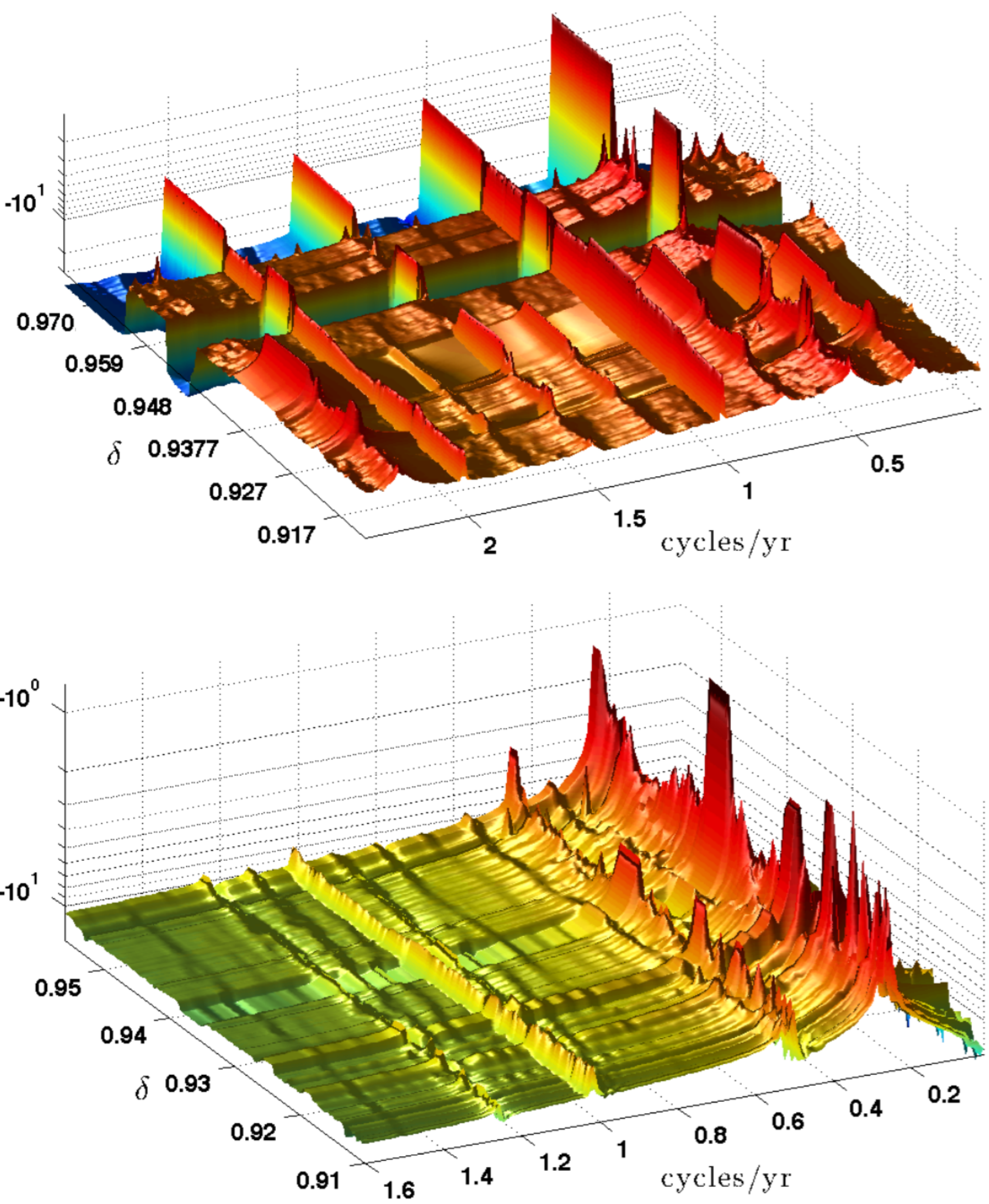}
\includegraphics[width=8cm]{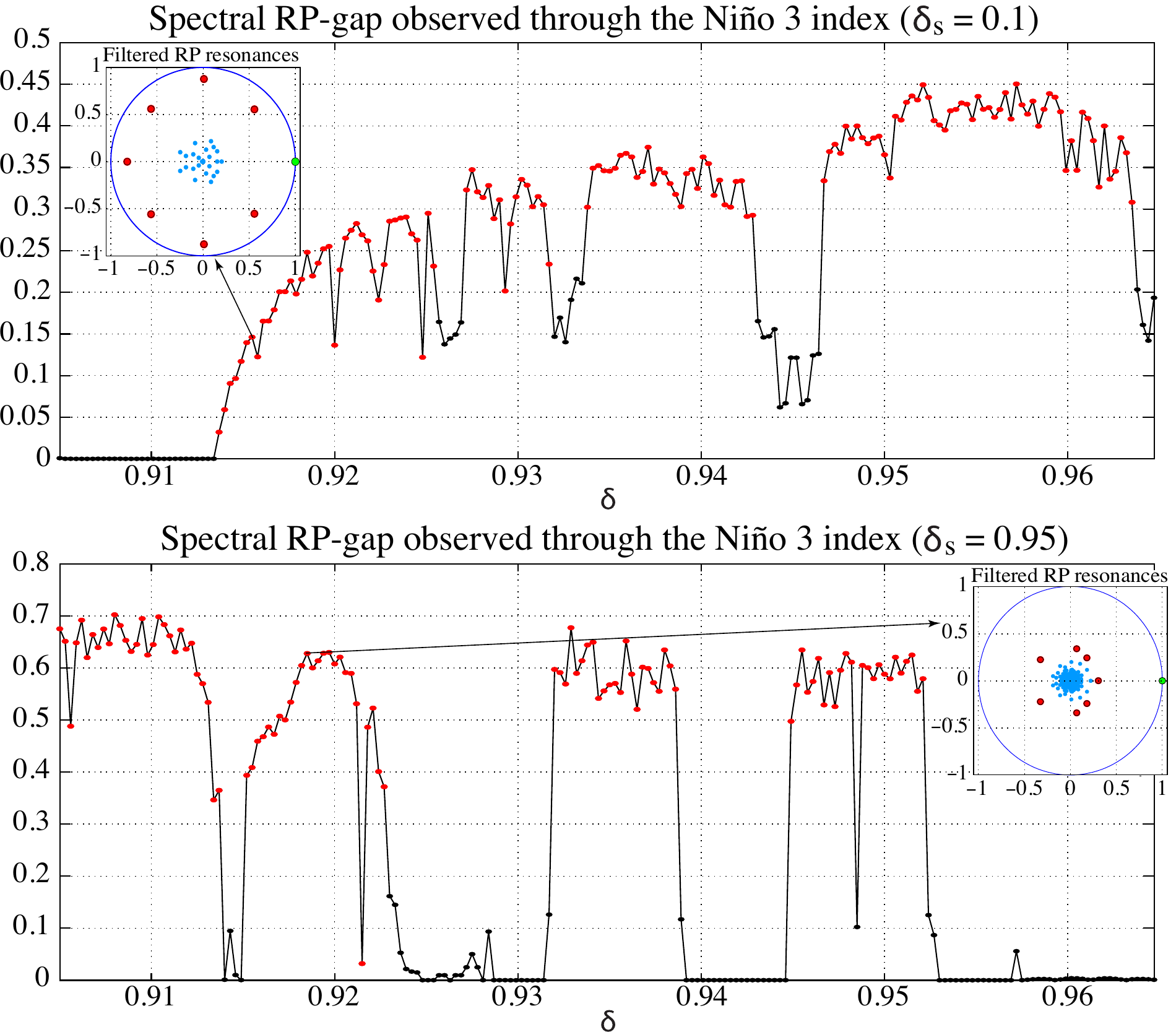}
\caption{Left: Parameter (denoted by $ \delta $) dependence of
power spectra of El Ni\~{n}o models considered in \cite{cheka}. Right: 
the dependence of the spectral gap on the same parameter $ \delta$; here
the resonance of maps obtained by using Markov partitions are presented
and they lie inside of the unit disk. The resonance for flows are obtained
by taking logarithms which is similar to the case of resonances of
quantum maps -- see Figure \ref{f:dylo}.}
\label{f:PNAS}
\end{figure}

\section{Potential scattering in three dimensions}
\label{pot3}

Operators of the form $ P_V := - \Delta + V $, where $ \Delta = 
\partial_{x_1}^2 + \partial_{x_2}^2 + \partial_{x_3}^2 $, and where
$ V$ is bounded and
compactly supported,  provide a
 setting in which one can easily present basic theory. Despite
the elementary set up, interesting open problems remain -- see \S \ref{othr1}. 
In this section we 
prove meromorphic continuation of the resolvent
of $ P_V$, $ R_V ( \lambda ) := ( P_V- \lambda^2 )^{-1} $, 
give a sharp bound on the number of resonances in discs, show
existence of resonance free regions and justify 
the expansion \eqref{eq:uoftx}. We also explain the
method of complex scaling in the simplest setting of one dimension. 
These are the themes which reappear
in \S\S \ref{srr},\ref{dsPR} when we discuss some recent advances.
However, in the setting of this section we can present them with complete proofs.
In \S \ref{resfree3} we also review some results in obstacle scattering
as they fit naturally in our narrative.

\subsection{The free resolvent}

Let $ P_0 := - \Delta$. This is a self-adjoint unbounded operator on $ L^2 
( \RR^3 ) $. We consider its resolvent, writing the spectral parameter as
 $ \lambda^2$. For $ \Im \lambda > 0 $, 
 \begin{equation}
 \label{eq:R00}  R_0 ( \lambda ) f ( x ) = \int_{ \RR^3 } R_0 ( \lambda, x , y ) 
 f ( y ) dy , \ \ f \in L^2 ( \RR^3 ) ,\end{equation}
 and we have an explicit formula for the {\em Schwartz kernel} of $ R_0 (\lambda ) $:
 \begin{equation}
 \label{eq:R01} R_0 ( \lambda , x, y ) = \frac{ e^{ i \lambda | x - y | } } { 4 \pi 
 | x - y | } .\end{equation} 
 (Since $R_0 ( \lambda, x, y ) $ depends only on $ |x-y| $ due to 
 translation and rotation invariance of $ - \Delta$, this can be seen
 using polar coordinates; see
 see \cite[Theorem 3.3]{res} for a different derivation.) 
 
 From this we see that \eqref{eq:R00} makes sense for {\em any} $ \lambda \in \CC $ 
 and  {\em any
 compactly supported} $ f $ (for instance $ f \in L^2_{\comp} ( \RR^3 ) $) and
 $ R_0 ( \lambda ) f $ is then a function {\em locally} in $ L^2 $,
 $ R_0 (\lambda ) f \in L^2_{\loc} ( \RR^3 ) $. In other words,
 \begin{equation}
 \label{eq:R0b}  R_0 ( \lambda ) : L^2 ( \RR^3 ) \to L^2 ( \RR^3 ) , \ \ \Im \lambda > 0 , \end{equation}
 has a {\em holomorphic} continuation,
\begin{equation}
\label{eq:R0bb} 
 R_0 ( \lambda ) : L^2_{\comp} ( \RR^3 ) \to L^2_{\loc} ( \RR^3 ) , \ \ 
 \lambda \in \CC . \end{equation}
This, and the connection to the wave equation, explains why we took 
$ \lambda $ as our parameter. The mapping property \eqref{eq:R0bb} remains
valid in all odd dimensions while in even dimensions continuation 
has to be made to the infinitely sheeted logarithmic plane.
In quantum scattering (as opposed to scattering of classical waves) 
the usual spectral parameter $ z = \lambda^2 $
is more natural in which case the continuation is from 
$ \CC \setminus [ 0 , \infty ) $ through the cut $ [ 0 , \infty ) $.

We have the following formula which will be useful later:
\begin{equation}
\label{eq:R0E}
\begin{gathered}
 R_0 ( \lambda ) - R_0 ( - \lambda ) = 
{
{\frac{i\lambda} 2} } 
\, \mathbb E( \bar \lambda )^* \mathbb E ( \lambda )  , \ \ \ \lambda \in \CC ,
 \\
\mathbb E ( \lambda ) : L^2 ( \RR^3 ) \to L^2_{\loc} (\SP^2 ),
\ \ \ 
\mathbb E ( \lambda ) g ( \omega ) := {
{\frac{1}{2\pi} }}
\int_{ \RR^3 } e^{- i \lambda \langle x , \omega \rangle } g (x) d x , 
\end{gathered}
\end{equation}
where $ L^2 ( \SP^2 ) $ is defined using the standard measure on the sphere. This follows 
from \eqref{eq:R01} and the elementary identity 
$\int_{ \SP^2} e^{ - i \langle y, \omega \rangle} d \omega = 
-2 i \pi ({ e^{ i |y| } - e^{ - i |y|}})/|y| $,
see \cite[Lemma 3.2]{res}.
In view of the spectral decomposition based on the Fourier transform,
\begin{equation}
\label{eq:Stone0} f = {
{\frac{1}{2\pi} }} \int_0^\infty  \mathbb E ( \lambda )^*\mathbb E ( \lambda ) f  \lambda^2 d \lambda , \ \
- \Delta  f = {
{\frac{1}{2\pi} }} \int_0^\infty  \lambda^2 \mathbb E ( \lambda )^* \mathbb E ( \lambda ) f  \lambda^2 d \lambda ,  
\end{equation}
$ f \in \CIc ( \RR^3) $, 
formula \eqref{eq:R0E} is a special case of the {\em Stone formula} 
relating resolvents and spectral measures \cite[(B.1.12)]{res}.

The spectral decomposition \eqref{eq:Stone0} and \eqref{eq:R0E} 
is one way to see the relation of $ R_0 ( \lambda ) $ with the wave equation:
for $ f \in \CIc ( \RR^3 ) $
\begin{equation}
\label{eq:R0wave}
\begin{split}
\frac{ \sin t \sqrt{- \Delta} }{ \sqrt{ - \Delta}} f  & = {
{\frac{1}{2\pi} }}
\int_0^\infty \sin t \lambda  \, \mathbb E ( \lambda )^* \mathbb E ( \lambda ) f
\, \lambda \, d \lambda \\
& = {
{\frac{1}{\pi i } }} \int_0^\infty 
\sin t \lambda \, ( R_0 ( \lambda ) - R_0 ( - \lambda ) ) f\, d \lambda \\
&= {
{\frac{1}{\pi i } }}  \int^\infty_0 \frac{e^{it\lambda}-e^{-it\lambda}}{2i}
(R_0(\lambda)-R_0(- \lambda)) f\, d \lambda\\
&={
{\frac{1}{ 2 \pi }}} \int_{\RR }
  e^{- it\lambda} R_0(\lambda) f  \, d \lambda -
  {
  {\frac{1}{ 2 \pi }}} \int_{\RR }
  e^{- it\lambda}
R_0 ( - \lambda)) f \, d \lambda .
\end{split}
\end{equation}
For $ t > 0 $ we can deform $ \RR $ in the second integral to the contour $ \RR - i \gamma $, $ \gamma > 0 $. By letting $ \gamma \to + \infty $ we then see that 
\begin{equation}
 \label{eq:R0wave1}
 \frac{ \sin t \sqrt{- \Delta} }{ \sqrt{ - \Delta}} f   = 
 {
 {\frac{1}{2\pi} }} \int_\RR e^{ - i \lambda t }  R_0 ( \lambda ) f d \lambda , \ \ t > 0 , \ \ \
 f \in \CIc ( \RR^3 ) .
\end{equation}
(The integrals above converge only in the distributional sense and hence one
should first ``integrate" both sides against $ \varphi ( t) \in \CIc ( ( 0 , \infty)) $ and then deform the contour.)
When combined with \eqref{eq:R01} we obtain the Schwartz kernel of  
$ \sin( t \sqrt{- \Delta} )/\sqrt{-\Delta} $, 
\begin{equation}
\label{eq:sintD}
 \frac{ \sin t \sqrt{- \Delta} }{ \sqrt{ - \Delta}}  ( x, y ) = 
 \frac{ \delta ( t - |x-y|) }{ 4 \pi t } , \ \ t > 0 . 
 \end{equation}
which gives Kirchhoff's formula
 \[  \frac{ \sin t \sqrt{- \Delta} }{ \sqrt{ - \Delta}} f  ( x ) 
 = \frac{1 }{4 \pi t } \int_{ \partial B ( x , t ) } f ( y ) d \sigma ( y ) , \ \
 t > 0 . \]
The important feature of \eqref{eq:sintD} is the support property
of the Schwartz kernel: it vanishes outside of the light cone
$ |x - y | = t $. That is the {\em sharp Huyghens principle} valid in 
all odd dimensions greater than one and violated in even dimensions. 

We conclude this section with

\begin{thm}
\label{t:1}
Let $ \Delta = \sum_{j=1}^3 \partial_{x_j}^2  $ be the Laplacian in $ \RR^3 $.
Then the resolvent 
$$ R_0 ( \lambda ) := ( - \Delta - \lambda^2 )^{-1} :
L^2 ( \RR^3 ) \to L^2 ( \RR^3 ) , \ \  \Im \lambda > 0 ,$$
extends
holomorphically to $ \CC $ as an operator 
\[   R_0 ( \lambda ) : L^2_{\comp} ( \RR^3 )  \to L^2_{\loc} ( \RR^3 ) . \]

Moreover, for any $ \rho \in \CIc ( B ( 0 , R ); [ 0, 1] ) $ we have
\begin{equation}
\label{eq:t1} 
\| \rho R_0 ( \lambda ) \rho \|_{ L^2 \to H^j } \leq 
C  e^{ 2 R (\Im \lambda )_- } ( 1 + |\lambda |)^{j-1} , \ \ j=0,1, 
\end{equation}
where $ C $ depends only on $ \rho $.
\end{thm}
\begin{proof}
To see \eqref{eq:t1}  we apply 
the (distributional) Fourier inversion formula to \eqref{eq:R0wave1}. That gives
\[  R_0 ( \lambda ) =  \int_0^\infty e^{  i \lambda t } U(t)
dt , \ \ \ U(t) := \frac{\sin t \sqrt { - \Delta } }{\sqrt{-\Delta}} .\]
Since the Schwartz kernel \eqref{eq:sintD} is supported 
on $ |x-y| = t $ and $ \rho $ is supported in $ |x| < R $ we can 
also have
\begin{equation}
\label{eq:rRr} \rho R_0 ( \lambda ) \rho = \int_0^{2R} e^{  i \lambda t } 
\rho U(t) \rho dt . \end{equation}
To use this to prove \eqref{eq:t1}, we note that
\[  \begin{split} \| U ( t ) \|_{ L^2 \to H^1 } & \simeq \| U( t ) \|_{ L^2                         
  \rightarrow L^2 } +
\| \sqrt { - \Delta } U ( t ) \|_{L^2 \to L^2 } 
= 
\sup_{ \lambda \in \RR  } \frac{ |\sin t \lambda |}{|\lambda |} + 
\sup_{ \lambda \in \RR } |\sin t \lambda | \\
& = 1 + |t| .
\end{split}\] 
This and \eqref{eq:rRr} give the bound  \eqref{eq:t1}  for $ j = 1$. For 
$ j = 0 $
we write (noting that $ U ( 0 ) = 0 $)
\[ \begin{split} \lambda \rho R_0 ( \lambda ) \rho
= & {
{\frac 1 i }} \int_0^{2R}  \partial_t (e^{ i \lambda t} )
\rho U ( t) \rho dt
= -  {
{\frac 1 i }}  \int_0^{2R}  e^{ i \lambda t}
\rho \partial_t U ( t) \rho dt 
\,.\end{split} \]
Since $ \partial_t U (t) =\cos t \sqrt{ - \Delta } $ is uniformly bounded on 
$ L^2 $, the bound \eqref{eq:t1} for $ j = 0 $ follows. \end{proof}

\subsection{Meromorphic continuation and definition of resonances}
\label{merc3}

We now consider $ P_V : = - \Delta + V $ where $ V \in L^\infty_{\rm{comp}} ( 
\RR^3 ) $ is allowed to be complex valued. Since
\begin{equation}
\label{eq:PlaP0la}   P_V - \lambda^2 = ( I + V R_0 ( \lambda ) ) ( P_0 - \lambda^2 ) 
\end{equation}
the study of $ ( P_V - \lambda^2 )^{-1} $ reduces to the study of
$ I + V R_0 ( \lambda ) $. For $ \Im \lambda > 0 $
\[  \widehat{R_0 ( \lambda ) f } ( \xi ) = \frac{ \hat f ( \xi ) }{
|\xi|^2 - \lambda^2 },  \]
where $ f \mapsto \hat f $ is the Fourier transform.
This implies that  
\begin{equation}
\label{eq:R0la}  \| R_0 ( \lambda ) \|_{ L^2 \to L^2 } = \frac{1}{ d ( \lambda^2, \RR_+) }
\leq \frac1{ |\lambda| \Im \lambda } , \ \ \ \Im \lambda > 0 .
\end{equation}
It follows that 
\[  \Im \lambda \gg 1 \ \Longrightarrow \ \| V R_0 ( \lambda ) \|_{ L^2 
\to L^2 } < 1 , \]
and hence 
\begin{equation}
\label{eq:RVR0} 
\begin{split} 
  R_V ( \lambda ) & := ( P_V - \lambda^2)^{-1} \\
  & = R_0 ( \lambda ) 
( I + V R_0 ( \lambda ) )^{-1}  \end{split} \end{equation}
is a holomorphic family of operators from $ L^2 $ to $ L^2 $ when $
\Im \lambda \gg 1 $.

We would like to continue $ R_V ( \lambda ) $ as a {\em meromorphic family} of
operators $ L^2_{\comp} ( \RR^3 ) \to L^2_{\loc} ( \RR^3 ) $. 
Because of \eqref{eq:RVR0} and Theorem \ref{t:1} this follows from the 
following statement 
\begin{equation}
\label{eq:VR0}  
\begin{gathered}  ( I + V R_0 ( \lambda ) )^{-1} : L^2_{\comp} ( \RR^3 ) 
\to L^2_{\comp} ( \RR^3 )  \\ \text{ is a meromorphic family of operators for
$ \lambda \in \CC $.} 
\end{gathered}
\end{equation}
Strictly speaking, we should really say that this family is a continuation 
of the holomorphic family of inverses defined for $ \Im \lambda \gg 1 $. 
By a {\em meromorphic family of
operators} $ \lambda \mapsto A ( \lambda ) $ we mean a family which 
is holomorphic outside a discrete subset of $ \CC $ and at any point $ \lambda_0$
in that subset we have
\begin{equation}
\label{eq:Ala}  A ( \lambda ) = \sum_{j=1}^J \frac{ A_j }{ ( \lambda - \lambda_0)^j} + A_0 ( \lambda ) , \end{equation}
where $ A_j $ are operators of {\em finite rank} and $ \lambda \mapsto 
A_0 ( \lambda ) $ is holomorphic near $ \lambda_0 $.

The proof of \eqref{eq:VR0} relies on {\em analytic Fredholm theory}:
suppose that $ K ( \lambda ) : L^2 \to L^2 $ is a holomorphic family of
compact operators for $ \lambda \in \CC $ and that 
$ ( I + K ( \lambda_0 ) )^{-1} :L^2 \to L^2 $ exists at some $ \lambda_0 \in 
\CC $. Then 
\begin{equation}
\label{eq:AFT}
\lambda \mapsto ( I + K ( \lambda ) )^{-1} \ \text{ is a meromorphic family 
operators for $ \lambda \in \CC $.}
\end{equation}
(See \cite[Theorem C.5]{res} for a general statement and a proof.) 

It is tempting to apply \eqref{eq:AFT} to obtain \eqref{eq:VR0} but
one immediately notices that $ V R_0 ( \lambda ) $ is not defined on 
on $ L^2 $ once $ \Im \lambda \leq 0 $. To remedy this we introduce
$ \rho \in \CIc ( \RR^3 ) $ equal to $ 1 $ on $ \supp V $ and write
\[ I + V R_0 ( \lambda ) = ( I + V R_0 ( \lambda ) ( 1 - \rho )) ( I + 
V R_0 ( \lambda ) \rho ) .\]
We consider this first for $ \Im \lambda \gg 1 $ in which 
case $ R_0 ( \lambda ) $ is bounded on $ L^2 $.
Then 
\[ ( I + V R_0 ( \lambda ) ( 1 - \rho ) )^{-1} = 
I - V R_0 ( \lambda ) ( 1 - \rho ) \]
is a bounded operator  $ L^2 \to L^2 $ when $ \Im \lambda > 0 $ and
it continues holomorphically to $ \CC $ as an operator $ L^2_{\comp}  
\to L^2_{\comp} $. 

For $ \Im \lambda \gg 1 $, $ V R_0 ( \lambda ) \rho $
has small norm and we conclude that
\begin{equation}
\label{eq:RVR01}  ( I + V R_0 ( \lambda ) )^{-1} = ( I + V R_0 ( \lambda ) \rho )^{-1} 
( I - V R_0 ( \lambda ) ( 1 - \rho ) ) , 
\end{equation}
is a bounded operator on $ L^2 $ when $ \Im \lambda \gg 1 $ and 
we need to continue it as an operator $ L^2_{\comp} \to L^2_{\comp} $

Hence to obtain \eqref{eq:VR0} we need to show that
\begin{equation}
\label{eq:VRr} ( I + V R_0 ( \lambda ) \rho )^{-1} : L^2_{\comp} ( \RR^3 ) \to 
L^2_{\comp} ( \RR^3 )  \ \text{ is meromorphic in $ \lambda \in \CC $.}
\end{equation}
But now we can apply \eqref{eq:AFT} with $ K ( \lambda ) := V R_0 ( \lambda ) \rho$.
In fact, \eqref{eq:t1} shows that 
if $\supp \rho \subset B ( 0 , R  ) $ then 
$ \rho R_0 ( \lambda ) \rho : L^2 \to H^1 
( B ( 0 , R ) ) $. The Rellich--Kondrachov Theorem \cite[Theorem B.3]{res}, shows that $ \rho R_0 ( \lambda ) \rho$
is a compact operator. But then so is $ V R_0 ( \lambda ) \rho = V 
( \rho R_0 ( \lambda ) \rho ) $. It follows from \eqref{eq:AFT} that
\begin{equation}
\label{eq:VRrr} 
 ( I + V R_0 ( \lambda ) \rho )^{-1} : L^2( \RR^3)  \to L^2 ( \RR^3 ) \ \  
\text{ is a meromorphic family for $ \lambda \in \CC $. } \end{equation}

To obtain \eqref{eq:VR0} we need to show that compactness of the support
is preserved by $ ( I + V R_0 ( \lambda ) \rho )^{-1} $. To see this
let $ \chi \in \CIc ( \RR^n ) $ be equal to $ 1 $ on $ \supp \rho $.
We claim that
\begin{equation}
\label{eq:phVRr}  ( I + V R_0 ( \lambda ) \rho )^{-1} \chi = \chi ( I + V R_0 ( \lambda ) 
\rho )^{-1} .\end{equation}
In fact, this is valid for $ \Im \lambda \gg 1 $ by expanding the inverse
in Neumann series and then it follows by analytic continuation. From 
\eqref{eq:VRrr} and \eqref{eq:phVRr} we obtain \eqref{eq:VRr} which then 
gives \eqref{eq:VR0}. 

Combining \eqref{eq:RVR0}  and \eqref{eq:RVR01} with \eqref{eq:VR0} we proved 

\begin{thm}
\label{t:2} 
The operator $ R_V ( \lambda ) $ defined in \eqref{eq:RVR0} 
continues to a meromorphic family
\[   R_V ( \lambda ) : L^2_{\comp} ( \RR^3 ) \longrightarrow L^2_{\loc} ( 
\RR^3 ) , \ \ \ \lambda \in \CC . \]
\end{thm}

This gives a mathematical definition of scattering resonances:

\begin{defi}
\label{d:1}
 Suppose that $ V \in L^\infty_{\comp} ( \RR^3 ) $ and 
that $ R_V ( \lambda ) $ is the scattering resolvent of Theorem \ref{t:2}.
The poles of $ \lambda \mapsto R_V ( \lambda ) $ are called 
{\em scattering resonances} of $ V$. If $ \lambda_0 $ is a scattering
resonance then, in the notation of \eqref{eq:Ala} with $ A( \lambda ) = 
R_0 ( \lambda ) $, the multiplicity of $ \lambda_0 $ is defined as
\[ m ( \lambda_0 ) = \dim \, {\rm{span}}\, \left\{ A_1 ( L^2_{\comp} ) , 
\cdots ,  A_J ( L^2_{\comp} )\right\} . \]
\end{defi}

There are other equivalent definitions of multiplicity which use the
special structure of $ A_j $'s in the case of a resolvent. For instance
for $ \lambda_0 \neq 0 $,
\begin{equation}
\begin{split} 
\label{eq:mult}  m ( \lambda_0 ) & = \rank A_1 \\
& = \rank \oint_{\lambda_0 } 
R_V ( \lambda ) 2 \lambda d \lambda, 
\end{split} 
\end{equation}
where the integral is over a small circle containing $ \lambda_0 $
but no other pole of $ R_V $. The situation at $ 0 $ is more complicated
-- see \cite[\S 3.3]{res} which can serve as an introduction to 
general theory of Jensen--Kato \cite{JK} and Jensen--Nenciu \cite{JN}.

To see the validity of \eqref{eq:mult} we use the 
equation $ ( P - \lambda^2 ) R_V ( \lambda ) f = f $, $ f \in L^2_{\comp} $,
to see that near a pole $ \lambda_0 \neq 0 $,
\begin{equation}
\label{eq:Jordan}
\begin{gathered}
R_V( \lambda ) = 
\sum_{ j=1}^J \frac{ ( P_V - \lambda_0^2)^{j-1} \Pi_{\lambda_0 } }{ ( \lambda^2 - 
\lambda_0^2)^j } + A( \lambda, \lambda_0 ) , \\ 
\Pi_{\lambda_0} = {\textstyle{\frac{1}{ 2 \pi i }}} \oint_{\lambda_0 }
R_V ( \lambda ) 2 \lambda d \lambda , 
\end{gathered}
\end{equation}
where $ \lambda \mapsto 
A ( \lambda, \lambda_0 ) $ is holomorphic near $ \lambda_0 $, 
$ ( P_V - \lambda_0^2)^{J} \Pi_{ \lambda_0 } = 0 $ and 
$ P_V - \lambda_0^2 : \Pi_{\lambda_0 } ( L^2_{\comp} ) \to 
\Pi_{\lambda_0 } ( L^2_{\comp} )  $ -- see \cite[\S 3.2,\S 4.2]{res}.
This leads to 

\begin{defi}
\label{d:2}
A function $ u \in \Pi_{\lambda_0} ( L^2_{\comp} ( \RR^3 ) ) 
\subset L^2_{\loc} ( \RR^3 ) $ is called a {\em generalized resonant state}.
If $ ( P_V - \lambda_0^2 ) u = 0 $ then $ u $ is called a {\em 
resonant state}. This is equivalent to 
$ u \in  ( P_V -\lambda_0^2 )^{J-1} \Pi_{\lambda_0} ( L^2_{\comp} ( \RR^3 ) ) $.
\end{defi}

Formula \eqref{eq:RVR0} and the expansion \eqref{eq:Jordan} give 
the following characterization (see \cite[Theorem 3.7, Theorem 4.9]{res})
\begin{equation}
\label{eq:outg}
\begin{gathered}
\text{ u is a resonant state for a resonance $ \lambda_0 $} \\
\Updownarrow \\
 \exists \, f \in L^2_{\comp} ( \RR^3 ) \ \
u = R_0 ( \lambda_0 ) f , \ \ ( P_V - \lambda_0^2 ) u = 0 . 
\end{gathered}
\end{equation}
The condition $ u = R_0 ( \lambda ) f $ is called the {\em outgoing} condition.
In \S \ref{rae} we will see a more complicated but much more natural 
characterization of outgoing states. 


When $ V $ is real valued then $ P_V $ is a self-adjoint operator 
(with the domain given by $ H^2 ( \RR^3 ) $). The poles of $ R_V ( \lambda ) $
in $ \Im \lambda > 0 $ correspond to negative eigenvalues of $ P_V $
and in \eqref{eq:Jordan} we have $ J = 1$. This of course is consistent
with \eqref{eq:outg} since for $ \Im \lambda_0 > 0 $, $ R_0 ( \lambda_0 ) f
\in L^2 ( \RR^3 ) $.

\subsection{Upper bound on the number of resonances}
\label{upper}

Once we have defined resonances it is natural to 
estimate the number of resonances. Since they are now 
in the complex plane the basic counting function is given by
\begin{equation}
\label{eq:NV}
N_V ( r ) := \sum \{ m ( \lambda ) : |\lambda |\leq r \} ,
\end{equation}
where $ m ( \lambda ) $ is defined in \eqref{eq:mult}.

We will prove the following optimal bound
\begin{thm}
\label{t:3}
Suppose that $ V \in L^\infty_{\rm{comp}} ( \RR^3 ) $. Then
\begin{equation}
\label{eq:NVb}
N_V ( r ) \leq C r^3 .
\end{equation}
\end{thm}

The bound, $ N_V ( r ) \leq C r^n $, valid
in all odd dimensions $n$, was proved in \cite{Z2} using methods
developed by Melrose \cite{Mel2},\cite{Mel3} who proved
$ N_V ( r ) \leq C r^{n+1} $. The proof presented here 
uses a substantial simplification of the argument due to Vodev 
\cite{Vo} who proved corresponding upper bounds in even 
dimension \cite{Vo2},\cite{Vo3}, following an earlier contribution by
Intissar \cite{in}. When $ n = 1 $ we have 
an asymptotic formula,
\begin{equation}
\label{eq:1da} N_V ( r ) = {\textstyle{\frac 2 \pi }} {\rm{diam}}\, ( \supp V ) \, r + o ( r) ,
\end{equation}
stated by Regge \cite{Re} and proved in \cite{Z1}.  

The exponent $ 3 $ in \eqref{eq:NVb} is optimal
as shown by the case of radial potentials.
When $ V ( x ) = v ( |x| ) ( R - |x|)_+^0 $,  where 
$ v $ is a $ C^2 $ even function, and $ v ( R ) > 0 $,
then, see \cite{Z11},
\begin{equation}
\label{eq:radas} N_V ( r )  =  C_R r^3 +
 o ( r^3 )   \,. \end{equation}
The constant $ C_R $ and its appearance in a refinement of \eqref{eq:NVb} is 
explained by Stefanov \cite{Stef}.

\medskip

\noindent
{\bf Interpretation.}  In the case of $ - \Delta + V $ on 
a bounded domain or a closed manifold, for instance on $ M = \TT^n := 
\RR^n / \ZZ^n  $ or $ M = \SP^{n} $ (the $n$-sphere), 
the spectrum is discrete and for $ V \in L^\infty ( \TT^n ; \RR)  $ 
we have the asymptotic Weyl law for the number of eigenvalues:
\begin{equation}
\label{eq:Weyll}
\begin{gathered}  | \{ \lambda: \lambda^2 \in \Spec( - \Delta_{M } + V ) \,, \  |\lambda|  \leq r \} | = c_n {\rm{vol}} \, ( M ) r^n  +  o ( r^n)  \,, \\
c_n = {  2\, {\rm{vol}} ( B_{\RR^n} ( 0 , 1 ) ) }/{ ( 2 \pi)^n  }
\,, 
\end{gathered}
\end{equation}
where the eigenvalues are included according to their multiplicities. 

In the case of $ - \Delta + V $ on $ \RR^n $, the discrete spectrum 
is replaced by the discrete set of resonances. Hence the bound 
\eqref{eq:NVb} is an analogue of the Weyl law. Except in dimension 
one -- see \eqref{eq:1da} -- the issue of
asymptotics or even optimal lower bounds remains unclear 
at the time of writing of this survey. 
We will discuss lower bounds and existence of resonances in \S \ref{othr1}.

\medskip

The strategy for the proof of the upper bound \eqref{eq:NVb} is to include
poles of $ R_V ( \lambda) $ among zeros of an entire function. 
The number of poles is then estimated by estimating the
growth of that function and using {\em Jensen's inequality} -- see
the proof of Theorem \ref{t:3} below. 
In particular, if $ h ( \lambda ) $ is our entire function then 
\eqref{eq:NVb} follows from a bound
\begin{equation}
\label{eq:hbd}
| h ( \lambda ) | \leq A e^{ A | \lambda |^3 } , 
\end{equation}
for some constant $ A$.

Hence our goal is to find a suitable $ h $ and to prove \eqref{eq:hbd}.
Before doing it we recall some basic facts about trace class operators
and Fredholm determinants. We refer to \cite[Appendix B]{res} for 
proofs and pointers to the literature. 

Suppose that $ A : H_1 \to H_2 $ is a bounded operator between two
Hilbert spaces. Then 
$ ( A^* A )^{\frac12} : H_1 \to H_1 $ is a bounded non-negative self-adjoint operator
and we denote its spectrum by $ \{ s_j ( A ) \}_{ j=0}^\infty $. 
The numbers $ s_j ( A )$ are called {\em singular values} of $ A $. 
We will need the following well known  inequalities \cite[Proposition B.15]{res}
\begin{equation}
\label{eq:KyFan}
s_{ j+k} ( A + B ) \leq s_j ( A ) + s_k ( B ) , \ \ 
s_{j+k} ( A B ) \leq s_j ( A ) s_k ( B ) . 
\end{equation}
If $ H_1 = H_2 = L^2 ( \RR^3 ) $ and in addition $ A : L^2 ( \RR^3 ) \to H^s ( \RR^3 ) $, $ s \geq 0 $, and the support of the
Schwartz kernel of $ A $ is contained in $ B ( 0 , R ) \times B ( 0 , R ) $,
then, with a constant depending on $ R $,  
\begin{equation}
\label{eq:sSob}
s_j ( A ) \leq C_R \| A \|_{ L^2 ( \RR^3 ) \to H^s ( \RR^3 ) } \, 
(1+j)^{-s/3} . 
\end{equation}
(See \cite[Example after Proposition B.16]{res} for a proof
valid in all dimensions $n$, that is with $ 3 $ replaced by $ n $ in 
\eqref{eq:sSob}.)

The operator $ A$ is said to be of {\em trace class}, $ 
A \in \mathcal L_1 ( H_1 , H_2 ) $, if 
\[  \| A \|_{\mathcal L_1 } := \sum_{j=0}^\infty s_j ( A ) < \infty .\]
For $ A \in \mathcal L_1 ( H_1, H_1 ) $ can we define  {\em Fredholm
determinant}, $ \det ( I + A ) $, 
\begin{equation}
\label{eq:detest}
|  \det ( I + A ) | \leq e^{ \| A \|_{ \mathcal L_1}} ,
\end{equation}
see \cite[\S B.5]{res}. A more precise statement is provided
by a {\em Weyl inequality} (see \cite[Proposition B.24]{res})
\begin{equation}
\label{eq:Weyli}
|  \det ( I + A ) | \leq \prod_{j=0}^\infty ( 1 + s_j ( A ) ) .
\end{equation}
Lidskii's theorem \cite[Proposition B.27]{res} states that
\[ \det ( I + A ) = \prod_{j=0}^\infty ( 1 + \lambda_j ( A ) ), 
\]
where $ \{ \lambda_j ( A ) \}_{j=0}^\infty $ are the eigenvalues
of $ A $. In particular, $ \det ( I + A ) = 0 $ if and only 
if $  I + A $ is not invertible. When $ \lambda \mapsto A ( \lambda ) $
is a holomorphic family of operators a more precise statement
can be made by taking account multiplicities -- see \cite[\S C.4]{res}.

We will now take $ A = - ( V R_0 ( \lambda ) \rho )^4 $, $ H_1 = L^2 ( \RR^3) $.
From \eqref{eq:t1}, \eqref{eq:KyFan} and \eqref{eq:sSob} we see
that
\begin{equation}
\label{eq:sj1}    s_j ( V R_0 ( \lambda ) \rho ) \leq
\| V \| s_j ( \rho R_0 ( \lambda ) \rho ) \leq 
C e^{ C( \Im \lambda )_- } (1+j)^{-1/3} . 
\end{equation}
Another application of \eqref{eq:KyFan} shows that
\begin{equation}
\label{eq:sj0} s_j ( ( V R_0 ( \lambda ) \rho )^4 ) \leq C e^{ C ( \Im \lambda)_- } 
(1+j)^{-4/3} , 
\end{equation}
 and hence $ ( V R_0 ( \lambda ) \rho)^4 \in \mathcal L_1$.
Thus we can define
\begin{equation}
\label{eq:defh}
h( \lambda ) := \det ( I - ( V R_0 ( \lambda ) \rho)^4 ) .
\end{equation}
We have
\[ I -  ( V R_0 ( \lambda ) \rho )^4 =  
  ( I - V R_0 ( \lambda ) \rho + ( V R_0 ( \lambda ) \rho )^2 - 
  ( V R_0 ( \lambda ) \rho )^{3} ) ( I +   V R_0 (
\lambda ) \rho  )\,. \]
Hence in view of \eqref{eq:RVR0} and \eqref{eq:RVR01} it is 
easy to believe that the poles of $ R_V ( \lambda ) $ are included, 
with multiplicities, among the zeros of $ h ( \lambda ) $ -- see
\cite[Theorem 3.23]{res}:
\begin{equation}
\label{eq:multh}
m_V ( \lambda_0 ) \leq \frac{1}{2 \pi i} \oint_{\lambda_0 } 
\frac{ h'( \lambda ) }{h( \lambda ) } d \lambda . 
\end{equation}

\medskip

\noindent
{\bf Remark.} It is not difficult to see that we could, just as
 in \cite{Z2}, 
take $ \det ( I - ( V R_0 ( \lambda ) \rho)^2 ) $. But the power
$ 4 $ makes some estimates more straightforward. We could
also have taken a modified determinant of $ I + V R_0 ( \lambda ) \rho
$ \cite[Theorem 3.23, \S B.7]{res} and that would give an entire
function whose zeros are exactly the resonances of $ V $.
That works in dimension three but in higher dimensions
the modified determinants grow faster than \eqref{eq:hbd} -- see
\cite{Z2}.

\medskip

\begin{proof}[Proof of Theorem \ref{t:3}]
Jensen's formula relates the number of zeros of an entire function 
$ h $ to its growth: if 
$ n ( r )$ is the number of zeros of $ h $ in $ |z| \leq r $
\[  \int_0^r \frac {n(t)}{t} dt = \frac{1}{ 2 \pi} 
\int_0^{2\pi} \log | h ( r e^{ i \theta } ) | d \theta - \log | h ( 0) | , \]
where we assumed that $ h ( 0 ) \neq 0 $ (an easy modification takes care of
the general case). By changing $ r $ to $ 2 r $, it follows that
\[  n ( r) \leq \frac{1}{ \log 2 } \sup_{ |\lambda | \leq2 r }  \log | h ( \lambda ) |
- \log | h ( 0) |. \]
It follows from this and \eqref{eq:multh} that
\[ | h ( \lambda ) | \leq A e^{ A |\lambda|^3 } 
\ \Longrightarrow \ n ( r ) \leq C r^3  \ \Longrightarrow \ N_V ( r ) \leq C r^3. \]

To prove the bound on $ h $ we first note that \eqref{eq:sj0} and \eqref{eq:detest} show that
$ | h ( \lambda ) | \leq C $ for $ \Im \lambda \geq 0 $. To get a good
estimate for $ \Im \lambda \leq 0 $ we use \eqref{eq:KyFan} and \eqref{eq:Weyli}:
\begin{equation}
\label{eq:HWey1} |  h ( \lambda ) | \leq  \prod_{k=0}^\infty \left( 1 + s_k( ( V R_0 (
\lambda ) \rho )^{4} )\right)  
\leq  \prod_{k=0}^\infty \left( 1 + \| V \|_{L^\infty}^4 s_{[k/4]}( \rho R_0 (
\lambda ) \rho )^{4} \right) 
. \end{equation}
Hence we need to estimate $ s_j ( \rho R_0 ( \lambda ) \rho ) $ for
$ \rho \in \CIc ( \RR^n ) $. 

For  $ \Im \lambda \geq 0 $ we already have the estimate
\eqref{eq:sj1}.  To obtain estimates for $ \Im \lambda \leq 0 $ we use
\eqref{eq:R0E} to write 
\[    \rho ( R_0 ( \lambda ) - R_0 ( -\lambda ) ) \rho = 
{\textstyle{\frac{i\lambda} 2} } 
\rho \mathbb E ( \bar \lambda )^* \mathbb E(  \lambda )\rho \]
Hence, using \eqref{eq:KyFan} and \eqref{eq:sj1} again, 
\begin{equation}
\label{eq:sj2} 
\begin{split} 
s_j ( \rho R_0 ( \lambda ) \rho ) & \leq C | \lambda| 
\| E ( \lambda ) \rho \|  s_{[j/2]} ( \mathbb E ( \lambda ) \rho ) + s_{ [ j/2] } ( \rho R_0 (
- \lambda ) \rho ) \\
& \leq C \exp ( C |\lambda | ) s_{[j/2] } ( \mathbb E ( \lambda ) \rho )  + C
(1+j)^{ - 1/3  } \,. 
\end{split}
\end{equation}

To estimate $ s_j ( \mathbb E ( \lambda ) \rho ) $ we use the Laplacian on
the sphere, $ - \Delta_{\SP^{2} }$ and the estimate:
\[ \begin{split}  \|    ( - \Delta_{\SP^{2} } + 1 )^{\ell}  \mathbb E ( \lambda ) \rho
\|  & \leq C \sup_{ \omega \in \SP^{n-1}, |x|\leq R } 
 \left|   ( - \Delta_{\omega } + 1 )^{\ell}  e^{ i \lambda \langle x ,
    \omega \rangle } \right|
    \\
    & \leq C^\ell  \exp ( C |\lambda |) ( 2 \ell)! , 
    \end{split}
\]
where we estimated the supremum using the Cauchy estimates and assumed
that $ \supp \rho \subset B ( 0 , R ) $. (We follow the usual 
practice of changing the value of $ C $ from line to line.)

From the explicit formula for the eigenvalues of 
$ - \Delta_{ \SP^2 } $ (given by $ k( k + 1 ) $ with multiplicity
$ k + 1 $), or from the general counting law \eqref{eq:Weyll}, we see
that 
\[ s_j ( ( - \Delta_{\SP^2 } + 1 )^{-\ell} ) \leq C^\ell (1+j)^{ - \ell} .\]
The combination of the last two estimates gives
\begin{equation}
\label{eq:spsj}
\begin{split} 
s_j ( \mathbb E ( \lambda )\rho ) & \leq s_j ( ( - \Delta_{\SP^2 } + 1
)^{-\ell} )  \|   ( - \Delta_{\SP^{2} } + 1 )^{\ell}  E_\rho (
\lambda ) \| \\
& \leq  C^\ell (1+j)^{- \ell  }  \exp ( C |\lambda |) ( 2 \ell
)! \,. 
\end{split}
\end{equation}
We now optimize the estimate \eqref{eq:spsj} in $ \ell $ which gives
\begin{equation}
\label{eq:sjE}   s_j ( \mathbb E ( \lambda ) \rho ) \leq C \exp \left( C | \lambda
 | - j^{\frac12}   / C  \right) \,. 
\end{equation}

 Going back to \eqref{eq:sj2} we obtain
\begin{equation}
\label{eq:skR}
\begin{split} s_j (  V R_0 ( \lambda ) \rho  ) & \leq 
C \exp \left( C  | \lambda
 | - j^{\frac12 } / C  \right)  + C (1+j)^{ - \frac13 }
\\
& 
\leq  \left\{ 
\begin{array}{ll}   e^{ C |\lambda | + C } \,, & j \leq C |\lambda
  |^{2} \\
 C (1+j)^{-\frac{1} 3 } \,, & j \geq C | \lambda |^{2} \,. 
\end{array} \right.
\end{split}
\end{equation}

Returning to \eqref{eq:HWey1} we use \eqref{eq:skR} as follows
\begin{equation}
\label{eq:hlaf}  \begin{split}
| h ( \lambda ) | & \leq \prod_{ k \leq C |\lambda|^2 } 
e^{C | \lambda| + C } \left( \exp \sum_{ k \geq C  | \lambda |^2 } 
 k^{ - 4/3} \right) 
\leq C e^{ C   |\lambda|^3}  \,,
\end{split} \end{equation}
which completes the proof.
\end{proof}

The key to fighting exponential growth when $ \Im \lambda < 0 $
is {\em analyticity} near infinity which is used implicitely 
in \eqref{eq:spsj}-\eqref{eq:hlaf}. That is a recurrent theme
in many approaches to the study of resonances -- see \ref{rae}.

Resonance counting has moved on significantly since these early 
results. Some of the recent advances will be reviewed in \S \ref{Weyl},
see also \cite{counting} for an account of other early works. 
The most significant breakthrough was Sj\"ostrand's  discovery \cite{SjDuke} of 
geometric upper bounds on the number of resonances in which the
exponent is no longer the dimension as in \eqref{eq:NVb} and \eqref{eq:Weyll}
but depends on the dynamical properties of the classical system.

\subsection{Resonance free regions}
\label{resfree3}

The imaginary parts of
resonances are interpreted as decay rates of the 
corresponding resonant states -- see \S \ref{expan} for a justification of that
in the context of the wave equation. If there exists a resonance closest
to the real axis then (assuming we can justify expansions like \eqref{eq:uoftx} -- see Theorems \ref{t:5},\ref{t:rel}) its imaginary part determines
the principal rate of decay of waves -- see the end of this 
section for some comments on that. If waves are localized in frequency
then imaginary parts of resonances with real parts near that frequency
should determine the decay of those waves. 

But to assure the possibility of having a {\em principal} resonance, that 
is a resonance closest to the real axis, 
we need to know that there exists a strip without any resonances. 
Hence it is of 
interest to study 
high frequency {\em resonance free regions} which are typically of the form 
\begin{equation}
\label{eq:resf1}
\Im \lambda > - F (\Re \lambda  )  , \ \ \Re \lambda > C  , \ \ F ( x ) = 
\left\{ \begin{array}{ll} 
(a) & e^{ - \alpha x} , \ \ \alpha > 0 \\
(b) & M \\
(c) & M \log x    \\
(d) &  \gamma x^\beta  , \ \ \beta \in \RR , \gamma > 0 \end{array} 
\right.
\end{equation}
where $ M> 0 $ may be fixed or arbitrarily large. In the setting
of compactly supported potentials we see cases of $ (c) $ and $ (d) $: fixed $ M$ for bounded potentials \cite{LP}, arbitrary $ M $ for 
smooth potentials \cite{Vai} and $ \beta = 1/a $ for 
potentials in the $a$--Gevrey class \cite{good}. For applications it
is important that the statement about resonance free region is quantitative
which means that it comes with some resolvent bounds -- see \eqref{eq:t4}
below and for applications \S \ref{expan}.

Here we present the simplest case:
\begin{thm}
\label{t:4}
Suppose that $  V \in L^\infty( \RR^3 ) $, $\supp V \subset B ( 0 , R_0 ) $,
and that 
$ R_V ( \lambda ) $ is the scattering resolvent of Theorem \ref{t:2}.

Then we can find $ A > 0 $ such that for any $ \chi \in \CIc ( B ( 0 , R_1 )  ) $,
$ R_1 \geq R_0 $,
there exists $ C $ for which 
\begin{equation}
\label{eq:t4}
\| \chi R_V ( \lambda ) \chi \|_{ L^2 \to H^j } \leq C 
( 1 + |\lambda |)^{j-1} e^{ 2 R_1 ( \Im \lambda)_- }   , \ \ j = 0 , 1 , 
\end{equation}
when  $ \Im \lambda > - {\textstyle\frac{1}{2R_0}} \log | \Re \lambda | + A $.
\end{thm}
\begin{proof}
Without loss of generality we can take $ \chi $ which is equal to $ 1 $ 
on the support of $ \rho $ used in \eqref{eq:VRr}. Then using
Theorem \ref{t:1}, \eqref{eq:RVR0}, \eqref{eq:RVR01}  and \eqref{eq:phVRr} we see
that
\[ 
\begin{split} \| \chi R_V ( \lambda ) \chi \|_{ L^2 \to H^j}  & 
= \| \chi R_0 ( \lambda ) \chi ( I + V R_0 ( \lambda ) \rho)^{-1} 
(I - V \chi R_0 ( \lambda ) \chi ( 1 - \rho ) ) \|_{ L^2 \to H^j } \\
& \leq \| \chi R_0( \lambda ) \chi \|_{ L^2 \to H^j } 
( 1 + \| V \|_\infty \| \chi R_0( \lambda ) \chi \|) \| ( I + V R_0 ( \lambda ) \rho )^{-1} \| \\
& \leq C_1  ( 1 + |\lambda |)^{ j-1} 
e^{ 4 R_1 ( \Im \lambda )_- } \| ( I + V R_0 ( \lambda ) \rho )^{-1} \|
, \end{split}
\] 
where the norms are operator norms $ L^2 \to L^2 $ unless indicated otherwise.
Assuming, as we may that, $ \supp \rho \subset B ( 0, R_0 ) $ (the only requirement
on $ \rho \in \CIc ( \RR^3 ) $ is that it is equal to $ 1 $ on the support
of $ V $), it is sufficient to prove 
that there exist $ A $ such that
\begin{equation}
\label{eq:t41}
\| ( I + V R_0 ( \lambda ) \rho )^{-1} \|_{ L^2 \to L^2 } < 2 \ \ \
\text{when $ \ \ \Im \lambda > - {\textstyle\frac{1}{2R_0}} \log | \Re \lambda | + A $.} \end{equation}
However, \eqref{eq:t1} applied with $ j = 0 $ gives
\[  \| V R_0 ( \lambda ) \rho \| \leq C_2 ( 1 + |\lambda |)^{-1} 
e^{ 2 R_0 ( \Im \lambda)_- } < \textstyle\frac12 , \]
once $ A $ is large enough (depending on $ C_2  $ which depends on 
$ \rho $ and $ V $).
This completes the proof as we can then invert $I + V R_0 ( \lambda ) \rho $
as a Neumann series.
\end{proof}

In the setting of obstacle scattering,
(see Figure \ref{f:helm}  for an example and \cite[\S\S 4.1,4.2]{res} for
a general framework which includes this case) the study of resonance free
regions and of the closely related decay of waves was initiated by 
Lax--Phillips \cite{LP}, Morawetz \cite{Mo} and Ralston \cite{rals}, see also 
Morawetz--Ralston--Strauss \cite{Mrs}. 
These early results stressed the importance of 
the {\em billiard flow}\footnote{That is the flow defined by 
propagation along straight lines with reflection at the boundary;
trapping refers to existence of trajectories which never escape to 
infinity.} and
provided an impetus for the study of
propagation of singularities for boundary value problems
by Andersson, Ivrii, 
Melrose, Lebeau, Sj\"ostrand and Taylor
-- see H\"ormander \cite[Chapter 24]{H3} and references given there.
A direct method, alternative to Lax--Phillips theory, for relating
propagation of singularities to resonance free strips was
developed by Vainberg \cite{Vai},\cite[\S 4.6]{res}. All 
the possibilities in \eqref{eq:resf1} can occur in the 
case of scattering by obstacles with {\em smooth} boundaries:

\begin{center}
\begin{tabular} {|| l | l ||}
\hline
\ & \ \\
\ \ \ \ \ \ \ \ Geometry & \ \ \ \ \ \ \ \ \ \ \ \
\ \ \ \ \ \ \ \ \ \ Resonance free region  \\
\ & \ \\
\hline
Arbitrary obstacle & (a) \cite{bu0}; optimal for obstacles with 
elliptic closed orbits \\ 
\ &  
\cite{StVo1},\cite{TaZ},\cite{StefDuke} \\
\hline
Two convex obstacles & (b) with a pseudo-lattice of resonances below the strip \\
\ & 
\cite{ik1},\cite{Ge1}, hence optimal \\
\hline
Several convex obstacles & (b) with $ M$ determined by a 
{\em topological pressure} of   \\
\ & the chaotic system \cite{GaRi},\cite{ik2}; an improved gap \cite{PeSt} \\
\hline
Smooth non-trapping  & (c) with arbitrarily large $ M$ \cite{Mel0},\cite{Vai},
using \\
obstacles & propagation of singularities  \cite[Chapter 24]{H3}; \\
\hline
Analytic non-trapping & (d) with $ \beta = \frac13 $ \cite{baleron},\cite{popo}
based on propagation of \\
obstacles & Gevrey-3 singularities \cite{le}; optimal for convex obstacles \\
\hline
Smooth strictly convex & (d) with $ \beta = \frac13 $ and $ \gamma $ determined
by the curvature  \\
obstacles & \cite{hale},\cite{SZ4}\cite{Long1}; Weyl asymptotics for resonances\\
\ &  in cubic bands $ \sim r^{n-1} $ \cite{SZ5} \\
\hline
\end{tabular}
\end{center}

There has been recent progress in the study of  
non-smooth obstacles, and in particular in 
taking account of the effect of diffraction at conic points
on the distribution of resonances -- see the lecture by Wunsch
\cite{wu2} for a survey and references. 

\begin{SCfigure}
\centering                                                         
\includegraphics[width=8cm]{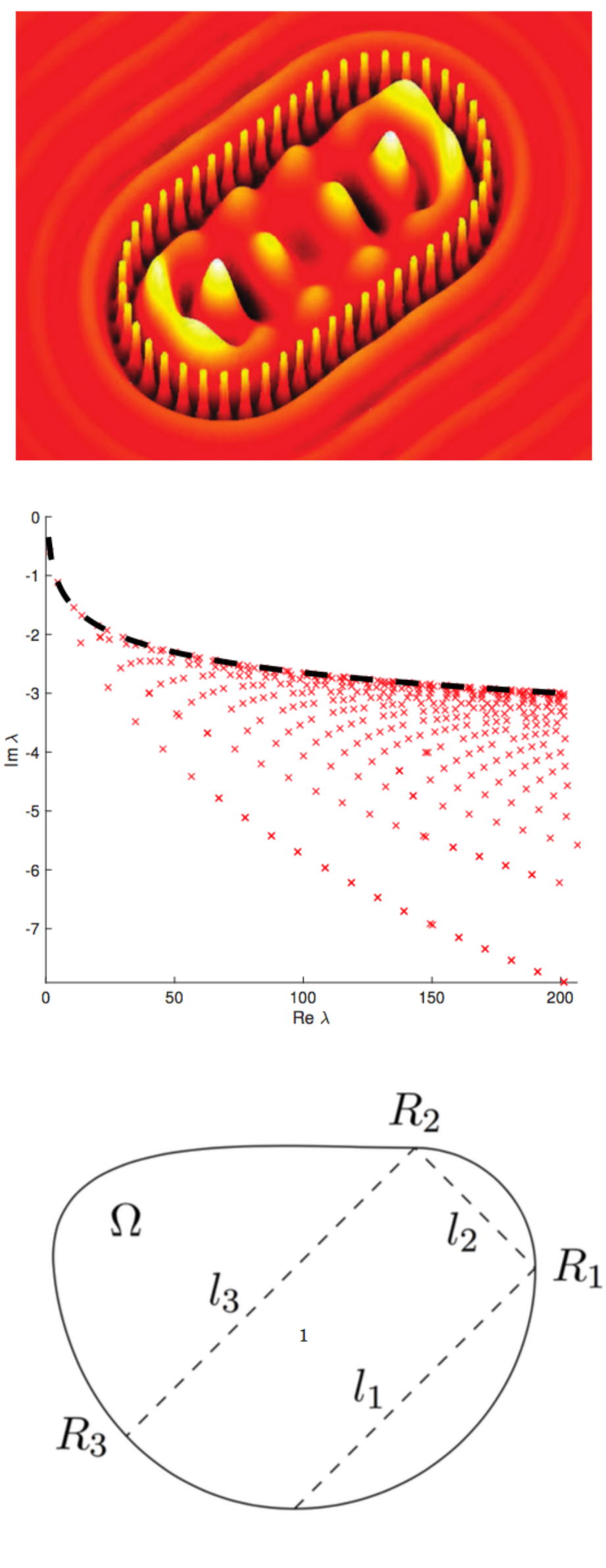}
\caption{An illustration of the results of \cite{gal1},\cite{gal2}. At the top an image of a quantum corral formed from individual iron atoms taken using a scanning tunneling microscope; this
motivated the quantum Sabine law of \cite{haze} where other numerical experiments
can be found. In the middle resonances computed for a simple model of a circular quantum corral: scattering by a delta function of the form $ \delta_{\SP^1}$ where $\SP^1\subset\mathbb{R}^2$ is the unit circle.
 It shows resonances and the bound on the resonance free region predicted by the Sabine law (the dashed black line). 
To understand the Sabine law, consider, $u_0$,  a function localized in space and 
momentum to $ ( x_0 , \xi_0 )\in \Omega \times \SP^{1} $ up to the scale allowed by the uncertainty principle (a ``wave packet"). A wave started with initial data $u_0$ propagates along the billiard flow starting from $(x_0,\xi_0)$.
 At each intersection of the billiard flow with the boundary, the amplitude inside of $\Omega$ will decay by a factor, $R$, depending on the point and direction of intersection. Suppose that the billiard flow from $(x_0,\xi_0)$ intersects the boundary at $(x_n,\xi_n)\in\partial\Omega\times \SP^{1}$, $n>0$. Let $l_n=|x_{n+1}-x_n|$ be the distance between two consecutive intersections with the boundary, 
 as shown at the bottom figure. Then the amplitude of the wave decays by a factor $\prod_{i=1}^nR_i$ in time $\sum_{i=1}^nl_i$ where $R_i=R(x_i,\xi_i)$. The energy scales as amplitude squared and since the imaginary part of a resonance gives the exponential decay rate of $L^2$ norm, this leads us to the heuristic that resonances should occur at 
$\Im \lambda=\left.\overline{\log |R|^2}\right/(2\bar{l})$
where the mean is defined by $\bar{f}=\frac{1}{N}\sum_{i=1}^Nf_i$. We call this heuristic rule the {\em quantum Sabine law}.}
\label{f:galk}
\end{SCfigure}

Another rich set of recent results concerns scattering by ``thin"
barriers modeled by delta function potentials  (possibly energy 
dependent) supported on hypersurfaces in $ \RR^n $. 
For example, refinements
of the Melrose--Taylor parametrix techniques 
give, in some cases, the optimal resonance free region defined
by $ F ( x ) = x^{ -\beta } $, $ \beta > 0 $ and 
 a mathematical explanation of a {\em Sabine law} for quantum corrals
\cite{haze}.  See
the works of Galkowski \cite{gal1},\cite{gal2},\cite{gal3} and
Smith--Galkowski \cite{galsm} where references to earlier
literature on resonances for transmission problems can also be found,
and Figure~\ref{f:galk} for an illustration.

Finally we make some comments on low energy bounds on resonance
widths. Universal lower bounds bounds on $ |\Im \lambda |$ for 
{\em star shaped obstacles} were obtained by Morawetz \cite{Mo},\cite{Mo2} 
for obstacle
problems 
and by Fern\'andez--Lavine \cite{fela} for more general operators.
In the case of {\em star shaped obstacles} the best bound is due to 
Ralston \cite{RalstonDecay} who using Lax--Phillips theory \cite{LP} showed that 
in all odd dimensions
\[  \Im \lambda < - 2\, {\rm{diam}}( \mathcal O)^{-1}  .\]
Remarkably this bound is optimal for the sphere in dimensions three and five.
The obvious difficulty in obtaining such bounds is the lack of a variational principle due to the
non-self-adjoint nature of the problem (see \S \ref{rae}). 

The opposite extreme of star shaped obstacles are {\em Helmholtz resonators} shown in Figure \ref{f:helm}. 
Here a cavity is connected to the exterior through a neck of 
width $ \epsilon $. A recent breakthrough  by 
 Duyckaerts--Grigis--Martinez \cite{dgm} provided precise
asymptotics (as $ \epsilon \to 0 $) of the resonance width associated
to the lowest mode of the cavity. 

\begin{figure}
\includegraphics[scale=0.37]{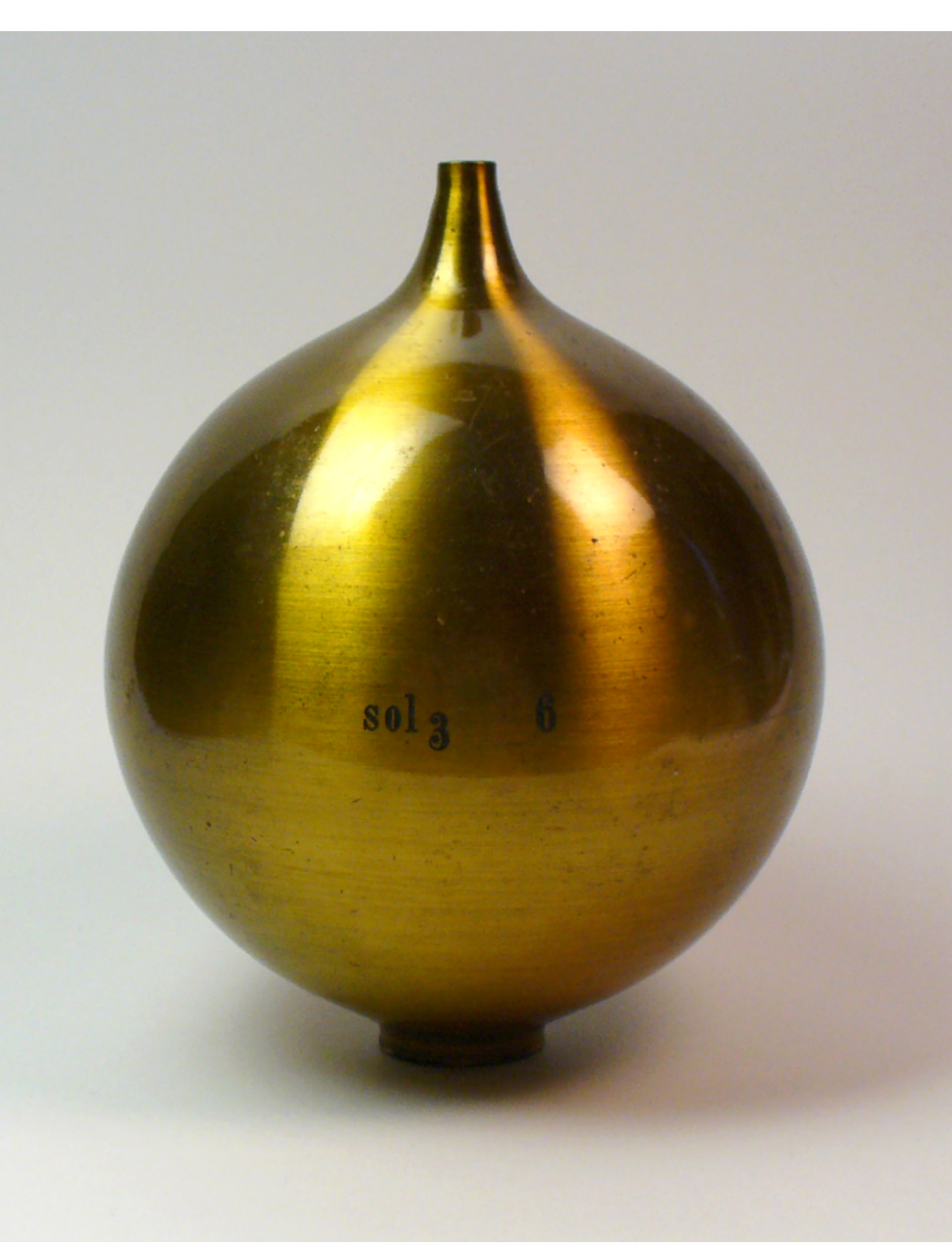} \hspace{0.2in} \includegraphics[scale=0.4]{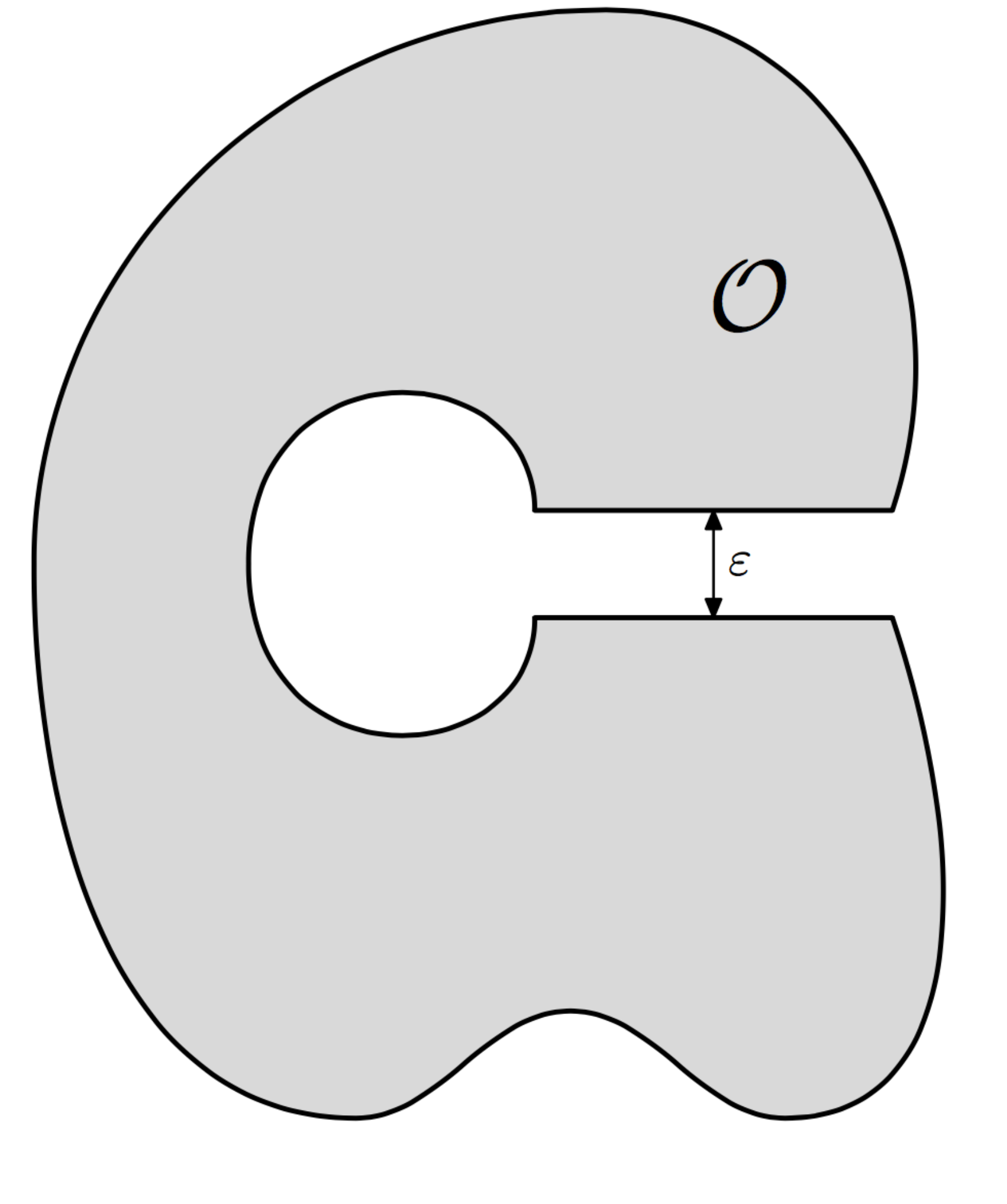}
\caption{\label{f:helm}
Left: an actual {\em Helmholtz resonator} (reproduced from
\url{https://en.wikipedia.org/wiki/Helmholtz_resonance} under the
Creative Commons licence). Right: 
a general mathematical model allowing for arbitrary cavities and 
exteriors \cite{dgm}. The shaded part is the obstacle $ \mathcal O $ and
resonances are the values of $ \lambda $ for which there exists 
a solution to $ ( - \Delta - \lambda^2 ) u = 0 $ in $ \RR^n \setminus 
\mathcal O $, $ u |_{\partial \mathcal O } = 0 $ which is 
outgoing, that is, $ u |_{ \RR^n \setminus B ( 0 , R) }  
= R_0 ( \lambda) f |_{\RR^n \setminus B( 0 , R ) }$
for some $ f \in L^2_{\comp} ( \RR^n ) $ and $ R \gg 1 $. 
This is also an example  
in which high energy quasimodes supported in the cavity 
provide an approximation to resonances
\cite{TaZ},\cite{StefDuke} with imaginary parts $ \mathcal O ( e^{ - c | \Re \lambda |} ) $.}
\end{figure}

\subsection{Resonance expansions of scattered waves}
\label{expan}

We will use Theorem \ref{t:4} and prove an expansion of solutions to the
wave equation in terms of resonances -- see 
\eqref{eq:uoftx} and Figures \ref{f:wavin} and \ref{f:ringin}. 
To simplify the presentation we make
the following assumption:
\begin{center}
{ $P_V $ has no eigenvalues and $ 0 $ is {\em not} a resonance; }\\
{all resonances of $ P_V $ are {\em simple}\footnote{This is expected to be true generically 
and is claimed in \cite{klop}. The proof is incomplete so this seems to 
be an open problem.},} 
\end{center}
see \cite[\S 2.2,3.2.2]{res} for the general case and \cite[\S 7.5]{res} for
resonance expansions in the case where trapping is allowed.

\begin{thm}
\label{t:5}
Suppose that $ V \in L^\infty ( \RR^3 ) $ is real valued,
$ \supp V \subset B ( 0 , R_0 ) $ and that the assumptions
above holds. 
Let $w(t,x)$ be the solution of                                            
\begin{equation}\label{419}
\begin{cases}
& (\partial^2_t + P_V)w(t,x) =0  \,, \\
& w(0,x)=f(x) \in  H^1_{ }(B ( 0 , R_1 ) )  \,, \\
&  \partial_tw(0,x)=g(x) \in L^2(B( 0 , R_1 ) )  \,,
\end{cases}
\end{equation}
where $ R_1 \geq R_0 $. Then for any $ a  >0  $, 
\begin{equation}\label{420}
\begin{split}
w ( t , x ) =
\sum_{ \Im \lambda_j > - a  } e^{- i 
\lambda_j t } w_{j } ( x )  + E_a ( t )  \,,
\end{split}
\end{equation}
where $ \{ \lambda_j \}_{j=1}^{\infty} $ are the resonances of $ P_V$ 
and $ w_j $ are the corresponding resonant states,
\begin{equation}
\label{eq:ujofx}
w_j = 
{\rm{Res}}_{
  \lambda = \lambda_j} \left(  i R_V ( \lambda ) g + \lambda R_V (
\lambda) f \right) .
\end{equation}
exists  a constants $ C $ depending on $V$, $ R_1 $ and $ a $ such that
\begin{equation}
\label{eq:EAoft}  \| E_a ( t ) \|_{ H^1 ( B ( 0 , R_1 ) ) } 
\leq C e^{ - t a }
\left( \| f \|_{ H^1 } + \| g \|_{ L^2 } \right) \,, \ \ t \geq 0 . 
\end{equation}
\end{thm}

\medskip

\noindent
{\bf Interpretation.} Suppose that instead of solving the wave equation \eqref{419} 
with $ x \in \RR^3 $ we consider it for $ x \in \Omega \Subset \RR^3 $,
with, say, Dirichlet boundary conditions, $ u ( x, t ) = 0 $, $ x \in \partial
\Omega $ (for $ \partial \Omega$ smooth). Then the solution can be 
expanded in a series of eigenfunctions of the Dirichlet 
realization of $ P_V$. Let us assume for simplicity that all eigenvalues
are positive, $ 0 < \lambda_1^2 \leq \lambda_{2}^2  \cdots \leq \lambda_{j}^2 
\to \infty $, $ P_V \varphi_{j} = \lambda_{j}^2 \varphi_{j} $, $ \varphi_{j} |_{\partial \Omega } = 0 $, 
$ \langle \varphi_{i}, \varphi_{j} \rangle_{L^2 ( \Omega ) } = \delta_{ij} $. For 
$ j \in - \NN^* $ we then put
\[  \lambda_{ -j } := - \lambda_{ j} < 0 , \ \ \varphi_{ -j} := \varphi_{j} . \]
The solution of \eqref{419} on $ \RR \times \Omega $ 
has an expansion converging in $ \CI ( \RR ; L^2 ( \Omega ) ) $:
 \begin{equation}
\label{eq:expeig}
\begin{gathered}
w ( t , x) = \sum_{ j \in \ZZ^* } e^{ - i \lambda_j t } u_j ( x ) , \\
u_j (x) = {\textstyle{\frac12}} \left( \langle f , \varphi_j \rangle_{L^2 ( \Omega ) } 
+ i \lambda_j^{-1} \langle g , \varphi_j \rangle_{ L^2 ( \Omega ) } \right) 
\varphi_j ( x ) .
\end{gathered}
\end{equation}
A more invariant way to express $ u_j's  $ is given by 
\eqref{eq:ujofx} with $ R_V ( \lambda ) = ( P_V - \lambda^2 )^{-1} $,
the resolvent of the Dirichlet realization of $ P_V $ on $ \Omega $.

Hence \eqref{420} is an analogue of the standard expansion \eqref{eq:expeig}.
However, for large times 
the solution $ u ( t , x ) $ (that is both 
$u_j$'s and the remainder $ E_A $) are more regular. 
The proof shows (by applying \eqref{eq:t4} with $ j = 2$, which is 
also valid) that 
 \begin{equation}
\label{eq:EAoft1}  \| E_a ( t ) \|_{ H^2 ( B ( 0 , R_1 ) ) } 
\leq C e^{ - t a }
\left( \| f \|_{ H^1 } + \| g \|_{ L^2 } \right) \,, \ \ t > 10 R_1 ,
\end{equation}
and if $ V \in \CI $ we could replace $ H^2 ( B ( 0 , R_1 )  ) $ in \eqref{eq:EAoft1} by 
$ H^p $ for any $ p $. That smoothing is visible in Figure \ref{f:wavin}.
This indicates in a simple case a relation between distribution of resonances
and propagation of singularities -- see \cite{Vai},\cite[\S 4.6]{res} and
also \cite{BS},\cite{BS1} and \cite{Ga} for some recent applications.

In \eqref{eq:EAoft1} we did not aim  at the 
optimality of the condition $ t > 10 R_1 $: using propagation of 
singularities it suffices to take $ t > 2 R_1 $.

\begin{proof}[Proof of Theorem \ref{t:5}]
We first consider \eqref{419} with $f \equiv 0$ and 
$ g \in H^2_{\rm{comp}} $. 
By the spectral theorem,
the solution of \eqref{419} can be written as
\begin{equation}
\label{eq:Utw}
w(t)= U(t) g := \frac{\sin t \sqrt { P_V}}{ \sqrt {P_V}} = 
\int^\infty_0\frac{\sin t\lambda}{\lambda} \,dE_\lambda(g)  \end{equation}
where $ dE_\lambda $ is the spectral measure, which just as in the case of
the free operator $ P_0 $ \eqref{eq:R0E},\eqref{eq:Stone0}, can be expressed using Stone's formula \cite[Theorem B.8]{res}: 
\begin{equation}
\label{eq:dEla}
 d E_\lambda = {
 {\frac{ 1 } { \pi i }}} ( R_V ( \lambda ) - R_V ( - \lambda
 ) ) \lambda d \lambda .
\end{equation}
Hence, as in \eqref{eq:R0wave},
\begin{equation}\label{422}
\begin{split}
w(t) 
&=\frac{1}{\pi i}\int^\infty_0  \sin t\lambda (R_V(\lambda)-R_V(-\!\lambda))gd\lambda
\\
& = \frac{1}{ 2 \pi } \int_\RR
  e^{- it\lambda}( R_V(\lambda)  - 
R_V( - \lambda)) gd\lambda  \,. 
\end{split}
\end{equation} 
To justify convergence of the integral 
we use 
the spectral theorem (see \eqref{eq:dEla}) which shows that
\[  ( R_V ( \lambda ) - R_V ( - \lambda ) ) ( 1 + P_V ) = 
( 1 + \lambda^2)
( R_V ( \lambda ) - R_V ( - \lambda ) ) 
\,. \]
From that we conclude that for $ \chi \in \CIc $ 
equal to $ 1 $ on $ \spt g $, 
\begin{equation}
\label{eq:w1inH2} \begin{split}  \chi( R_V ( \lambda ) - R_V ( - \lambda ) )\chi g 
  & = \chi( R_V ( \lambda ) - R_V ( - \lambda ) )  g \\
&  = \chi( R_V ( \lambda ) - R_V ( - \lambda ) )\chi ( 1 + \lambda^2)^{-1}
( 1+ P_V)  g  \,. \end{split}\end{equation}
Theorem \ref{t:4} and the assumption that there is no zero resonance\footnote{We 
do not prove it here but it is a consequence of {\em Rellich's uniqueness theorem} that for $ V $ real valued $ R_V ( \lambda ) $ has no poles in $ \RR 
\setminus \{ 0 \} $ -- see \cite[\S 3.6]{res}.}  shows that
\[  \| \chi ( R_V ( \lambda ) - R_V ( -\lambda )\chi  \|_{L^2 \to H^1}
 \leq C  , \ \  \lambda \in \RR . \]
Hence the integral on the right hand side of 
\eqref{422} converges in $ H^1_{\rm{loc}} $.

\begin{figure}
\begin{center}
\includegraphics[scale=1]{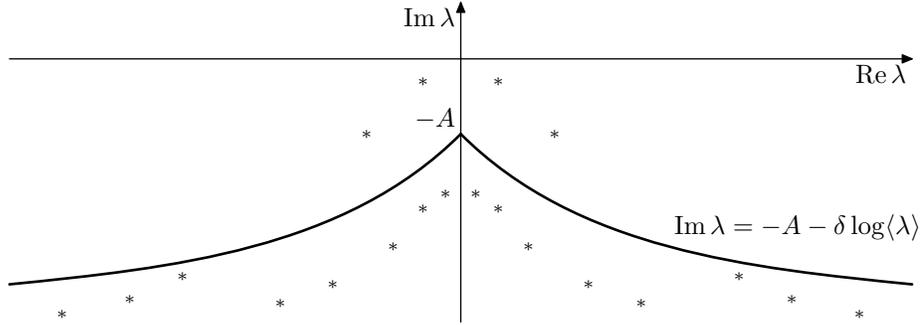}
\caption{\label{f:conto} The contour used to obtain the resonance
expansion.}
\end{center}
\end{figure}

We now want to deform the contour in \eqref{422}. For that 
we choose $ r $ large enough so that all the
resonances with $ \Im \lambda > - a - \delta \log ( 1 + | \lambda | ) $ 
are contained in $ | \lambda | \leq r $. If we choose $ \delta < \frac1{2 R_0 } 
$ that is possible thanks to Theorem \ref{t:4}.

We now deform the contour
of integration using the following contours:
\begin{equation}
\label{eq:defGam}
\begin{gathered}
\Gamma := \{ \lambda - i ( a + \epsilon  + \delta \log ( 1 + |  \lambda |
) )  \; : \; \lambda \in \RR \} \,,    
\\
\gamma_r^\pm  :=  \{ \pm r - it \; : \; 0 \leq t \leq a +
 \epsilon + \delta \log ( 1  + r ) \} \,,  \ \ \gamma_r := \gamma_r^+
 \cup \gamma_r^- \,, \\ 
\Gamma_r := \Gamma \cap \{ | \lambda| \leq r \} \,, \ \ 
\gamma_r^{\infty } := ( - \infty , - r) \cup ( r , \infty) \,. 
\end{gathered}
\end{equation}
Here we choose $ \epsilon $ and so that there are no resonances on $
\Gamma $.  We also put 
\[ \Omega_a  := \{ \lambda \; : \; \Im \lambda \geq  - a - \epsilon - \delta \log ( 1 + |  \lambda |
) \} \,. \]
and define
\[ \Pi_{a}  ( t ) := i \sum_{ \lambda \in \Omega_a } {\rm{Res}}_{
  \lambda = \lambda_j } (  \chi R_V ( \lambda ) \chi e^{ - i \lambda t })   \,. \]
In this notation  the residue theorem shows that
$ U ( t) $ defined in \eqref{eq:Utw} is given by 
\begin{equation}
\label{eq:rUr}
  \chi \, U ( t ) g  =   \Pi_A ( t ) g + E_{\Gamma_r} ( t ) +
E_{\gamma_r} ( t )  + E_{\gamma_r ^\infty } ( t ) 
 \,, 
\end{equation}
where (with natural orientations -- see Figure \ref{f:conto})
\begin{equation}
\label{eq:defEgamma}
E_\gamma ( t ) :=  \frac{1}{ 2 \pi } \int_\gamma 
  e^{- it\lambda} \chi ( R_V(\lambda)) -  R_V( - \lambda)) \chi g d\lambda\,. 
\end{equation} 
Using \eqref{eq:t4} and \eqref{eq:w1inH2} we obtain
\[ \| E_{\gamma^\infty _r} ( t ) g \|_{H^1}  \leq 
C \int_r^\infty ( 1 + |\lambda |^2)^{-1} \|g\|_{H^2}  \leq 
\frac C r \| g\|_{H^2} \to 0 , \ \ \ r \to \infty \,, \]
and
\[ \| E_{\gamma_r} ( t ) g \|_{H^1} \leq \frac {C ( 1 + r )^{4 R_1 \delta} }{ 1 + r^2 } \| g
\|_{H^2} \to 0 , \ \ \ r \to \infty 
\  \ \text{ if $ \delta < \frac 1{ 2 R_1 } $.}
\]

Returning to \eqref{eq:rUr} we see that 
\begin{equation}
\label{eq:rUr'}  \chi \, U ( t) \chi g =  \Pi_A ( t ) g + E_\Gamma
( t ) g \,, \ g \in H^2 \,,
 \end{equation}
where $ E_\Gamma $ is defined using \eqref{eq:defEgamma} with
$ \Gamma $ given in \eqref{eq:defGam}.
 
We now show that when $ t $ is large enough,
\begin{equation}
\label{eq:EGt}  \|  E_\Gamma ( t ) g \|_{H^1} \leq C e^{ - t a } \| g \|_{L^2}
\,.
\end{equation}
For that we use \eqref{eq:t4} with $ j = 1 $ for $ |\lambda | > R $, and the
assumption that there are no poles of $ R_V ( \lambda ) $ near $
\Gamma $. Thus we obtain:
\[ \begin{split}   \|  E_\Gamma ( t ) g \|_{H^1}  &  \leq C e^{ - a t } 
\int_\RR e^{ - t \delta \log ( 1 + | \lambda | ) } e^{  \delta 2 R_1 \log ( 1
  + | \lambda | ) }\| g \|_{L^2 } d \lambda  \\
& \leq C e^{- a t } \int_\RR ( 1 + | \lambda | )^{ -\delta (  t - 2R_1 )
   }  \| g \|_{L^2 }
 d \lambda  \\
& \leq C_\epsilon e^{-at }  \| g \|_{L^2 } \,, \ \  t > 2 R_1 + 1/\delta + 
\epsilon  \,. 
\end{split}
\] 
Since $ H^1 $ is dense in $ L^2 $ the decomposition \eqref{eq:rUr'} 
is valid for $ g \in L^2 $ proving theorem for $ f = 0 $, once $ t $ is
large enough.
But for bounded $ t $, $ \| w ( t )\|_{H^1 } \leq C \| g \|_{L^2} $.
The result follows by noting that $ \chi E_a ( t )  = E_\Gamma ( t) g $.

 The case of arbitrary 
$ f \in H^1_{\rm{comp} } $ and  $g \equiv 0$ follows by replacing
${\sin t\lambda}/{\lambda}$ by $\cos t\lambda$ in the formula for $w(t,x)$.
\end{proof}

The expansion of waves provides an (admittedly) weak approximation to 
correlations (see Figure \ref{f:core}) related to the Breit--Wigner approximation.
Suppose $ g  \in L^2 ( B ( 0 , R_1 ) ) $, $ h \in H^{-1} ( B ( 0 , R_0 ) )$
and put 
$ U ( t ) := \sin t \sqrt P_V / \sqrt P_V $. We define the following correlation
function
\begin{equation}
\label{eq:corr1}  \rho_{g,h} ( t) := \int_{\RR^3 } U ( t ) g ( x ) \overline { h ( x ) } dx, \end{equation}
and its {\em power spectrum}:
\begin{equation}
\label{eq:power1}  \hat \rho_{ g, h } ( \lambda ) = \int_0^\infty \rho_{ g, h } ( t ) e^{ i \lambda
t } d t . \end{equation}
A formal application of the expansion \eqref{422} suggests
\begin{equation}
\label{eq:BW1}
\hat \rho_{ g, h } ( \lambda ) \sim  
\sum_{ \Im \lambda_j > - a } \frac{ \Res_{ \lambda = \lambda_j} 
\langle R_V ( \lambda ) g, h \rangle }{ \lambda - \lambda_j } 
\end{equation}
but it is difficult to give useful remainder estimates in general.
Expansion \eqref{eq:BW1} is just an expansion of a meromorphic function
with simple poles (as assumed here) following an analogue of 
\eqref{eq:rRr}.

For Breit--Wigner type formula near a single resonance in the semiclassical limit see 
G\'erard--Martinez \cite{gma} and G\'erard--Martinez--Robert \cite{gmar}
and for high energy results and results for clouds of resonances, 
Petkov--Zworski \cite{pz1},\cite{pz2} and Nakamura--Stefanov--Zworski
\cite{NaStZw}.

We finally comment on expansions on the case of the Schr\"odinger equation.
In that case the dispersive nature of the equation makes it harder to 
justify resonance expansions. It can be done when a semiclassical parameter
is present, see Burq--Zworski \cite{bz} and 
Nakamura--Stefanov--Zworski \cite{NaStZw}, or when one considers a localized
resonance -- see Merkli--Sigal \cite{messy}, Soffer--Weinstein \cite{soweto}
and references given there.

\subsection{Complex scaling in dimension one}
\label{rae}

In \S \ref{merc3} we established meromorphic continuation using
analytic Fredholm theory. That allowed us to obtain general bounds on the
number of resonances and to justify resonance expansions. To 
obtain more refined results relating geometry to the distribution of 
resonances, as 
indicated in \S \ref{ressl}, we would like to use spectral theory of partial
differential equations. And for that we would like to have a differential
operator whose eigenvalues would be given by resonances. That operator
should also have suitable Fredholm properties. We would call this
an {\em effective meromorphic continuation}.

The method of {\em complex scaling} 
produces a natural family of non-self-adjoint
operators whose discrete spectrum consists of resonances.
It originated in the work of  
Aguilar-Combes \cite{AgCo}, Balslev-Combes \cite{BaCo} and was 
developed by Simon \cite{Si1}, Hunziker \cite{Hunz}, Helffer-Sj\"ostrand 
\cite{He-Sj0}, Hislop--Sigal \cite{His} and other authors.
For very general compactly supported ``black box" perturbations and 
for large angles of 
scaling it was
studied in \cite{SZ1} while the case of long range black box perturbations
was worked out in \cite{SjL}. The method has been 
extensively used in computational chemistry -- see Reinhardt \cite{Rei} 
for a review. As the method of {\em perfectly matched layers}
it reappeared in numerical analysis -- see Berenger \cite{ber}.

In this section we present the simplest case of the method by working
in one dimension\footnote{Our presentation is based on \cite[\S 2]{SZ1}
and on an upublished
 note by Kiril Datchev \url{http://www.math.purdue.edu/~kdatchev/res.ps}.}.
The idea is to consider $ D_x^2 $ as a restriction of the complex
second derivative $ D_z^2 $ to the real axis thought of as a contour
in $ \CC $. This contour is then deformed away from the support of $ V
$ so that $ P = D_x^2 + V ( x ) $ can be restricted to it. This provides ellipticity
at infinity at the price of losing self-adjointness.

Before proceeding we need to discuss the definition of resonances
in dimension one. In that case the free resolvent, 
$ R_0 ( \lambda ) = ( P_0 - \lambda^2 )^{-1} $, $ P_0 = - \partial_x^2 $,
has a very simple Schwartz kernel,
\[  R_0 ( \lambda , x , y ) = \frac{i}{2 \lambda } e^{ i \lambda | x - y| } , \]
which means that as an operator $ R_0 ( \lambda ) : L^2_{\comp} ( \RR ) 
\to L^2_{\rm{loc}} ( \RR ) $, the resolvent continues to a meromorphic
family with a single pole at $ \lambda = 0 $\footnote{This is 
consistent with the expansion \eqref{eq:uoftx} for the free wave
equation: the single pole is ``responsible" for the failure of the 
sharp Huyghens principle in that case.}. 

The approach of \S \ref{merc3} carries through without essential 
modifications and in particular the characterization of  
resonant states \eqref{eq:outg} applies. Since for $ \lambda \neq 0 $,
\[ u ( x ) = R_0 ( \lambda ) f, \  \ f \in L^2_{\comp} 
\ \Longleftrightarrow \ u ( x ) = a_\pm e^{ \pm i \lambda x } , \ \ 
\pm x \gg 1 , \ \ u \in H^2_{\loc} ( \RR ) , \]
we have the following characterization of scattering resonances
\begin{equation}
\label{eq:scatres1}
\begin{gathered}
\text{ $ \lambda \neq 0 $ is a scattering resonance of $ P_V = 
- \partial_x^2 + V ( x ) $, $ V \in L^\infty_{\comp} ( \RR ) $}
\\
\Updownarrow 
\\
\exists \, u \in H^2_{\rm{loc}} ( \RR ) , \ \
( P_V - \lambda^2 ) u = 0 , \ \ 
u ( x ) = a_\pm e^{ \pm i \lambda x } \neq 0 , \ \ 
\pm x \gg 1 , \end{gathered}
\end{equation}
For reasons of simplicity we will not consider multiplicities.
The only possibility here is having algebraic multiplicities 
since solutions satisfying the condition in \eqref{eq:scatres1}
are unique up to a multiplicative constant.

\begin{figure}
\includegraphics[scale=1]{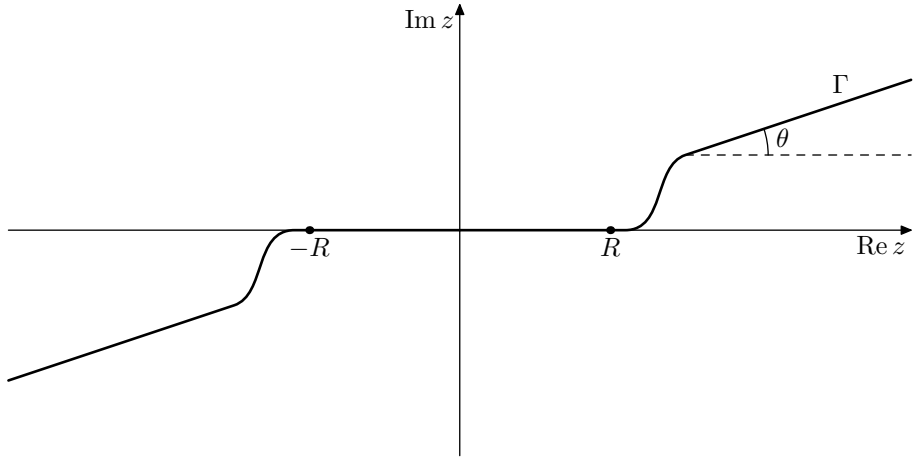}
\caption{\label{f:dat1}
 Curve $\Gamma$ used in complex scaling. The curve is given by $ x
 \mapsto x + i g ( x) $ for 
  a $C^\infty$ function $g$ satisfying $g(x) = 0$ for $-R \le x \le R$
  and $g(x) = x \tan \theta$ for $|x|$ sufficiently large, where
  $\theta$ is a given constant.}
\end{figure}

We now construct an operator which has resonances as its discrete
spectrum. 
For that let  $\Gamma \subset \CC$ be a $ C^1 $ {\em simple curve}.
We define differentiation and integration of functions mapping
$\Gamma$ to $\CC$ as follows. Let 
$\gamma(t)$ be a parametrization 
$\RR \to \Gamma$, and let $f \in C^1(\Gamma)$ in the sense 
that $f \circ \gamma \in C^1(\RR)$. We define
\[
\partial_{z,\Gamma} f(z_0) = \gamma'(t_0)^{-1} \partial_t(f \circ
\gamma)(t_0)\,, \ \ \ \gamma(t_0) = z_0 \,, 
\]
where the inverse and the multiplication are in
the sense of complex numbers.

The chain rule shows that $ \partial^\Gamma_z f ( z_0 )$ 
is independent of parametrization and that for $ F $ holomorphic
near $ \Gamma $,
\[ \partial_{z,\Gamma} { F|_\Gamma }  = \partial_{z} F |_\Gamma \,, \]
where $ \partial_z = \frac 12 ( \partial_x - i \partial_y) $ (see \cite[\S 4.6]{res}). 

We make the following assumption on the behaviour of $ \Gamma $ at infinity:
\begin{gather}
\label{eq:assGam}
\begin{gathered}
\exists \, 0 < \theta < \pi \,, z_\pm \in \CC \,, 
K \Subset \CC , \ \ \ 
\Gamma \setminus K  =
 \bigcup_{\pm} \, \left( z_\pm \pm e^{i \theta } ( 0 , \infty ) \right)
 \setminus K . 
\end{gathered}
\end{gather}

An example of  $ \Gamma $ is shown in
Figure \ref{f:dat1}. One can consider more general behaviour at
infinity such as shown in Figure \ref{f:dat2} where $ \Gamma = \{ x + i g
( x ) : x \in \RR\} $ for a specific $ g$.

Given a potential $ V \in L^\infty_{\rm{comp} } ( \RR; \CC ) $ we
further assume that 
\begin{equation}
\label{eq:GAML}     \Gamma \cap \RR \supset [ - L , L ], \ \ \supp V \subset ( - L
, L ) . \end{equation}
The potential $V$ is then a well defined function on $\Gamma$, so that putting 
\begin{equation}
\label{eq:Pgamma} P_{V, \Gamma} :=  -\partial_{z,\Gamma}^2 + V(z) \,, \end{equation}
makes sense.

We now want to relate eigenvalues of $ P_{V,\Gamma } $ to the
resonances, that is to $ \lambda$'s for which \eqref{eq:scatres1} holds.
For that let us write $ \Gamma $ as a disjoint union of connected components:
\[ \Gamma = \Gamma_- \cup [ - L , L ] \cup \Gamma_+ \]
where $ \Im z \to \pm \infty $ on $ \Gamma_\pm $. 
We then observe that  
\begin{equation}
\label{eq:apmbmp}  (  P_{V, \Gamma } - \lambda^2 ) u_\Gamma  \ \Longrightarrow \
u_\Gamma ( z ) = a_{\pm } e^{ \pm i \lambda z }+ b_{\pm } e^{ \mp i \lambda z }, \ \ 
z \in \Gamma_{\pm } .\end{equation}
That is because on $ \Gamma_\pm $ we are away from the 
support of the potential and the unique solutions to our differential 
equations have to come from holomorphic solutions to 
\begin{equation}
\label{eq:Upm}   ( - \partial_z^2 - \lambda^2 ) U_\pm = 0 , \ \ 
 u_\Gamma |_{\Gamma_{\pm} } = U_\pm|_{\Gamma_\pm} . 
 \end{equation}
Condition \eqref{eq:assGam} shows that
\[ \begin{gathered} 
 \Re ( \pm i \lambda z |_{\Gamma_\pm } ) = 
- \sin ( \theta + \arg \lambda ) |\lambda | | z| + \mathcal O ( 1 ) < 
- \epsilon |\lambda | | z| + \mathcal O ( 1 ) , \\
 \Re ( \mp i \lambda z |_{\Gamma_\pm } ) = 
 \sin ( \theta + \arg \lambda ) |\lambda | | z| + \mathcal O ( 1 ) >
 \epsilon |\lambda | | z| + \mathcal O ( 1 ) .
\end{gathered}
\] 
Hence $ u_\Gamma \in L^2 ( \Gamma ) $ if and only if $ b_\pm = 0 $ in 
\eqref{eq:apmbmp}. But then
\[ u ( x) := \left\{ \begin{array}{ll} U_+|_{[L, \infty )} ( x), & x > L,  \\
\ \ u_\Gamma ( x ) , & x \in [ -L, L ], \\
 U_-|_{  (-\infty , -L  ]} ( x), & x < -  L,
 \end{array} \right.
 \]
satisfies the condition in \eqref{eq:scatres1} and hence $ \lambda $ is 
a resonance.

\begin{figure}
\includegraphics[scale=1]{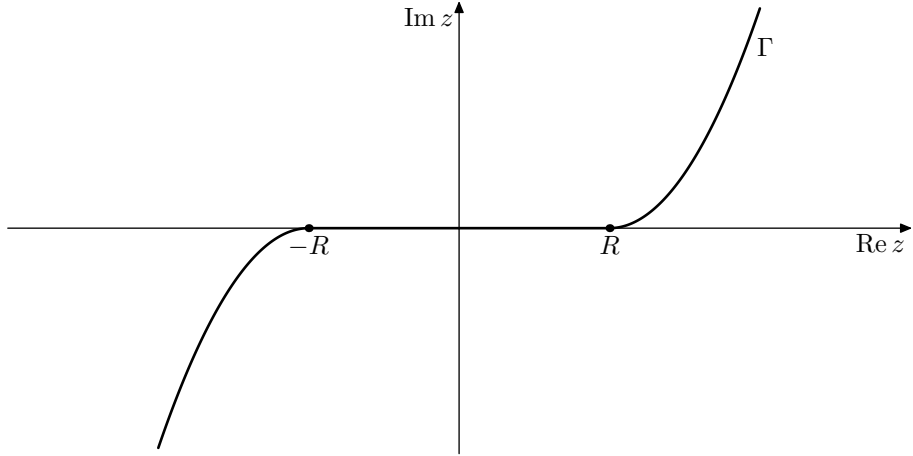}
\caption{\label{f:dat2}
Curve $\Gamma$ used in PML computations -- see \cite{BG} 
and references given there. A typical curve is given by
a function $ \RR \ni x \mapsto x + i g ( x )  $ where 
$ g(x) = -|x+R|^\alpha$ for  $x < -R$, $g(x) = 0$ for $-R \le x \le R$, and $g(x) = (x-L)^\alpha$ for $x > R$, where $\alpha > 1$.}
\end{figure}

This argument can be reversed and consequently we basically proved 
\cite[Theorem 2.20]{res}:

\begin{thm}[\bf Complex scaling in dimension one]
\label{t:cs1}
Suppose that $ \Gamma $ satisfies \eqref{eq:assGam} and $ P_{V,\Gamma} $
is defined by \eqref{eq:Pgamma} with $ V \in L^\infty_{\rm{comp} } (
\RR; \CC ) $. Define
\[ \Lambda_\Gamma := \{ \lambda \in \CC \setminus {\overline \RR}_- : 
- \theta < \arg \lambda < \pi - \theta \}, \]
where $ \arg : {\overline \RR}_-  \to ( - \pi, \pi ) $.

For $ \lambda \in \Lambda_\Gamma $, 
\begin{equation}
\label{eq:Pgamm1} 
P_{V,\Gamma} - \lambda^2 : H^2 ( \Gamma ) \to L^2 ( \Gamma ) \,, 
\end{equation}
is a Fredholm operator and the spectrum of $ P_{\Gamma, V } $ in $
\Lambda_\Gamma $ is discrete. 

Moreover, the eigenvalues of $ P_{V, \Gamma } $ in $ \Lambda_\Gamma $
coincide with the resonances of $ P_V $ there, and  
\begin{equation}
\label{eq:Pgamm2}
m ( \lambda ) = 
{\textstyle{\frac{ 1}{ 2 \pi i} } } \tr_{ L^2 ( \Gamma ) } \oint_\lambda ( \zeta^2 - P_{V,\Gamma})^{-1}
2 \zeta d\zeta ,   \ \ \ \lambda \in \Lambda_\Gamma , 
\end{equation}
where $ m ( \lambda ) $ is the multiplicity of the
resonance at $ \lambda $ (see Definition \ref{d:1}) and the integral is over a sufficiently small positively oriented
circle enclosing $ \lambda $.
\end{thm}

\medskip

\noindent
{\bf Interpretation.}
We first remark that 
\[ \Pi_{ \lambda, \Gamma } = \frac{ 1 } {2 \pi i } \oint_\lambda (
\zeta^2 - P_{V, \Gamma } )^{-1} \,  2 \zeta \, d \zeta : L^2 ( \Gamma
) \to L^2 ( \Gamma ) , \]
is a projection.  
Hence, the method of complex scaling identifies
the multiplicity of a resonance with a trace
of a (non-orthogonal) projection. 
We gain the advantage
of being able to use methods of spectral theory, albeit in the murkier
non-normal setting. 
The resonant states, that is the outgoing solutions
to  $ ( P - \lambda^2 ) u $, are restrictions to $ \RR $ of functions
which continue holomorphically to fuctions with $ L^2 $ restrictions
to $ \Gamma $. Since have dealt only  with compactly supported
potentials our countours $ \Gamma $ had to coincide with $ \RR$
near the support of $ V$.  The method generalizes to the 
case of potentials which are analytic and decaying in conic
neighbourhoods of $ \pm ( L , \infty ) $. Generalizations to higher dimensions 
are based on the same principles but the treatment is no longer as explicit.

\subsection{Other results and open problems}
\label{othr1}
For $ V \in \CIc ( \RR^n ) $, real valued, existence of infinitely many 
resonances was proved by 
Melrose \cite{melb} for $ n = 3 $ 
and by 
S\'a Barreto--Zworski \cite{SaBZ} for all odd $ n $ (including 
for any superexponentially decaying 
potentials).
The surprising fact that there exist complex valued potentials in 
odd dimensions greater than or equal to three that have {\em no}
resonances at all was discovered
by Christiansen \cite{Ch}.

Quantitative statements about the
counting function, $ N_V ( r ) $, of resonances in $ \{ |\lambda | \leq r \}$ were
obtained by Christiansen \cite{Ch1} and S\'a Barreto \cite{Sa}:
\[ \limsup_{ r \to \infty } \frac{ N_V ( r ) } { r } > 0 , \ \ 
V \in \CIc ( \RR^n; \RR ) , \ \ V \neq 0 . \]
For potentials generic in $ \CIc ( \RR^n; \mathbb F)  $ or $ L^\infty_{\rm{c}} ( 
\RR^n ; \mathbb F) $, $ {\mathbb F} = \RR $ or $ \CC $, Christiansen and Hislop
\cite{CH} proved a stronger statement 
\begin{equation}
\label{eq:CH}  \limsup_{ r \to \infty } \frac{ \log N_V ( r ) } {\log r } = n . 
\end{equation}
The limit 
\eqref{eq:CH} 
means that the upper bound $ N_V ( r ) \leq C r^n $ in Theorem \ref{t:3} is 
optimal for generic complex or real valued potentials. The only case 
of asymptotics $ \sim r^n $ for {\em non-radial} potentials was 
provided by Dinh and Vu \cite{dv13} who proved that potentials 
in a large subset of $ L^\infty (B ( 0 , 1 )) $ have resonances satisfying the Weyl law \eqref{eq:radas}.
The proofs of these results use techniques from several complex variables.

Recently Smith--Zworski \cite{SmZ} showed that any real valued
\[  V \in 
 H^{\frac{n-3}2} ( \RR^n ) \cap L^\infty_{\rm{comp}} ( \RR^n) , \]
has infinitely many resonances and
any  $ V \in L^\infty_{\rm{comp}} ( \RR^n) $ 
has some resonances ($n$ odd). It is still a (ridiculous) open question if
every real bounded potential ($ \neq 0 $) has infinitely many resonances.

The most frustrating open problem is existence of
an optimal lower bound. It is not clear if we can expect an asymptotic formula.
\begin{conj}
\label{c:1} 
Suppose $ V \in L^\infty_{\comp}  ( \RR^n ) $, $ n $ odd, is 
real-valued and non-zero. Then there exists $ c > 0 $ such that
\[  N_V ( r ) \geq c r^n , \]
where the counting function $ N_V ( r ) $ is defined in \eqref{eq:NV}.
\end{conj}

All of the above questions can be asked in even dimensions and for obstacle
problems in which case the bound \eqref{eq:NVb} was established by 
Melrose \cite{Mel3} for odd dimensions and by Vodev \cite{Vo2},\cite{Vo3}
in even dimension. A remarkable recent advance is due to Christiansen
\cite{Ch16}
who proved that for any obstacle in even dimensions we always have
$ r^n $ growth for the number of resonances -- see that paper for 
other references concerning lower bounds in even dimension.

Existence of resonances can be considered as a primitive inverse
problem: a potential with no resonances is identically zero.
In one dimension or in the radial case finer inverse results have been obtained: see 
Korotyaev \cite{Ko1},\cite{Ko2}, 
Brown--Knowles--Weikard \cite{bkw}, Bledsoe--Weikard \cite{bw}, 
Datchev--Hezari \cite{dh}  
and references given there.

Another recent development in the study of resonances for compactly 
supported potentials concerns highly oscillatory potentials
$ V ( x ) = W ( x , x/\epsilon ) $ where $ W : \RR^n \times
\RR^n/ \ZZ^n  \to \RR $ (or $\CC $) is compactly supported
in the first set of variables. Precise results for $ n = 1 $
were obtained by Duch\^{e}ne--Vuki\'cevi\'c--Weinstein \cite{dvw}:
 resonances are close to resonances of $ W_0 ( x ) := 
 \int_{\RR^n/\ZZ^n} W ( x , y ) dy $,  and the difference
is given by $ \epsilon^4 \alpha + \mathcal O ( \epsilon^5 ) $
where, in the spirit of homogenization theory, $ \alpha $ can be computed
using an effective potential. In a remarkable follow-up Drouot \cite{Dr} generalized this result
to all odd dimensions and obtained a full asymptotic expansion in powers of  
$\epsilon $.

\section{Some recent developments}
\label{srr}
We will now describe some recent developments
in meromorphic continuation, resonance free regions, resonance counting 
and resonance expansions for semiclassical operators 
$ - h^2 \Delta_g + V $ on Riemannian manifolds with Euclidean and non-Euclidean infinities. We will also formulate some conjectures: some 
quite realistic (even numbered) and
some perhaps less so (odd numbered)\footnote{\label{foot} Small prizes are offered
by the author for the first proofs of the conjectures within five years
of the publication of this survey: a dinner in a 
{\raisebox{-2pt}{\includegraphics[scale=0.03]{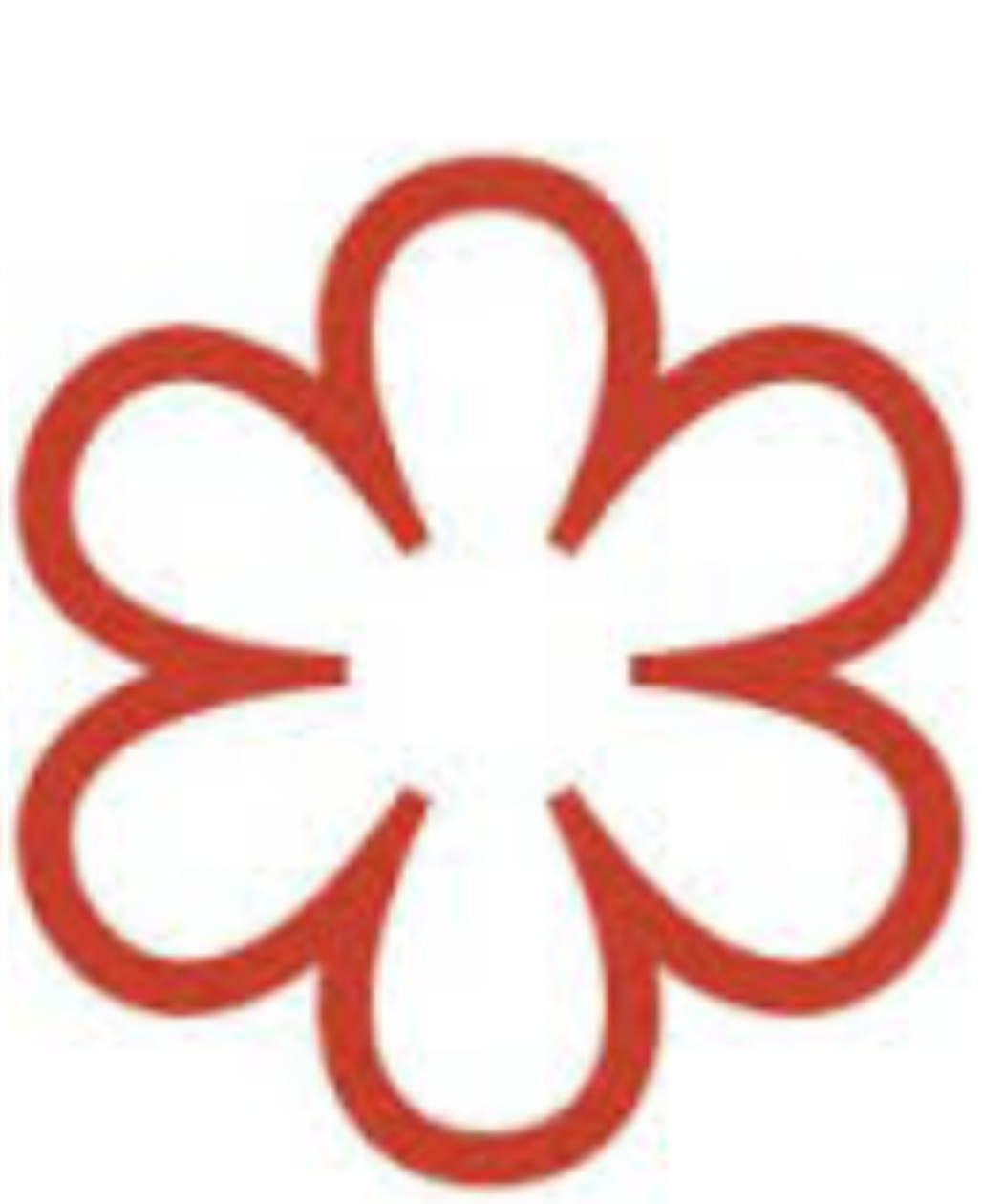}\includegraphics[scale=0.03]{mstar}}} restaurant
for an even numbered conjecture and in 
a 
{\raisebox{-2pt}{\includegraphics[scale=0.03]{mstar}\includegraphics[scale=0.03]{mstar}\includegraphics[scale=0.03]{mstar}}} restaurant, for an odd numbered one.}.

Before doing it let us mention some interesting topics which lie beyond the
scope of this survey. We will discuss hyperbolic trapped sets
but not {\em homoclinic trapped sets.} Although {\em not} stable under
perturbations these trapped sets occur in many situations
 each with its own
interesting structure in distribution of resonances. An impressively precise
study of this has been made by Bony--Fujiie--Ramond--Zerzeri \cite{bfrz}. 
In our survey of counting results we will concentrate on {\em 
fractal Weyl laws} and we refer to Borthwick--Guillarmou \cite{bogus} 
for recent results on resonance counting in geometric settings. In the semiclassical 
Euclidean setting Sj\"ostrand \cite{SjWeyl} obtained precise asymptotic counting
results in the case of {\em random potentials}. The introduction of
randomness should be a beginning
of a new development in the subject. We also do not discuss
the important case of resonances for magnetic Schr\"odinger operators
and refer to Alexandrova--Tamura \cite{alt},  Bony--Bruneau--Raikov, 
\cite{bobrr} and Tamura \cite{tam} for some recent 
results and pointers to the literature. We do not 
address issues around threshold resonances (see Jensen--Nenciu \cite{JN}
and references given there) and their importance in dispersive estimates 
(see the survey by Schlag \cite{Sch})
and for non-linear equations (for an example of a linearly
counterintuitive phenomenon, see \cite[Figure 6]{hz}). Finally, for the role that
shape resonances have had in the study of blow-up phenomena we defer to 
Perelman \cite{Per} and Holmer--Liu \cite{holm} and to the references given there.

The section is organized as follows: in \S \ref{vasy} we review 
Vasy's method for defining resonances for asymptotically hyperbolic manifolds.
In \S \ref{resfree} we first review general results on 
resonance free regions and then describe the case of hyperbolic (in the dynamical
sense) trapped sets. As applications of results in \S \ref{vasy} and 
\S \ref{resfree} we
discuss expansions of waves in black hole backgrounds in
\S \ref{resrel}. The last \S \ref{Weyl} is devoted to mathematical study
of fractal Weyl laws  with some references to the growing physics literature
on that subject.

\subsection{Meromorphic continuation in geometric scattering}
\label{vasy}

In \S \ref{merc3} we showed how to continue the resolvent meromorphically 
across the spectrum using analytic Fredholm theory. In \S \ref{rae}
we presented the {\em method of complex scaling} which provides
an {\em effective meromorphic} continuation in the sense that
resonances are identified with eigenvalues of non-self-adjoint Fredholm
operators. As one can see already in the one dimensional presentation 
that method is closely tied to the structure of the operator near infinity.

With motivation coming from general relativity in physics (see 
\S \ref{resrel}) and from analysis on locally symmetric spaces
it is natural to consider different structures near infinity.
Here we will discuss 
complete asymptotically hyperbolic Riemannian manifolds modeled on 
the hyperbolic space near infinity. For more general hyperbolic 
ends see Guillarmou--Mazzeo \cite{guma} where meromorphy of the 
resolvent was established using analytic Fredholm theory
without providing an effective meromorphic continuation as defined above. 
We also remark that complex scaling method is possible in the case
of manifolds with cusps. For a subtle application of that
see a recent paper by Datchev \cite{Dat3} where existence of  arbitrarily wide
resonance free strips for negative curvature perturbations of 
$ \langle z \mapsto z + 1\rangle\backslash \HH^2 $ is established.
For a recent analysis of a higher rank symmetric spaces and references to earlier
works see Mazzeo--Vasy \cite{mv1},\cite{mv2} and Hilgert--Pasquale--Przebinda \cite{przebic}.

A basic example of our class of manifolds
is given by the hyperbolic space $\mathbb H^{n}$, which can be viewed as the open unit ball in $\mathbb R^n$
with the metric
\index{hyperbolic space}
\begin{equation}
  \label{e:hypspace}
g=4{dw^2\over (1-|w|^2)^2},\quad
w\in B_{\mathbb R^n}(0,1).
\end{equation}
A more interesting family of examples is provided by \emph{convex co-compact hyperbolic surfaces}, which are
complete two-dimensional Riemannian manifolds of constant sectional curvature $-1$ whose infinite ends are \emph{funnels},
that is they have the form
\begin{equation}
  \label{e:funnel-end}
[0,\infty)_v\times \mathbb S^1_\theta,\quad
\mathbb S^1=\mathbb R/\ell \mathbb Z,\
\ell >0;\quad
g=dv^2+\cosh^2 \!v \, d\theta^2.
\end{equation}
Convex co-compact hyperbolic surfaces can be viewed as the quotients of $\mathbb H^2$ by certain discrete subgroups
of its isometry group $\PSL(2;\mathbb R)$ \cite{borth}, and have profound applications in algebra and number theory -- see for instance 
Bourgain--Gamburd--Sarnak \cite{bourgain}. Furthermore,
they give fundamental examples of \emph{hyperbolic trapped sets} and are a model object to study the
effects of hyperbolic trapping on distribution of resonances. Examples which will
be discussed in \S\S \ref{resfree},\S \ref{Weyl} are shown in Figure \ref{f:borth1}.

\begin{figure}
\includegraphics[scale=0.3]{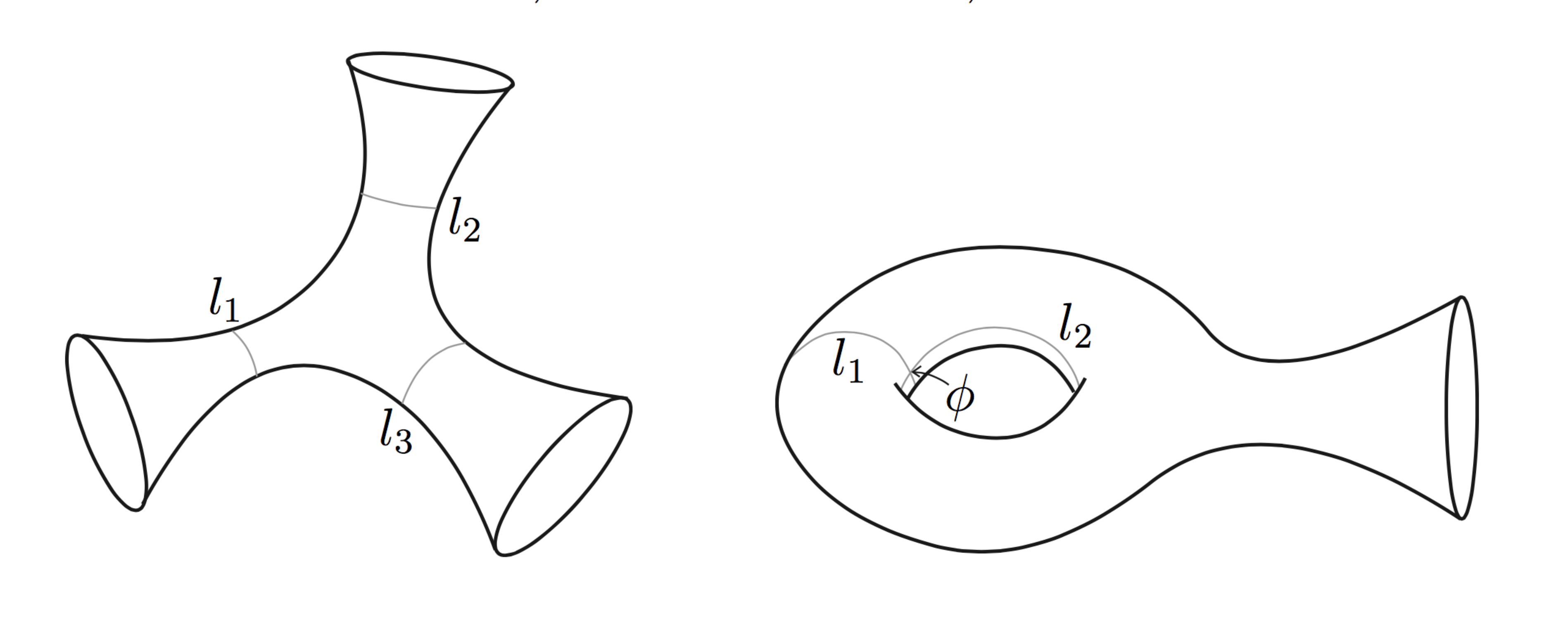}
\caption{\label{f:borth1} Left: surfaces $ X( \ell_1, \ell_2, \ell_3 ) $ with three funnels determined by lengths
of the geodesic boundaries of the funnels \eqref{e:funnel-end} -- see 
\cite[\S 16.2.1]{borth}. Right: funneled tori $ Y( \ell_1 , \ell_2 , \varphi ) $
parametrized by the length of the two geodesics generating homotopy group
of the torus and the angle between. The left example has infinity with three
components; the right one with one component.}
\end{figure}

The examples above are generalized as 
{even asymptotically hyperbolic manifolds}. To define that class of manifolds,
suppose that $  \overline M$ is a compact manifold with boundary $ \partial M \neq \emptyset$ of dimension $ n +1 $. We denote by $ M $ the interior of $ \overline  M $. 
The Riemannian manifold $ ( M , g ) $ is {\em even asymptotically hyperbolic}
if there exist functions $ y' \in \bCI ( M ; \partial M )$ and
 $ y_1  \in {\bCI} (  M ; (0, 2)  ) $\footnote{Here we follow the notation 
of \cite[Appendix B]{H3} where $ \bCI ( M ; V)$ denotes functions 
$ M \to V $, which 
are smoothly extendable across $ \partial M $ and $ \dCI (  \overline M ; V) $ functions which are 
extendable to smooth functions {\em supported} in $ \overline M $.}, 
$ y_1|_{\partial M } = 0 $, $ dy_1|_{\partial 
M } \neq 0 $, such that 
\begin{equation}
\label{eq:coords}  \overline M \supset y_1^{-1} ( [0 , 1 ) ) \ni m \mapsto ( y_1( m), y'( m ) ) 
\in [ 0 , 1 ) \times \partial M \end{equation}
is a diffeomorphism, 
and near $ \partial M $ the metric has the form,
\begin{equation}
\label{eq:gash}   g|_{ y_1 \leq \epsilon }  = \frac{ dy_1^2 + h ( y_1^2 )  }{ y_1^2 } ,  \end{equation}
where 
 $ [ 0, 1 ) \ni t \mapsto h ( t ) $, 
is a smooth family of Riemannian metrics on $ \partial M $.
 For 
the discussion of invariance of this definition and of its geometric meaning
we refer to \cite[\S 5.1]{res} and \cite[\S 2]{g}. Here we will point out how it fits 
with the two examples. For \eqref{e:hypspace} we have 
\begin{equation} 
\label{eq:hypspace} 
\begin{gathered}
 g=4{dw^2\over (1-|w|^2)^2} = \frac{ dy_1^2 + h ( y_1^2 ) }{ y_1^2}, \\ y_1 := \frac{2 ( 1 - |w| ) }{ 1 + |w|} \in [0, 2) , \ \ 
y' := \frac{w}{|w|} \in \SP^{n-1} , \ \
h ( t) = { ( 1 - t)^2 } h_0 (y',dy') ,
\end{gathered}
\end{equation}
where $ h_0 $ is the standard metric on the unit sphere $ \SP^{n-1} $.
In the case of \eqref{e:funnel-end},  we have a coordinate system $ y$ on $ \{  v > 0 \} $:
\begin{equation}
\label{eq:funnel-end}
\begin{gathered}
g = dv^2 + \cosh^2 \! v \, d \theta^2 = \frac{ dy_1^2 + h (y_1^2 ) }{y_1^2} , \\
y_1 = 2 e^{-v}  \in [0, 2 ) , \ \ y_2 = \theta \in 
\SP^1, \ \ h(  t ) = ( 1 + t)^2 dy_2^2 .
\end{gathered}
\end{equation}

Let $ - \Delta_g \geq 0 $ be the Laplace--Beltrami operator for the metric $ g $.
Since $ - \Delta_g $ is a self-adjoint operator, the spectrum is contained in $ [ 0 , \infty ) $ the operator 
$- \Delta_g - \lambda^2 -  (\frac n 2)^2 $ is invertible on $ H^2 ( M , d
\vol_g ) $ for $ \Im \lambda > \frac n 2  $. Hence we can define
\begin{equation}
\label{eq:Res4Deltag}   R (\lambda ) := ( - \Delta_g - \lambda^2 - 
{\textstyle{\frac {\, n^2} 4 }}  )^{-1} : L^2 ( M, d\!\vol_g ) \to H^2 ( M, d \!\vol_g ) , \ \ 
\Im \lambda  > {\textstyle{\frac n 2}}  . 
\end{equation}
It turns out that $ R ( \lambda ) 
: L^2 \to H^2 $ is meromorphic for $ \Im \lambda  > 0 $: the poles correspond
to $ L^2 $ eigenvalues of $ - \Delta_g $ and hence lie in $ i (0 , \frac n 2 ) $.
A closely related standard fact is that the continuous spectrum of 
$ - \Delta_g $ is equal to $ [ (\frac n 2)^2 , \infty ) $. This explains
our ``shifted" convention in defining $ R ( \lambda ) $.

Elliptic regularity shows that $ R ( \zeta ) : \dCI ( M ) \to 
\CI ( M )$, $ \Im \lambda  > n/2 $. 
Hence is natural to consider {\em meromorphic continuation} of 
\begin{equation}
\label{eq:merchyp}
R ( \lambda ) \; : \; \dCI ( M ) \longrightarrow  \CI ( M ) , \ \ 
\lambda \in \CC . 
\end{equation}
That meromorphy was first established
for any asymptotically hyperbolic metric (that 
is a metric of the form \eqref{eq:gash} but with 
$ h = h ( y_1 ) $) by Mazzeo--Melrose \cite{mm}.
Other early contributions were made by Agmon \cite{Ag}, Fay \cite{Fa}, Lax--Phillips \cite{LaPh}, Mandouvalos \cite{Man}, Patterson \cite{Pa} and Perry \cite{Pe}. 
Guillarmou \cite{g} showed that the evenness condition was needed for 
a global meromorphic continuation and clarified the construction given in 
\cite{mm}. 

All these arguments relied on analytic Fredholm theory \eqref{eq:AFT}
and did not provide an effective continuation in the sense \S \ref{rae}.
A recent breakthrough due to Vasy \cite{vasy1} provided such 
effective continuation by expressing $ R ( \lambda ) $ using 
$ P ( \lambda )^{-1} $ where 
\[  \lambda \mapsto P ( \lambda )  \ \text{ is a
holomorphic family of differential operators.}\] 
Hence, microlocal methods can now be used to prove 
results which were not available before, for instance
existence of resonance free strips for non-trapping metrics \cite{vasy2}. 
Other applications in the theory of resonances will be 
presented in \S\S \ref{resfree},\ref{Weyl}. Roughly speaking,
thanks to Vasy's method we can now concentrate on the interaction 
region where interesting dynamical phenomena occur and treat
infinity as a black box, almost in the same way as complex scaling
allowed us in the Euclidean case. (See Wunsch--Zworski
\cite{WuZw} for a class of 
asymptotically Euclidean manifolds to which the method  of 
complex scaling also applies.) We also mention some 
applications not covered by this survey: 
a quantitative version of Hawking radiation by Drouot \cite{Dr}, 
exponential decay of waves in the Kerr--de Sitter case and
the description of quasi-normal modes for perturbations of Kerr--de Sitter 
black holes by Dyatlov \cite{xpd},\cite{dya},
rigorous definition of quasi-normal modes for Kerr--Anti de Sitter 
black holes\footnote{A related approach to meromorphic continuation, also motivated by the study of 
Anti-de Sitter black holes, was independently developed by Warnick \cite{Wa}.
It is based on physical space techniques for hyperbolic equations
and it also provides meromorphic continuation of resolvents for even 
asymptotically hyperbolic metrics \cite[\S 7.5]{Wa}.}
 by Gannot \cite{ga}. The construction of the Fredholm family 
also plays a role in the study of linear and
non-linear scattering problems -- see the works of
Baskin--Vasy--Wunsch \cite{BaVWu}, Hintz--Vasy \cite{HiV1},\cite{HiV2}
and references given there. In particular, it is important in the
recent proof of nonlinear stability of Kerr--de Sitter black holes
by Hintz and Vasy \cite{HiV3}.

We will follow the presentation of \cite{V4D} and for simplicity will
not consider the semiclassical case. That of course is essential 
for applications and can be found in a textbook presentation of 
\cite[Chapter 5]{res}.

Let $ y' \in \partial M $ denote the variable on $ \partial M $.
Then \eqref{eq:gash} implies that near $ \partial M $, the Laplacian has the form 
\begin{gather}
\label{eq:Deltagg} 
\begin{gathered} - \Delta_g = ( y_1 D_{y_1} )^2 + i ( n  + y_1^2 \gamma ( y_1^2, y') ) y_1
D_{y_1} - y_1^2 \Delta_{h (y_1^2) } , \\
\gamma ( t, y') := - \partial_t \bar h ( t ) / \bar h ( t ) , \ \
\bar h ( t ) := \det h ( t ) , \ \ D := \textstyle{\frac 1 i } \partial .
\end{gathered}
\end{gather}
Here  $ \Delta_{h (y_1^2 ) } $ is the Laplacian for the family of metrics on $ \partial M $ depending smoothly on $ y_1^2 $ and 
$ \gamma \in \CI ( [0,1]\times \partial M ) $. (The logarithmic derivative 
defining $ \gamma $ is independent of of the density on $ \partial M$ needed to 
define the determinant $ \bar h $.)

As established in \cite{mm} (see \cite[\S 6]{V4D} for a direct argument)  
the unique $ L^2 $ solution, $ u $, to 
$ ( - \Delta_g - \lambda^2 - ({\textstyle{\frac n 2 }})^2  ) u =f \in 
\dCI ( M) $, $ \Im \lambda > 0 $,
satisfies
\begin{equation}
\label{eq:u2F} u = y_1^{-i \lambda + \frac n 2 } \bCI ( M )  \ \ \ \text{ 
and $ \ \ \ \   y_1^{ i \lambda - \frac n 2 }u |_{ y_1 < 1 } = 
F ( y_1^2 , y') , \ \ F \in \bCI ( [0,1 ]\times \partial M )  $.} \end{equation}
This suggests two things:

\begin{itemize}

\item in order to reduce the investigation to the study of regularity we 
should conjugate $ - \Delta_g  $ by the weight $ y_1^{-i \lambda + \frac n 2 } $.

\item the desired smoothness properties should be stronger in the 
sense that the functions should be smooth in $ (y_1^2 , y' ) $.

\end{itemize}

Motivated by this we calculate, 
\begin{equation}
\label{eq:firstconj} y_1^{ i \lambda - \frac n 2 }   ( - \Delta_g - 
\lambda^2 - ({\textstyle{\frac n 2}})^2  ) y_1^{-i \lambda + \frac n 2 } =
x_1 P ( \lambda ) , \ \   x_1 = y_1^2 , \ \ x' = y',   , \end{equation}
where, near $ \partial M $, 
\begin{equation}
\label{eq:Plag} P ( \lambda ) = 4 ( x_1 D_{x_1}^2 - ( \lambda + i ) D_{x_1} ) 
- \Delta_h + i \gamma ( x ) \left( 2 x_1 D_{x_1} - \lambda - i 
{\textstyle\frac{ n-1} 2 } \right) . \end{equation}
We note that \eqref{eq:firstconj} makes sense globally since away from the 
boundary $ y_1 $ is a smooth non-zero function on $ M $. 

To define the operator $ P ( \lambda ) $ geometrically we introduce a new
manifold using coordinates \eqref{eq:coords} and $ x_1 = y_1^2 $ for $ y_1 > 0 $:
\begin{equation}
\label{eq:coordX}  X = [ -1 , 1 ]_{x_1}  \times \partial M \sqcup \left( M \setminus y^{-1} 
(( 0, 1 ) ) \right). \end{equation}
We note that $ X_1 := X \cap \{ x_1 > 0 \} $ is diffeomorphic to $ M $ but
$ \overline X_1 $ and $ \overline M $ have different 
$ \CI $-structures.
We can extend $ x_1 \to h ( x_1 ) $ to a family of smooth non-degenerate
metrics on $ \partial M $ on $ [-1,1]$. Using \eqref{eq:Deltagg} 
that provides a natural extension 
of the function $ \gamma $ appearing  \eqref{eq:firstconj}.
If in local coordinates we define
\begin{equation}
\label{eq:dmug}  d \mu_g =  \bar h ( x) dx . \end{equation}
then with respect to $ L^2 = L^2 ( X, d\mu_g ) $ we have 
\begin{equation}
\label{eq:formaladj}
P ( \lambda )^* = P ( \bar \lambda ) .
\end{equation}

To define spaces on which $ P ( \lambda ) $ is a Fredholm operator
we recall the notation for Sobolev spaces on manifolds with boundary. 
We denote by $ \bar H^s ( X^\circ ) $ the space
of restrictions of elements of $ H^s $ on an extension of $ X$ across the 
boundary to the interior of $ X $.

\begin{defi}
\label{d:hypXY}
We define the following Hilbert spaces 
\begin{equation}
\label{eq:hypXY}
\begin{gathered}  \mathscr Y_s := \bar H^s ( X^\circ ) ,  \ \
\mathscr X_s := \{ u \in \mathscr Y_{s+1} : P ( \lambda) u \in \mathscr Y_s \},
\end{gathered}
\end{equation}
and the operator
$ P ( \lambda ) : \mathscr X_s \to \mathscr Y_s $,
given by \eqref{eq:firstconj} and \eqref{eq:Plag}.
\end{defi}

Since the dependence on $ \lambda $ in $ P ( \lambda ) $ occurs only in 
lower order terms we can replace $ P ( \lambda ) $ by $ P ( 0 ) $ 
in \eqref{eq:hypXY}. Hence the definition of $ \mathscr X_s $ is
independent of $ \lambda $.

\noindent
{\bf Motivation:} Since for $ x_1 < 0 $ the operator $ P ( \lambda ) $
is hyperbolic with respect to the surfaces $ x_1 = a < 0 $ we can 
motivate Definition \ref{d:hypXY} as follows.  Consider
$ P = D_{x_1}^2 - D_{x_2}^2 $ on $ [ -1, 0 ] \times \SP^1 $ and define
\begin{gather*}  Y_s := \{ u \in \bar H^s ( [-1, \infty ) \times \SP^1 ) : \supp u \subset 
[ -1, 0 ] \times \SP^1 \}, \\ X_s := \{ u \in Y_{s+1} : P u \in Y_s \}. 
\end{gather*}
Then standard hyperbolic estimates -- see for instance 
\cite[Theorem 23.2.4]{H3} -- show that for any $ s \in \RR $, the operator
$ P : X_s \to Y_s $ is invertible. Roughly, the support condition 
gives $ 0 $ initial values at $ x_1 = 0 $ and hence $ P u = f $ can 
be uniquely solved for $ x_1 < 0 $. 


We can now state the result about mapping properties of $ P ( \lambda ) $:
\begin{thm}
\label{t:6}
Let $ \mathscr X_s , \mathscr Y_s $ be defined in \eqref{eq:hypXY}.
Then for $ \Im \lambda > - s - \frac12 $ the operator
$   P ( \lambda ) : \mathscr X_s \to \mathscr Y_s $,
has the Fredholm property, that is 
$\dim \{ u \in \mathscr X_s : P ( \lambda ) u = 0 \} < \infty$ ,
$ \dim \mathscr Y_s / P ( \lambda ) \mathscr X_s < \infty $, 
and $ P ( \lambda ) \mathscr X_s $ is closed.

Moreover for $ \Im \lambda > 0 $, $ \lambda^2 + (\frac n 2 )^2 \notin 
\Spec ( - \Delta_g ) $ and $ s > -\Im \lambda - \frac12 $, 
$ P ( \lambda ) : \mathscr X_s \to \mathscr Y_s $
is invertible. 
Hence,  for $ s \in \RR $ and $ \Im \lambda > - s - \frac12 $,
$ \lambda \mapsto P ( \lambda )^{-1} : \mathscr Y_s \to \mathscr X_s , $
is a meromorphic family of operators with poles of finite rank.
\end{thm}

In view of \eqref{eq:firstconj} this immediately recovers 
the results of Mazzeo--Melrose \cite{mm} and Guillarmou \cite{g}
in the case of even metrics
(Guillarmou showed that for generic non-even metrics global meromorphic
continuation does not hold; he also showed that
the method of \cite{mm} does provide a meromorphic continuation 
to $ \CC \setminus -i \NN^* $ for all asymptotically hyperbolic 
metrics and to $ \CC $ for the even ones):

\begin{thm}
\label{t:7}
Suppose that $ ( M , g)  $ is an even asymptotically hyperbolic manifold
and that $ R ( \lambda ) $ is defined by \eqref{eq:Res4Deltag}.
Then 
$    R (\lambda ) : \dCI ( M ) \to \CI ( M ) $, 
continues meromorphically from $ \Im \lambda > \frac n 2 $
to $ \CC$ with poles of finite rank. 
\end{thm}
For self-contained proofs of these theorems we refer to \cite{V4D},
and for the semiclassical version in the same spirit to \cite[\S 5.6]{res}. 
The general idea is to estimate $ u $ in terms of  $ P ( \lambda ) u $
with lower order corrections:
\[  \| u \|_{\bar H^{s+1} (X^\circ ) } \leq C \| P ( \lambda ) u \|_{\bar H^s ( 
X^\circ ) } + C \| \chi u \|_{ H^{-N} ( X ) } , \]
where $ \chi \in \CIc ( X^\circ )$, $ \chi = 0 $ in $ x_1 < - 2 \delta $, 
$ \chi = 1 $ in $ x_1 > - \delta $, $ \delta > 0 $. That is done by using
ellipticity of $ P ( \lambda ) $ in $ x_1 > \delta $, hyperbolicity of
$ P ( \lambda ) $ in $ x_1 < - \delta $ and propagation of singularities
 in the transition region $ x_1 = 0 $ -- see \cite[\S 4,5]{V4D}. 

The key propagation estimate used by Vasy \cite{vasy1} comes from the
work of Melrose on propagation estimates at {\em radial points}
occurring in scattering on asymptotically Euclidean spaces \cite{mel}. These
estimates also play a role in  applications of microlocal 
methods to dynamical systems -- see \S \ref{dsPR}. 
Here we will only state one basic consequence -- see \cite[\S 4, Remark 3]{V4D}:
\begin{equation}
\label{eq:regPla}
P ( \lambda ) u \in \bCI ( X ) , \ \ u \in \bar H^{s+1} ( X ) , \ \ 
s > - \Im \lambda - {\textstyle{\frac12}} \ 
\Longrightarrow \ u \in \bCI ( X) . 
\end{equation}
This means that above a threshold of regularity given by 
$ - \Im \lambda + \frac12 $ the operator is hypoelliptic and that
\begin{equation}
\label{eq:charker}  \ker_{ \mathscr X_s}  P ( \lambda ) \neq \{ 0 \} , \ \
s > - \Im \lambda - {\textstyle{\frac12}} \ 
\Longleftrightarrow \ \exists \, u \in \bCI ( X) ,  \ P ( \lambda ) u = 0 .
\end{equation}
This implies that resonant states, in the original 
coordinates, are characterized by 
\begin{equation}
\label{eq:reshyp}
( - \Delta_g - \lambda^2 - (\textstyle{\frac n 2})^2 ) w = 0 , \ \ 
y_1^{ i \lambda - \frac n 2 } w \in \bCI ( M ) ,
\end{equation}
which should be compared to \eqref{eq:outg}.

Some understanding of \eqref{eq:regPla} can be obtained as follows.
Suppose we consider a simplified operator $ P_0 ( \lambda ) = 
x D_{x}^2 - ( \lambda + i ) D_{x}   $, 
Then solutions of $ P_0 ( \lambda ) u = 0 $  are given by 
\[  u ( x ) = a_+ x_+^{ i \lambda } + a_- x_-^{i\lambda } + b , \ \ 
\lambda \notin - i \NN , \ \ x \in \RR . \]
Here we use the notation of \cite[\S 3.2]{H1}. From \cite[Example 7.1.17]{H1}
we then see that $ x_\pm^{i \lambda } \in H_{\loc}^{  - \Im \lambda + \frac12 - }\setminus
H_{\loc} ^{ - \Im \lambda + \frac12 }$. Hence, 
\[  u \in \bar H^{s+1} ( (-1,1) ) , \ \  s > - \Im \lambda  - {\textstyle{\frac12}}
\ \Longrightarrow \ 
 a_\pm =  0 \ \Longrightarrow \ u \in \bCI ( ( -1,1) ).  \]
This also shows that \eqref{eq:regPla}, and consequently Theorem \ref{t:6},
are essentially optimal.

The actual proof of \eqref{eq:regPla} is based on the analysis of 
the Hamilton flow of the principal symbol of $ P ( \lambda ) $, 
$ x_1 \xi_1^2 + |\xi'|^2_{h(x)} \in \CI ( T^* X ) $, and of 
positive commutator estimates depending on lower order terms -- that is where
the dependence on $ \lambda $ comes from. For that we refer to 
 \cite[\S 4]{V4D},\cite[\S E.5.2]{res}.

One weakness of the method lies in the fact that it provides
effective meromorphic continuation only in strips, even though
the resolvent is meromorphic in $ \CC $.
Hence, results which involve larger regions (such as
asymptotics of resonances for convex obstacles 
\cite{SZ5},\cite{Long} where resonances lie in cubic regions,
$ \Im \lambda \sim - | \Re \lambda |^{\frac13} $) are still inaccessible
in the setting of asymptotically hyperbolic manifolds (or even 
$ \HH^n \setminus \mathcal O $).
Analyticity near infinity should play a role when larger regions
are considered and towards that aim we formulate a conjecture which 
could perhaps interest specialists in analytic hypoellipticity. It is
an analytic analogue of \eqref{eq:regPla} and it  also has a
microlocal version:
\begin{conj}
\label{c:hypo}\footnote{After this survey first appeared Conjecture \ref{c:hypo}  was proved by Zuily \cite{Zu16} who used results of Bolley--Camus \cite{BoCa1} and Bolley--Camus--Hanouzet \cite{BoCa2}. These methods also showed analyticity of radiation patterns of 
resonant states, $ F $ in \eqref{eq:u2F}, also when the metric is not even. The award of the prize (see footnote \ref{foot}) took place at 
\url{http://www.latabledulancaster.fr} on November 30, 2016.}
Suppose that $ P ( \lambda ) $ is given by \eqref{eq:Plag} and that
near $  x_1 = 0 $ the coefficients of $ P ( \lambda ) $ are real 
analytic. Let $ U $ be a sufficiently small neighbourhood of $ x_1 = 0 $.
Then,
\[  P ( \lambda ) u \in C^\omega ( U ) , \ \ 
u \in H^{s+1} ( U ) , \ s > - \Im \lambda - {\textstyle{\frac12}} 
\ \Longrightarrow \ u \in C^\omega ( U ) , \]
where $ C^\omega $ denotes the space of analytic functions.
\end{conj}
We remark that in the analytic case the operator $ P ( \lambda ) $
belongs to the class of {\em Fuchsian differential operators}
studied by Baouendi--Goulaouic \cite{BaG} and that the conjecture
is true for $ P ( \lambda ) = 4 ( x_1 D_{x_1}^2 - ( \lambda + i ) 
D_{x_1 } ) - \Delta_h $, where $ h $ is a metric on $ \partial M$,
independent of $ x_1 $ and $ \lambda \notin - i \NN^* $ \cite{LebZ}.

\subsection{Resonance free regions}
\label{resfree}

We will  now consider semiclassical operators 
\begin{equation} 
\label{eq:defPofh} P = P ( h ) :=  - h^2 \Delta_g + V , 
\end{equation}
on Riemannian manifolds $ ( M , g ) $ where $ M$ 
is  isometric to $ ( \RR^n, g_0 )  $ outside of a compact set,
with $ g_0 $ the Euclidean metric, and $ V \in \CIc ( \RR^n ; \RR ) $.
More general classes of metrics and potentials can be considered -- see
\cite{SjL}. We could also generalize the class of manifolds 
 -- see \cite{WuZw}. Later in this section we will also discuss the
case of asymptotically hyperbolic manifolds of \S \ref{vasy}.

The method of complex scaling (see \cite{SjL},\cite{WuZw}) gives
(see Figure \ref{f:res1})
\begin{equation}
\label{eq:mercP}
\begin{split}
( P - z )^{-1} : \CIc ( M ) \to \CI ( M ) \ \ & \text{ continues
meromorphically from $ \Im z > 0 $}\\
& \ \ \ \ \text{  to $ \Im z > - \theta \Re z $, $ \Re z > 0 $.}
\end{split}
\end{equation}
We denote the set of poles of $ ( P( h)  - z )^{-1} $,
that is the set of resonances of $ P ( h ) $, by $ \Res ( P ( h ) ) $
(we include $ h $ to stress the dependence on our semiclassical 
parameter).

\begin{figure}
\includegraphics[width=3.5in]{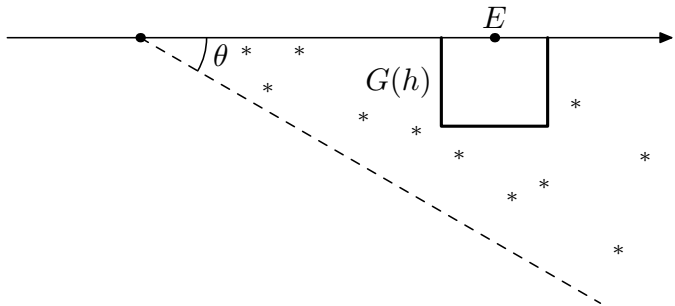}
\caption{\label{f:res1}
The meromorphic continuation \eqref{eq:mercP}: 
The resonances of $ P $ are identified with the eigenvalues the 
scaled operator $ P_\theta $ constructed by a higher dimensional
version of the method presented in \S \ref{rae}. Resonances are 
studied near a fixed energy level in a neighbourhood of size 
depending on $ h $.}
\end{figure}
The asymptotic parameter $ h$ is supposed to be small and we
will consider resonances of $ P ( h ) $ near a fixed energy level.
When $ V \equiv 0 $
then the limit $ h \to 0 $ corresponds to the high energy limit
but even in that case the link with physical intuitions
around classical/quantum correspondence is useful. 

The distribution of resonances is closely related to the properties
of the classical Hamiltonian flow
\begin{equation}
\label{eq:Hamf}
\varphi_t : T^* M \to T^* M , \ \ \varphi_t := \exp ( t H_p ) , \ \
p := |\xi|^2_g + V ( x ) , 
\end{equation}
where $ T^* M $ is the cotangent bundle of $ M $, and $ H_p $ 
is the Hamilton vector field of $ p $. In local coordinates
$ ( x, \xi ) \in T^*M $, $ x \in M $, 
\begin{equation}
\label{eq:Hamp}  H_p = \sum_{ j=1}^n \partial_{\xi_j} p \, \partial_{x_j}
- \partial_{x_j} p \, \partial_{\xi_j} .
\end{equation}

To define the trapped set we first define incoming and outgoing sets
at a given interval of energies $ J \subset \RR  $:
\begin{equation} 
\label{eq:defGama}
\begin{gathered}
\Gamma^\pm_J := \{ ( x, \xi ) \in p^{-1} ( J ) : \pi ( \varphi_ t( x, \xi ) )
\not \to \infty \text{ as $ t \to \mp \infty $} \}, \\
\pi : T^* M \to M , \ \ \pi ( x, \xi ) := x  . 
\end{gathered}
\end{equation}
and the trapped set at energy $ E $:
\begin{equation}
\label{eq:defKE}
K_J = \Gamma_J^+ \cap \Gamma_J^- .
\end{equation}
If $ J =\{ E\} $ then we write $ K_{ \{ E \} } = K_E $.
The simplest is ``non-trapping scattering" in which $ K_E = \emptyset $. 
This implies that 
$ K_J = \emptyset $ for some neighbourhood $ J$ of $ E $
 -- see \cite[\S 6.1]{res}
for other basic properties of $ \Gamma^\pm $ and $K$.

The behaviour of the flow near the trapped set has an effect
on the distribution of scattering resonances and in particular on 
the type of resonance free regions. As in \S \ref{resfree3} we consider
different possibilities:
\begin{equation}
\label{eq:resf1bis}
\Im z > - G ( h  )  , \ \ | \Re z - E | < \delta   , \ \ G ( h ) = 
\left\{ \begin{array}{ll} 
(a) & e^{ - \alpha /h} , \ \ \alpha > 0 \\
(b) & \gamma h  \\
(c) & M h \log {\textstyle{\frac 1 h } }     \\
(d) &  \gamma h^\beta  , \ \ \beta \in \RR , \gamma > 0 \end{array} 
\right.
\end{equation}
In the case of $ P ( h ) = - h^2 \Delta_g $ (or for obstacle problems)
this is equivalent to resonance free regions \eqref{eq:resf1} with 
\begin{equation}
\label{eq:FtoG}  F ( x ) = 2 x G ( 1/x ) .
\end{equation}
Note that $ \lambda^2 = h^2 z $; this 
is consistent with dynamical quantities defined for $ p = |\xi|_g^2 $
in the $ z$-picture and for $ p = |\xi|_g $ in the $ \lambda$ picture.
Nevertheless this is potentially confusing when comparing results involving
\eqref{eq:pressure} and \eqref{eq:Lyap}.

The following table indicates various stable dynamical configurations with 
pointers to the literature (we refer to \eqref{eq:resf1bis} for the types of resonance free regions):
\begin{center}
\begin{tabular} {|| l | l ||}
\hline
\ & \ \\
\ \ \ \ \ \ Hamiltonian flow  & \ \ \ \ \ \ \ \ \ 
 Resonance free region and resolvent bounds \\
\ & \ \\
\hline
General case & (a) \cite{bu},\cite{cv},\cite{vo},\cite{Dat2},\cite{rt},\cite{jsh},\cite[\S 6.4]{res}; 
\\
\ & optimal for shape resonances \cite{He-Sj0},\cite{Marr}, 
and \\
\ &  for ``resonances from quasimodes"  
\cite{TaZ},\cite{StefDuke}\\
\ &  \cite{gann},\cite[\S 7.3]{res}; corresponding cut-off resolvent 
\\
\ & bounds (cf.\eqref{eq:t4}) also optimal \cite{ddz}
\\
\hline
Normally hyperbolic trapping & (b) with $ \gamma $ given by a ``Lyapounov exponent"
\cite{dy4}, \\
\ & \cite{GeSj1}, \cite{NZ3}, \cite{WuZw2},\cite[\S 6.3]{res};    \\
\ & optimal for one closed hyperbolic orbit \cite{GeSj} \\
& \ and for $r$-normally hyperbolic
trapping \cite{dy} 
\\
\hline
Hyperbolic trapped set & (b) with $ \gamma $ determined by a 
{\em topological pressure} of   \\
\ & the trapped set \cite{NZ1}; an improved gap expected; \\
\ &  known for hyperbolic quotients \cite{Naud},\cite{Stoya},\cite{DyZa}
\\
\hline
Smooth non-trapping  & (c) with arbitrarily large $ M$ \cite{M2},\cite{SZ10} \\
\hline
$a$-Gevrey non-trapping  & (d) with $ \beta = 1- 1 /a $ \cite{rou} \\
\hline
Analytic non-trapping & (d) with $ \beta = 0 $ \cite{He-Sj0},\cite{SjDuke};
optimal with density   \\
\ &  $ \sim h^{-n} $  via Sj\"ostrand's trace formula \cite{SjX},
\cite[\S 7.4]{res} \\
\hline
\end{tabular}
\end{center}
(See Definitions \ref{d:hyp},\ref{d:nhyp} for definitions 
hyperbolic and normally hyperbolic trapped sets 
and \eqref{eq:pressure},\eqref{eq:Lyap}
for definitions of ``Lyapounov exponenents" and topogical pressure respectively;
smooth, Gevrey and analytic refers to the regularity of the coefficients
of the operator $ P $.)

The first case we will consider is that of {\em hyperbolic trapped sets}:
\begin{defi}
\label{d:hyp}
Suppose that $ dp |_{ p^{-1} ( E ) } \neq 0 $. 
We say that the flow $ \varphi^t$ 
is {\em hyperbolic on $ K_E $}, if
for any $\rho\in K_{E} $,
the tangent space to $ p^{-1} ( E ) $ at $\rho$ splits into flow, unstable and stable 
subspaces:
\begin{equation} \label{eq:aa}
\begin{split}
i)\ & T_\rho (p^{-1} ( E )) = \RR H_p (\rho ) \oplus E^+ _ \rho \oplus 
E^- _ \rho \,, \quad  \dim E^\pm _ \rho  = n-1 \\
ii)\ &  d \varphi^t_\rho ( E^\pm_\rho ) = E^\pm _{ \varphi^t ( \rho ) }\,, 
\quad \forall t\in\RR\\
iii)\ & \exists \; \lambda > 0\,, \text{ and a smooth metric $ \rho 
\mapsto \|\bullet \|_\rho $ such that }
\\ &  \ \| d\varphi^t_\rho ( v ) \|_{\varphi^t ( \rho ) } 
 \leq  e^{-\lambda |t| } \| v \|_{\rho}  \,,
  \ \  \text{for all $ v \in E^\mp_ \rho  $, $ \pm t \geq 0 $.}\\
\end{split}
\end{equation}
\end{defi}
For more on this in dynamical systems see
\cite[\S 17.4]{kahal}. One important property is the stability of this condition.
We also remark that  the existence of the metric 
$ \| \bullet \|_\rho $ in part (iii) or 
\eqref{eq:aa} follows from the same estimates {\em for some} metric 
with $ C e^{- \lambda |t| } $ on the right hand side.
Examples include $ p = |\xi|_g^2 $ where the curvature is negative near
the trapped set or $ p = |\xi|^2 + V ( x ) $ where $ V $ is a potential
given by several bumps (see \cite[Figure 1]{NZ3}). In obstacle scattering
trapped sets are hyperbolic for several convex obstacles satisfying
Ikawa's non-eclipse condition (see \cite[(1.1)]{NSZ}).

A dynamically defined function which plays a crucial role here is
the {\em topological pressure}. We present  a definition valid when
closed trajectories are dense in $ K_E $ -- that is the case in the examples
mentioned above. Let $ \gamma $ denote a closed trajectory of $ \varphi_t $ 
on $ p^{-1} ( E ) $ and $ T_\gamma $ its length. We then define
\begin{equation}
\label{eq:pressure}
\begin{gathered}
\mathcal P_E ( s ) := \lim_{ T \to \infty } {{\frac 1 T }} \log
\sum_{ T \leq T_\gamma \leq T +1} J^+ (\gamma )^{-s}  , \\ 
J^+ ( \gamma ) := \int_0^{T_\gamma } \varphi^+ ( \varphi_t ( \rho_\gamma) ) dt , \ \ \rho_\gamma \in \gamma , \ \ \varphi^+ ( \rho ) := \frac{d}{dt} \log  
\det ( d \varphi_t ( \rho) |_{ E_{\rho_\gamma } }^+ ) |_{t=0} 
\end{gathered}
\end{equation}
and the expression is independent of 
the choice of $ \rho_\gamma \in \gamma $ and of the density 
defining the determinant.  (Strictly speaking this is the pressure
associated to $ - s \varphi^+ $ but as that is the only one we will
consider we just call it the pressure.)
 
This definition 
is not the one that is in fact useful in the study of resonances but
it is the simplest one to state. It can be interpreted as follows:
because of (iii) in \eqref{eq:aa} we have $  J^+ ( \gamma ) > 1 $
(here we choose the determinant using the density coming from
the metric $ \| \bullet \|_\rho $)
and this expression measures the averaged rate of expansion in ustable directions 
$ E_{\rho_\gamma}^+ $ along the closed orbit $ \gamma $. Hence 
$ J^+ ( \gamma )^{-s} $ gets smaller with increasing $ s $. On the
other hand, the count over closed orbits of length $ T_\gamma \sim T $
measures the complexity of the dynamical system: it is more complex if
there are more closed orbits of lengths in $ [ T , T+1 ]$. It follows that
 for $ s > 0 $,
$ \mathcal P_E ( s ) $ measures the fight between complexity and
dispersion in the unstable directions. If $ \mathcal P_E ( s_0 ) = 0 
$ then for $ s > s_0 $ 
 the dispersion wins. For two 
dimensional systems the {\em Bowen pressure formula} shows that
the Hausdorff dimension of the trapped set is given in terms of $ s_0 $:
\begin{equation}
\label{eq:BowP}
\dim K_E = 2 s_0 + 1 ,
\end{equation}
-- see for instance \cite{BowP} and references given there. 
When $ p = |\xi|_g^2 $ on a {\em convex co-compact surface}, 
$ M = \Gamma \backslash \HH^2 $ -- see
Figure \ref{f:borth1} -- then, in that special constant curvature case,
we have for any 
$ E > 0 $, 
\begin{equation}
\label{eq:defdel}  
\begin{gathered} \dim K_E = 2 \delta + 1 , \ \ 
\mathcal P_E ( s ) = 2 E^{-\frac12} ( \delta - s ) , \ \ \delta := \dim \Lambda ( \Gamma ) , \\
\Lambda ( \Gamma ) := \overline {\Gamma z } \cap \partial \HH^2 . 
\end{gathered}
\end{equation}
The set $ \Lambda ( \Gamma )  $ is called the {\em limit set} of $ \Gamma$
and the definition does not depend on the choice of $ z \in \HH^2 $ -- 
see \cite{borth} and references given there.

The analogue of Ikawa's result for several convex obstacles \cite{ik2}, 
postulated independently in the physics literature by Gaspard--Rice \cite{GaRi},
was established  for operators of the form \eqref{eq:defPofh} by 
Nonnenmacher--Zworski \cite{NZ1},\cite{NZ2}.
Thanks to the results of \S \ref{vasy}  the result also holds
for $ -h^2 \Delta_g $ on even asymptotically hyperbolic manifolds. In that
case it generalizes celebrated results of Patterson and Sullivan 
(see \cite{borth}) formulated using dimension of the limit set \eqref{eq:defdel}.
That gap is shown as the top line on the right graphs in Figure \ref{f:dy1}.

\begin{thm}
\label{t:8}
Suppose that an operator \eqref{eq:defPofh} has a hyperbolic trapped 
set at energy $ E $, in the sense of Definition \ref{d:hyp}. If the pressure
defined in \eqref{eq:pressure} satisfies
\begin{equation}
\label{eq:presscond}
0 > \mathcal P_E ( {\textstyle {\frac12} } ) 
\end{equation}
then there are no resonances near $ E $ with $ \Im z > h \mathcal P_E ( \frac12) $.

More precisely, if $ \gamma_E (\delta) := \min_{|E-E'|\leq \delta}( -  \mathcal P_{E'} ( \frac12) ) $ then 
\begin{equation}
\label{eq:t8} \exists \, \delta > 0 \, \, \forall \, \gamma < \gamma_E ( \delta)  \, \, \exists \, h_0 \ 
\ \Res ( P ( h ) ) \cap  \left( [ E - \delta, E + \delta ] - i 
h [0 , \gamma ] \right) = 
\emptyset ,
\end{equation}
when $ h < h_0 $.
In addition for any $ \chi \in \CIc ( M ) $, we have, for some $ c, C  > 0$,
\begin{equation}
\label{eq:t8bis}   \| \chi ( P - z )^{-1} \chi \|_{ L^2 \to L^2 } \leq 
C h^{-1 + c {\Im z} /h } \log {\textstyle{\frac 1h }}\,,
\end{equation}
for $ z \in [ E - \delta, E + \delta ] - i 
h [ 0 ,  \gamma ]$.
\end{thm}

It is not clear if the pressure condition \eqref{eq:presscond} is 
needed to obtain a resonance strip -- see Theorem \ref{t:dylo}
below. What seems to be clear is that $ \mathcal P_E ( \frac12 ) $ 
is a robust classical quantity determining the strip while
the improvements require analysing quantum effects. The upper 
bound \eqref{eq:t8bis} 
is optimal for $ \Im z = 0 $ thanks to a general result of
Bony--Burq--Ramond \cite{BBR},\cite[\S 7.1]{res}.

Applications of the estimate \eqref{eq:t8bis} include local smoothing
estimates with a logarithmic loss by Datchev \cite{Dat} and 
Strichartz estimates with {\em no loss} by Burq--Guillarmou--Hassell \cite{buhai}.
The method of proof was used by Schenck \cite{Schenck} to obtain decay estimates for damped wave equations and by Ingremeau \cite{Ingrem} to describe distorted
plane waves in quantum chaotic scattering. 

For an outline of the proof and a more relevant description of the pressure
function we refer to \cite[\S 2]{NZ1} and to an 
excellent survey by Nonnenmacher \cite[\S 8.1]{Nonn} where a discussion 
of optimality can also be found. Here we make one general comment
relevant also in the case of Theorem \ref{t:9} below. Thanks to the 
``gluing" results of Datchev--Vasy \cite{dv1} one can prove
existence of the pole free region \eqref{eq:t8} and the cut-off resolvent
bound \eqref{eq:t8bis} for a simpler operator modified by a 
{\em complex absorbing potential\,}\footnote{This method is related to 
a method used in chemistry to compute resonances -- 
see \cite{RiMe} and \cite{semi} for the
original approach and \cite{Jag} for some recent developments
and references. A simple mathematical result justifying this computational
method is given in \cite{Vis}: if $ V \in L^\infty_{\comp } ( \RR^n ) $ then
the eigenvalues of $ - \Delta + V - i \epsilon x^2 $ converge to resonances of $ - \Delta + V$ 
uniformly on compact subsets  of
the region $ \arg z > - \pi/4 $.}:
\[  P_W := -h^2 \Delta_g + V ( x) - i W , \ \ W \in \CI ( M ; [ 0 , 1 ] ) , \]
with $ W ( x ) \equiv 1 $ for $ |x| > R $ where $ R $ is chosen large enough.
The operators $ P_W - z : H^2 \to L^2 $ are Fredholm operators for $ \Im z > -1 $
and are invertible for $ \Im z > 0 $. As shown in \cite[\S 8]{NZ3} 
the estimates on $ ( P_W - z )^{-1} $ imply estimates \eqref{eq:t8bis} while
being easier to obtain.

\begin{figure}
\includegraphics[width=15cm]{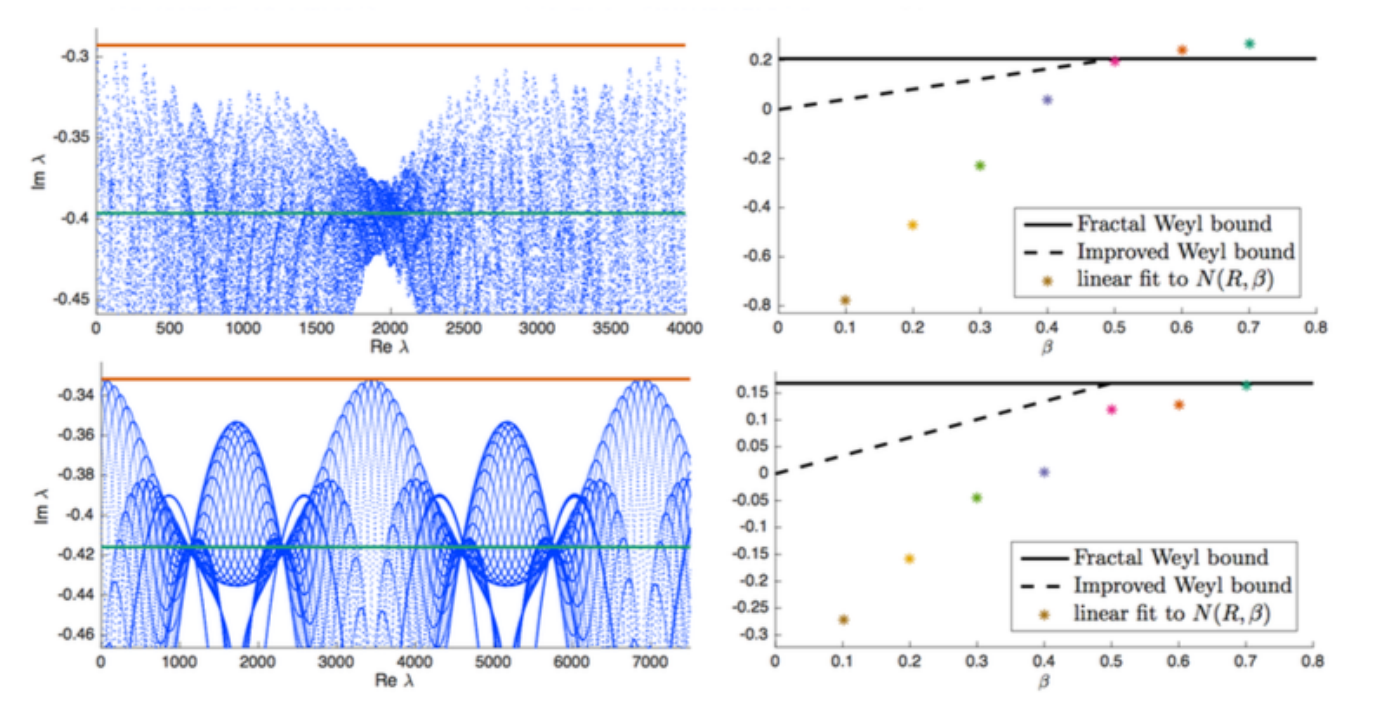}
\caption{Left: resonances for $ X ( 6, 7 , 7 ) $ 
and $ Y ( 7 , 7 , \pi/2 ) $ from 
\cite[Appendix]{dy6} (see Figure \ref{f:borth1} for the definition and
picture of these surfaces; here we show the 
poles of the continuation of $ ( - \Delta_g - \lambda^2 - \frac1 4 )^{-1} $). 
The top line is the pressure gap \eqref{eq:t8} with $ P (\frac12) =  \frac12  - \delta $. The lower line is given by half of the classical decay rate, $ P ( 1 ) /2 = ( 1 - \delta)/2 $ where $ \delta $ is the dimension of the limit sets for the two surfaces -- see
\eqref{eq:defdel}. 
Right: numerical illustration of Theorems \ref{t:fwl} and \ref{t:dy6} with the standard \cite{glz},\cite{fwl} and 
improved \cite{dy6} fractal Weyl laws -- see \S \ref{Weyl} -- for the corresponding surfaces;
here $ N ( R , \beta ) $ is the number of resonances in 
$ | \Re \lambda | \leq R $, $ \Im \lambda > \frac12 + ( \beta - 1 ) \delta $.
The ``chain patterns" visible on the left (which are even more pronounced for more
symmetric systems such as $ X ( \ell, \ell, \ell ) $ -- see \cite{borwe})
have been recently investigated by Weich \cite{W16}.}
\label{f:dy1}
\end{figure}

The trapped sets to which Theorem \ref{t:8} applies are typically 
{\em fractal}. In fact, in dimension 2 the condition $ \mathcal P_E ( \frac12) < 0 $
is equivalent to the trapped set being {\em filamentary} in the sense
that $ \dim K_E < 2 $, that is the trapped set is below the mean of the
maximal dimension $ 3 $ and the minimal dimension $ 1 $ (direction of the flow). 
We will now consider another case in which the trapped set is {\em smooth}
but with hyperbolicity in the transversal direction:

\begin{defi}
\label{d:nhyp}
We say that the Hamiltonian flow is {\em normally hyperbolic} at 
energy $ E $ if for some $ \delta$ and $ J = [ E - \delta, E + \delta ] $,
$ K_J $ is a smooth symplectic submanifold of $ T^*M $, 
and there exists a continuous distribution of
linear subspaces,
$ K_J \ni \rho \longmapsto  E^{\pm}_\rho
\subset T_\rho  ( T^* M ) $,
invariant
under the flow, $ d\varphi_t ( E^\pm_{ \rho } ) = E^\pm_{ \varphi_t ( \rho ) } $,
 satisfying, 
$ \dim E^\pm_\rho = d_{\perp}$  and 
\begin{gather}
\label{eq:NH}
\begin{gathered}
 T_\rho K_J \cap E^\pm_\rho =  E^+_\rho
\cap E^-_\rho = \{ 0 \}, 
\ \ \ 
T_{\rho}(T^{*}X) = T_{\rho}K^{\delta} \oplus  E_{\rho}^{+} \oplus 
E_{\rho}^{-} , \\
\forall v \in E^\pm_\rho,\ \ \forall t>0,\quad \| d\varphi_{\mp t}(\rho) v \|_{\varphi_{\mp t} ( \rho ) }  \leq C
e^{ - \lambda t } \| v \|_\rho \,,
\end{gathered}
\end{gather}
form some smoothly varying norm on $ T_\rho ( T^*M ) $, $ \rho \mapsto  \| \bullet \|_\rho $.
\end{defi}

This dynamical configuration is stable under perturbations under a stronger addition 
condition of {\em $ r$-normal hyperbolicity}. Roughly speaking that means that
the flow on $ K_J $ has weaker expansion and contraction rates than the
flow in the transversal directions -- see \cite{HPS},\cite{dy} for 
precise a definition and \eqref{eq:NHrel} below for an example. Normally hyperbolic trapping occurs in many 
situations: for instance, for null flows for black hole metrics, see 
\cite{physrev},\S \ref{resrel}, and in molecular dynamics, see \cite[Remark 1.1]{NZ3},\cite{SWW06},\cite{GSWW10}.
Another important example comes from contact
Anosov flows lifted to the cotangent bundle, see \cite{Ts},\cite[Theorem 4]{NZ3},
\S \ref{resfreeA}.
The following general result was proved by Nonnenmacher--Zworski \cite{NZ3}:
\begin{thm}
\label{t:9}
Suppose that for $ P $ given by \eqref{eq:defPofh} and that 
at energy $ E $ the trapped set is hyperbolic in the sense of Definition 
\ref{d:nhyp}. Define the minimal expansion rate by 
\begin{equation}
\label{eq:Lyap}
\nu_{\min} :=\liminf_{ t \to \infty } \frac 1 t \inf_{ \rho \in K_E } \log
\det\left(d\varphi_t|_{ E^+_ \rho } \right)\, . 
\end{equation}
Then near $ E $ there are no resonances with $ \Im z > -h \nu_{\min} /2 $.
More precisely 
\begin{equation}
\label{eq:t9}
\exists \, \delta \, \forall \, \epsilon > 0 \, \exists \, h_0 \ \ 
\Res ( P ( h ) ) \cap \left( [ E - \delta, E + \delta ] -
i h {\textstyle {\frac 12 }} [ 0 , \mu_{\min } - \epsilon ]\right) = \emptyset , 
\end{equation}
and the estimate \eqref{eq:t8bis} holds in this resonance free region.
\end{thm}

\begin{figure}
\hspace{-0.6in}\includegraphics[width=10cm]{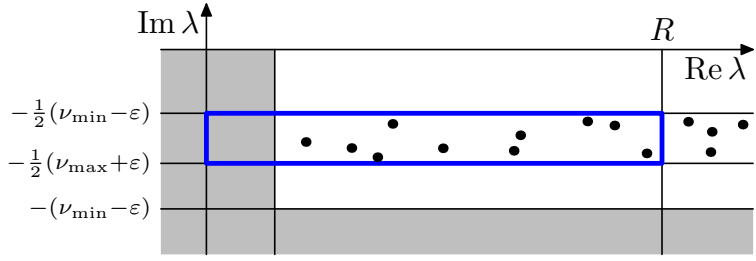}
\caption{A schematic presentation of the results of \cite{dy} for 
an operator $ - \Delta_g $ with an $r$-normally hyperbolic trapped set. Here $ 
R = 1/h  $ and we use the resolvent $ ( - \Delta_g - \lambda^2 )^{-1} $.
The resonance free strip is the same as in Theorem \ref{t:9} and 
the resonances in the strip $ \Im \lambda > - \frac12 (\nu_{\max} + 
\epsilon )  $ satisfy a Weyl law: 
$| \{ \lambda  : \Im \lambda  > - \frac12 ( \nu_{\max} + \epsilon ) , 0<  \Re \lambda < R \}
| \sim R^{n-1} \vol( K_{[0,1]} ) / (2 \pi)^{n-1} $ -- see \cite[Theorem 3]{dy}.} 
\label{f:iv}
\end{figure}

For more general trapped sets but for analytic coefficients of $ P $ and 
without the estimate \eqref{eq:t8bis}, the resonance free region \eqref{eq:t9}
was obtained early on by G\'erard--Sj\"ostrand \cite{GeSj1}. Sometimes, 
for instance in the case of black holes,  $ d^\perp = 1 $ (see \eqref{eq:NH}), $ \Gamma_J^\pm $ (see
\eqref{eq:defGama}) are smooth, and 
$ T_\rho K_J \oplus E_\rho^+ = T_\rho \Gamma^{\pm}_J $. In that 
case a resonance free region $ \Im z > - h /C $
and the bound \eqref{eq:t8bis} was obtained
by Wunsch--Zworski \cite{WuZw2}. An elegant proof with the  
width \eqref{eq:Lyap} and a sharp resolvent estimate was given by 
Dyatlov \cite{dy4}\cite[\S 6.3]{res}. The width given by $ \nu_{\min} $ is essentially optimal
as one can already see from the work of G\'erard--Sj\"ostrand \cite{GeSj} on 
the pseudo-lattice of resonances generated by a system with one 
hyperbolic closed orbit. A general result was provided by Dyatlov \cite{dy}:
if $ K_E $ is $r$-normally hyperbolic for any $ r $ and if
the maximal expansion rate $ \nu_{\max}$ satisfies
$ \nu_{\max} < 2 \nu_{\min} $, then the resonance free region 
is essentially optimal in the sense that in the strip below $ \Im z = 
h \nu_{\min}/2 $ there exist
infinitely many resonances -- see Figure~\ref{f:iv}.

We now discuss improvements over the pressure gap. The first
such improvement for scattering resonances was achieved by Naud \cite{Naud}
who extended {\em Dolgopyat's method} \cite{dole} to the case of 
convex co-compact quotients and showed that there exists $ \gamma > 0 $
such that when $ \delta < \frac12 $ (see \eqref{eq:defdel}) 
then there are no resonances other than $ i ( \delta - \frac12 ) $ 
with $ \Im \lambda > \delta - \frac12 - \gamma $. The Dolgopyat method 
was further developed by Stoyanov \cite{Stoya} for higher dimensional 
quotients (as an application of general results) 
and by Oh--Winter \cite{oh} for uniform gaps for arithmetic quotients. 
Petkov--Stoyanov \cite{PeSt} also adapted Dolgopyat's method
to the case of several convex obstacles. All of these results assume
that $ \mathcal P_E ( \frac12)  \leq 0 $. 

A new method for obtaining improved resonance
free regions, applicable also when $ \mathcal P_E ( \frac12) > 0 $,
was introduced by Dyatlov--Zahl \cite{DyZa}. It is based on a 
{\em fractal uncertainty principle }(FUP) which in \cite{DyZa} was
combined with an investigation of additive structure of limit sets
(see \eqref{eq:defdel}) to obtain an improved gap for quotients
$ \Gamma \backslash \HH^2 $, 
with $ \delta ( \Gamma ) \approx \frac12 $. In particular that
produced the first  resonance free strips when $ \mathcal P_E ( \frac12) > 0 $.

Roughly speaking
the standard uncertainty principle says that wave functions cannot be strongly localized in both position and frequency near a point.
The fractal uncertainty principle (FUP) states that a wave function cannot be strongly localized in both position and frequency near a fractal set.
To describe FUP rigorously we will use
a simpler model studied recently by Dyatlov and Jin \cite{dylo}:
open quantum maps with products of exact Cantor sets as trapped sets.

\begin{figure}
\includegraphics[scale=0.8]{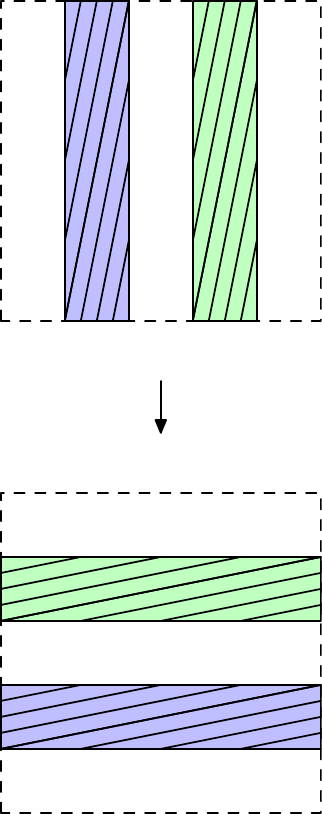}\quad\quad\quad \quad \quad
\includegraphics[scale=0.55]{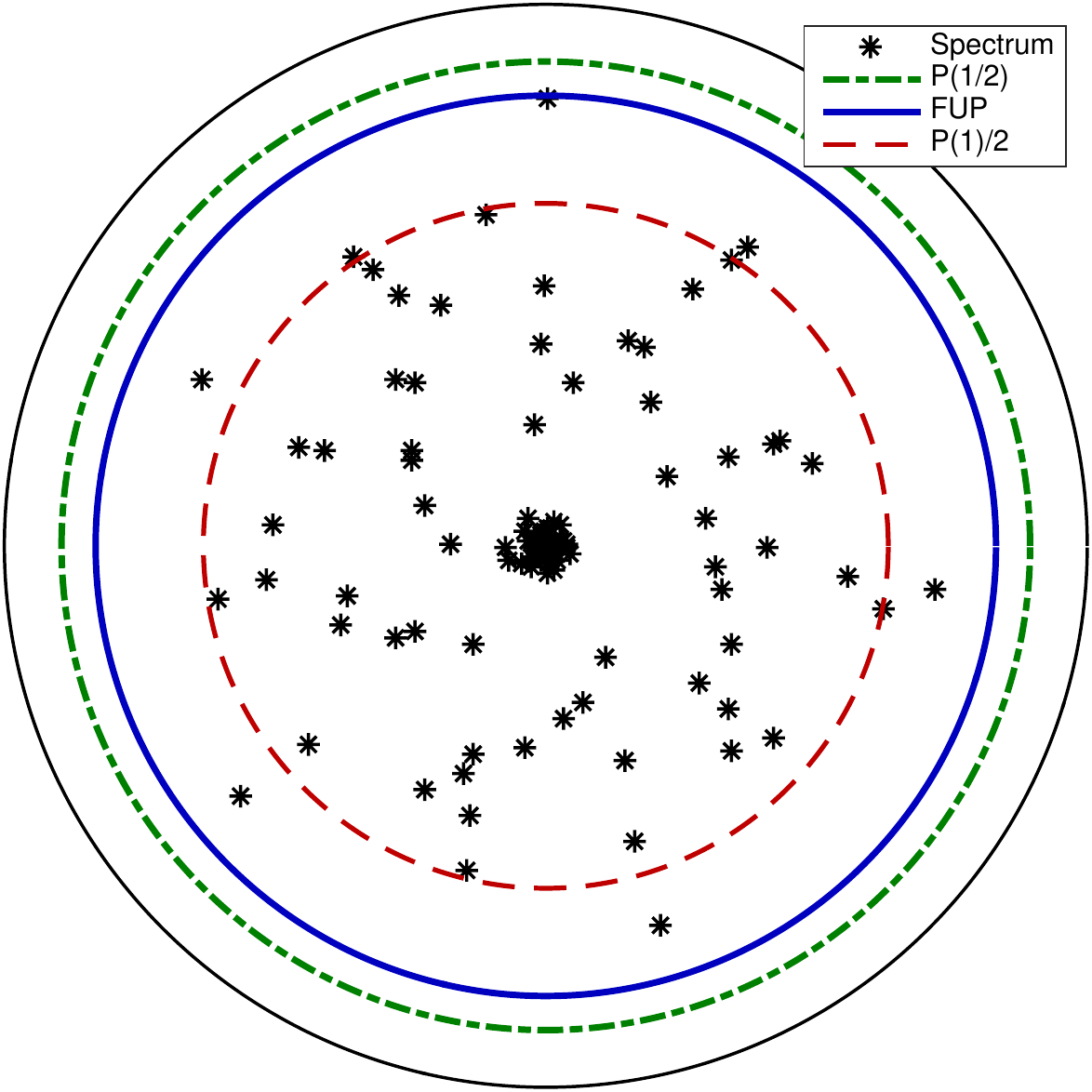}
\caption{Left: a schematic representation of an open
baker map (a piecewise smooth symplectic relation) $ \kappa_{ M, \mathcal A} $:
$ \RR^2/\ZZ^2 \ni ( y , \eta ) \mapsto ( My-a , (\eta + a )/M )
$,  $ ( y , \eta ) \in ( a /M, ( a+1)/M) $, $ a \in \mathcal A
\subset \{ 0 , \ldots , M- 1\}$, here with $ M = 5 $, $ \mathcal A 
= \{ 1,3 \}. $  A quantization, $ B_{5N} $, of this open map (relation) is given
in \eqref{eq:oqm}. 
Right: resonances of  $ B_{5N } $ with $ N = 5^4 $. 
Different theoretical gaps are also shown: for this quantum map
the fractal uncertainty principle (FUP) of \cite{dylo},\cite{DyZa} does produce
an optimal gap.}
\label{f:dylo}
\end{figure}

Open quantum maps have been studied in physics and mathematics.
They are quantizations of symplectic relations on $ \TT^2 = \RR^2 / \ZZ^2 $
where the symplectic relations have features of Poincar\'e maps 
in scattering theory: see 
Nonnenmacher \cite[\S 5]{Nonn} for a general introduction,
\cite[\S 1.4]{dylo} for references to the physics literature, 
Nonnenmacher--Sj\"ostrand--Zworski \cite{NSZ1} 
for a reduction of chaotic scattering problems to quantizations of Poincar\'e
maps and Nonnenmacher--Zworski \cite{NZ0} for quantization of piecewise
smooth relations.

An example of a piecewise smooth symplectic relation is shown on the left in  
Figure~\ref{f:dylo}: the torus is divided into five ``vertical strips", three of them
are thrown out and the remaining two are stretched/contracted and then shifted
(the caption to Figure \ref{f:dylo} has a general formula for $ \kappa_{M, \mathcal A}$.) This relation, $ \kappa_{ 5, \{ 1,3 \}} $, is quantized by the following
family of matrices:
\begin{equation}
\label{eq:oqm}
\begin{gathered}
\hspace{-0.5in} B_{5N} := \mathcal F_{5N} ^*  \left( \begin{array}{lllll}
  0 & \ \ \ \ 0 & 0 & \ \ \ \ 0 & 0  \\ 
 0 & \chi_N \mathcal F_{N}\chi_N   &  0 & \ \ \ \ 0 & 0  \\
  0 & \ \ \ \ 0 & 0 & \ \ \ \ 0 & 0  \\
 0 & \  \ \ \ 0  & 0 & \chi_N \mathcal F_{N} \chi_N & 0 \\
   0 & \ \ \ \ 0 & 0 & \ \ \ \ 0 & 0 
\end{array} \right) 
\end{gathered}
\end{equation}
where $ \mathcal F_L $ is the unitary Fourier transform on $\ell^2_L := \ell^2 
( \ZZ/L \ZZ ) $, 
$ \chi_N := {\rm{diag}} \,( \chi ( j/N ) )$ for $ \chi \in \CIc ( ( 0 , 1 ); [ 0,1 ])$. 

On the classical level a relation $ \kappa $ defines discrete time analogues of \eqref{eq:defGama}
and \eqref{eq:defKE}:
\[  \Gamma^\pm := \bigcap_{ \pm r \geq 0 } \kappa^r ( \TT^2 ) , \ \ 
K = \Gamma^+ \cap \Gamma^- . \]
In the case of $ \kappa_{ M, \mathcal A} $ we have 
\begin{equation}
\label{eq:Cantor}
\begin{gathered}
\Gamma_- = \mathcal C \times \SP^1 , \ \ \Gamma_+ = 
\SP^1 \times \mathcal C , \ \ K = \mathcal C
\times \mathcal C , \ \
\dim K = 2 \delta, \\
\mathcal C := \bigcap_k \bigcup_{ j \in \mathcal C_k } \left[
\frac{ j }{M^k} , \frac{j+1}{M^k }\right], \ \ 
\mathcal C_k := \left\{ \sum_{ j=0}^{k-1} a_j M^j : a_j \in \mathcal A \right\} , \\
\delta := \dim \mathcal C = \frac{ \log | \mathcal A |}{\log M } .
\end{gathered}
\end{equation}

If we think of $ \TT^2 $ as being an analogue of a hypersurface $ \Sigma $ 
transversal to the flow then $ K $ is the analogue $ K_E \cap \Sigma $.
The linearization of $ \kappa_{ M , \mathcal A} $ (at regular points)
has eigenvalues $ M, M^{-1}$ and hence we can think of $ \kappa $ as the flow
at time $ t = \log M $ with the expansion rate $ \lambda = 1 $ (see (iii)
in \eqref{eq:aa}). The pressure of the suspension at time $ \log M$ 
is given by 
\begin{equation}
\label{eq:pressureK}
\mathcal P (s) = \delta - s ,
\end{equation}
see \cite[\S 8]{Nonn} for a discussion of pressure for suspensions.

The operator $ B_N $ is considered as quantization of the propagator
at time $ t = \log M $ with $ h = 2\pi /N $. Hence the 
conceptual correspondence between eigenvalues
of $ B_N $ and resonances of semiclassical operators $ P ( h ) $
(established in some cases in \cite{NSZ1}) is 
\[   \Spec ( B_N ) \ni \lambda = e^{ - i t z / h } \ \longleftrightarrow \  
z \in \Res ( P ( h ) ) . \]
Since $ t = \log M $  the correspondence between 
resonances free regions is given by the following 
\begin{equation}
\label{eq:corres}
( [ 1 , 2 ] - i h [ 0 , \gamma ] ) \cap \Res ( P ( h ) ) 
= \emptyset  \ \longleftrightarrow \ \Spec ( B_N ) \cap \{ |\lambda | \geq 
M^{-\gamma } \} = \emptyset .
\end{equation}
In view of this and \eqref{eq:pressureK} the analogue of \eqref{eq:t8} is
\begin{equation}
\label{eq:pressure1}
\forall \, \epsilon > 0   \ \ \  \Spec ( B_N ) \cap \{ |\lambda | \geq 
M^{\delta-\frac12 + \epsilon } \} = \emptyset , \ \ \ N \geq N_\epsilon ,
\end{equation}
see \cite[\S 8]{Nonn}.

The general principle of \cite{DyZa} and \cite{dylo} for quantization of
$ \kappa_{M, \mathcal A} $ 
goes as follows. 
In the notation of \eqref{eq:oqm} and \eqref{eq:Cantor}  we say that 
\begin{equation}
\label{eq:FUP1}
\text{ FUP holds with exponent $ \gamma $ }  \Longleftrightarrow  
\| \indic_{\mathcal C_k } \mathcal F_N \indic_{\mathcal C_k } \|_{ 
\ell^2_N \to \ell^2_N } \leq  C N^{-\gamma} .
\end{equation}
Then we have (see \cite[Proposition 2.7]{dylo} for this case and 
\cite[Theorem 3]{DyZa} for the case of convex co-compact quotients)
\begin{equation*}
\text{ FUP holds with exponent $ \gamma $ } \ \Longrightarrow \
\forall \epsilon > 0 \, \ \Spec ( B_N ) \cap \{ |\lambda | \geq 
M^{-\gamma + \epsilon } \} = \emptyset ,
\end{equation*}
for $ N > N_\epsilon $,

The case of $ \gamma = \frac12 - \delta $ is easy to establish but a
finer analysis of FUP gives \cite{dylo}:
\begin{thm}
\label{t:dylo}
Suppose that $ B_N $ is a quantization of $ \kappa_{M, \mathcal A } $
and $ \delta $ is given in \eqref{eq:Cantor}.
Then there exists
\begin{equation} 
\label{eq:FUP}
\beta > \max ( 0 , {\textstyle{\frac12}} - \delta ) 
\end{equation}
such that
\begin{equation}
\label{eq:improv}
\forall \, \epsilon > 0 \, \ \ \Spec ( B_N ) \cap \{ |\lambda | \geq 
M^{- \beta + \epsilon } \} = \emptyset , \ \ N > N_\epsilon .
\end{equation}
\end{thm}

A numerical illustration of Theorem \ref{t:dylo} is shown in 
Figure~\ref{f:dylo}. In the example presented there the numerically 
computed best exponent in FUP \eqref{eq:FUP1} is sharp.
Some other examples in \cite{dylo} show however
that it is not always the case. We will return to open quantum maps
in \S \ref{Weyl}.

Theorem \ref{t:dylo} is the strongest gap result for hyperbolic trapped sets and
it suggests that a resonance free strip of size $ h$ exists
for all operators $ P ( h)$ with such trapped sets:

\begin{conj}
\label{c:2}\footnote{After this survey first appeared Conjecture \ref{c:2} was proved in the case of finitely generated hyperbolic surfaces by 
Bourgain--Dyatlov \cite{Body}. That means that for convex co-compact surfaces,
$ \Gamma \setminus \HH^2 $ -- see Figure~\ref{f:borth1} and \cite{borth} -- 
there is a high energy gap for all values, $0 \leq \delta < 1 $, of the dimension of the limit set of $ \Gamma $ -- see \eqref{eq:defdel}. This is the first result  about gaps for quantum Hamiltonians for any value of the pressure $ \mathcal P_E ( \frac12) $ defined in \eqref{eq:pressure}. The proof is based on the 
fractal uncertainty principle of \cite{DyZa} and fine harmonic analysis estimates related to the Beurling--Malliavin multiplier theorem, see \cite{Nazad} and references given there.} Suppose that $ P ( h ) $ given by \eqref{eq:defPofh} has a compact hyperbolic 
trapped set $ K_E $ in the sense of Definition \ref{d:hyp} (and 
$ M$ is a manifold for which we have an effective meromorphic continuation of
$ (P ( h ) -z)^{-1} $). Then there exists
$ \gamma > 0 $, $ \delta > 0 $ and $ h_0 $ such that 
\[    \Res ( P ( h ) ) \cap 
\left( [ E - \delta, E + \delta ] - i h [ 0  , \gamma ] \right) = \emptyset,  \ \
h < h_0 .
\]
\end{conj}

\subsection{Resonance expansions in general relativity}
\label{resrel}

In this section we will describe some recent results concerning
expansions of solutions to wave equations for black hole metrics
in terms of {\em quasi-normal modes} (QNM). That is the name
by which scattering resonances
go in general relativity \cite{KokkotasSchmidt}. We want to 
emphasize the connection to the results of \S \ref{vasy},\S \ref{resfree} and
in particular to the $r$--normally hyperbolic dynamics \cite{WuZw2}.

We will consider the case of Kerr--de Sitter, or
asymptotically Kerr--de Sitter metrics. 
These model rotating black holes in the case of positive 
cosmological constant $ \Lambda > 0 $. From the mathematical 
point of view that makes infinity ``larger" and provides
exponential decay of waves which makes a rigorous formulation
of expansions easier. When one adds frequency localization 
weaker expansions are still possible in the Kerr case
-- see \cite[Theorem 2]{dya},\cite[(13)]{physrev}. References
to the extensive mathematics literature in the case of $ \Lambda = 0 $
can be found in \cite{dya}.

To define the Kerr--de Sitter metric we consider the manifold 
$ X = ( r_+ , r_C ) \times \mathbb S^2 , $ where $ r_+ $ is 
interpreted as the event horizon of the black hole and $ r_C $ the
cosmological horizon.  
The wave equation is formulated on $ \mathbb R 
\times X $ using a Hamiltonian defined on 
$T^*(\mathbb R\times X)$.
We denote the coordinates on $ \mathbb R \times X $ by $ ( t,r, \theta, \varphi ) $  and write $ (\xi_t, \xi_r , \xi_\theta, \xi_\varphi )
$ for the corresponding conjugate (momentum) variables. In this 
notation the Kerr--de Sitter Hamiltonian has the following frightening
form:
\begin{gather}
\label{eq:Gr}
\begin{gathered}
G=\rho^{-2}(G_r+G_\theta), \ \ \rho^2=r^2+a^2\cos^2\theta, \\
G_r=\Delta_r\xi_r^2-{(1+\alpha)^2\over\Delta_r}((r^2+a^2)\xi_t+a\xi_\varphi)^2,\\
G_\theta=\Delta_\theta\xi_\theta^2+{(1+\alpha)^2\over\Delta_\theta\sin^2\theta}(a\sin^2\theta\,\xi_t+\xi_\varphi)^2,
\ \ 
\alpha={\Lambda a^2\over 3}.
\\
\Delta_r=(r^2+a^2)\Big(1-{\Lambda r^2\over 3}\Big)-2Mr,\quad
\Delta_\theta=1+\alpha\cos^2\theta.
\end{gathered}
\end{gather}
Here $ \Lambda > 0 $ is the cosmological constant, $ a $ is the 
rotation speed ($ a = 0 $ corresponds to the simpler case of the
Schwarzschild--de Sitter metric), and $ M > 0 $ is a parameter
(in the case of $ \Lambda = 0 $ it is the mass of the black hole) -- see
Figure \ref{f:dyz2} for the allowed range of parameters. 
The Hamiltonian $G$ is the  dual metric to the semi-Riemannian Kerr--de
Sitter metric $g$ .
The principal symbol of the wave operator 
 $\Box_g$ is given by $ G $. We refer to Hintz \cite{hi} and
 Hintz--Vasy \cite{HiV2} for a geometric description of 
 these metrics and definitions of asymptotically Kerr--de Sitter metrics.

\begin{figure}
\includegraphics[width=7cm]{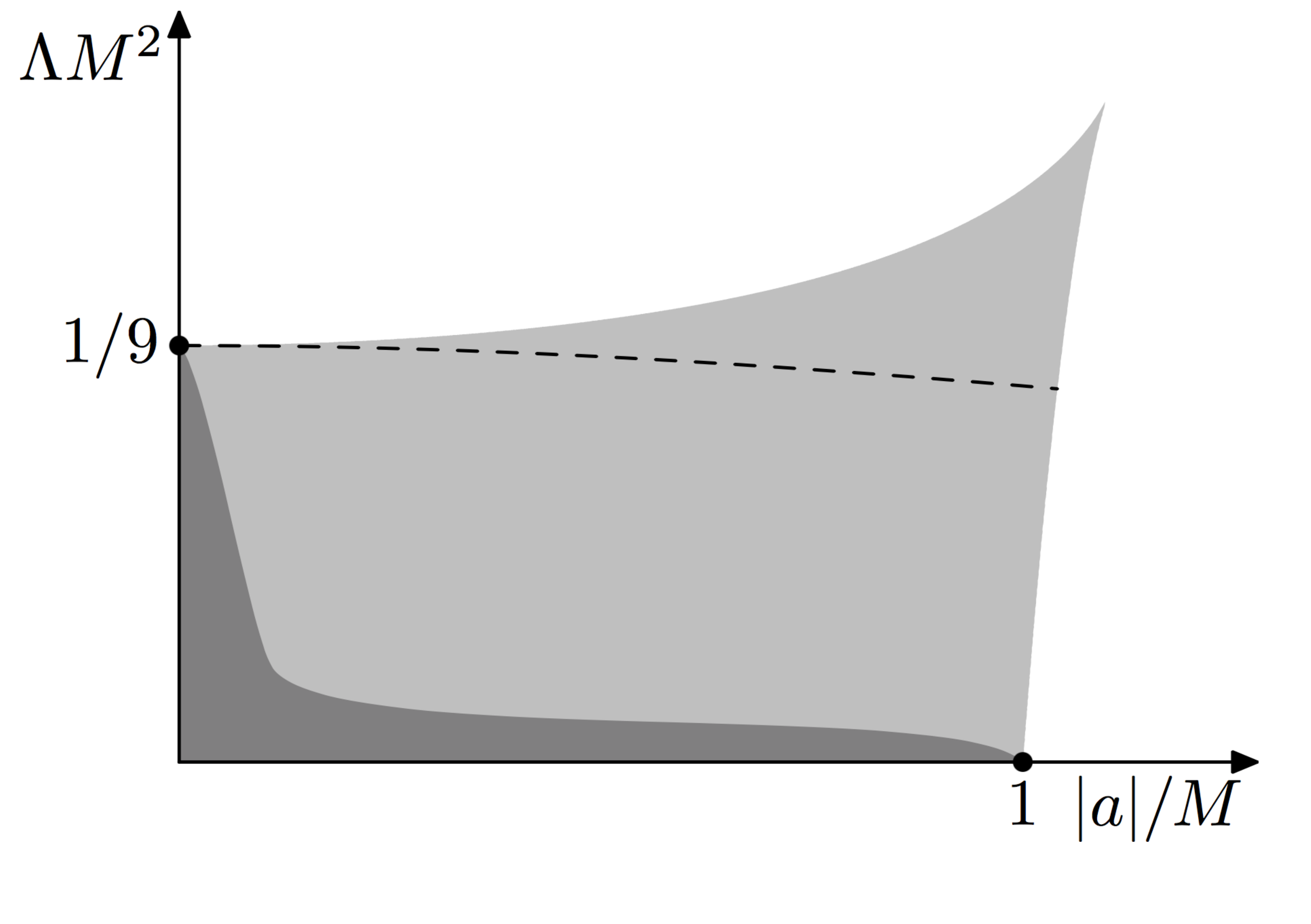}
\includegraphics[width=7cm]{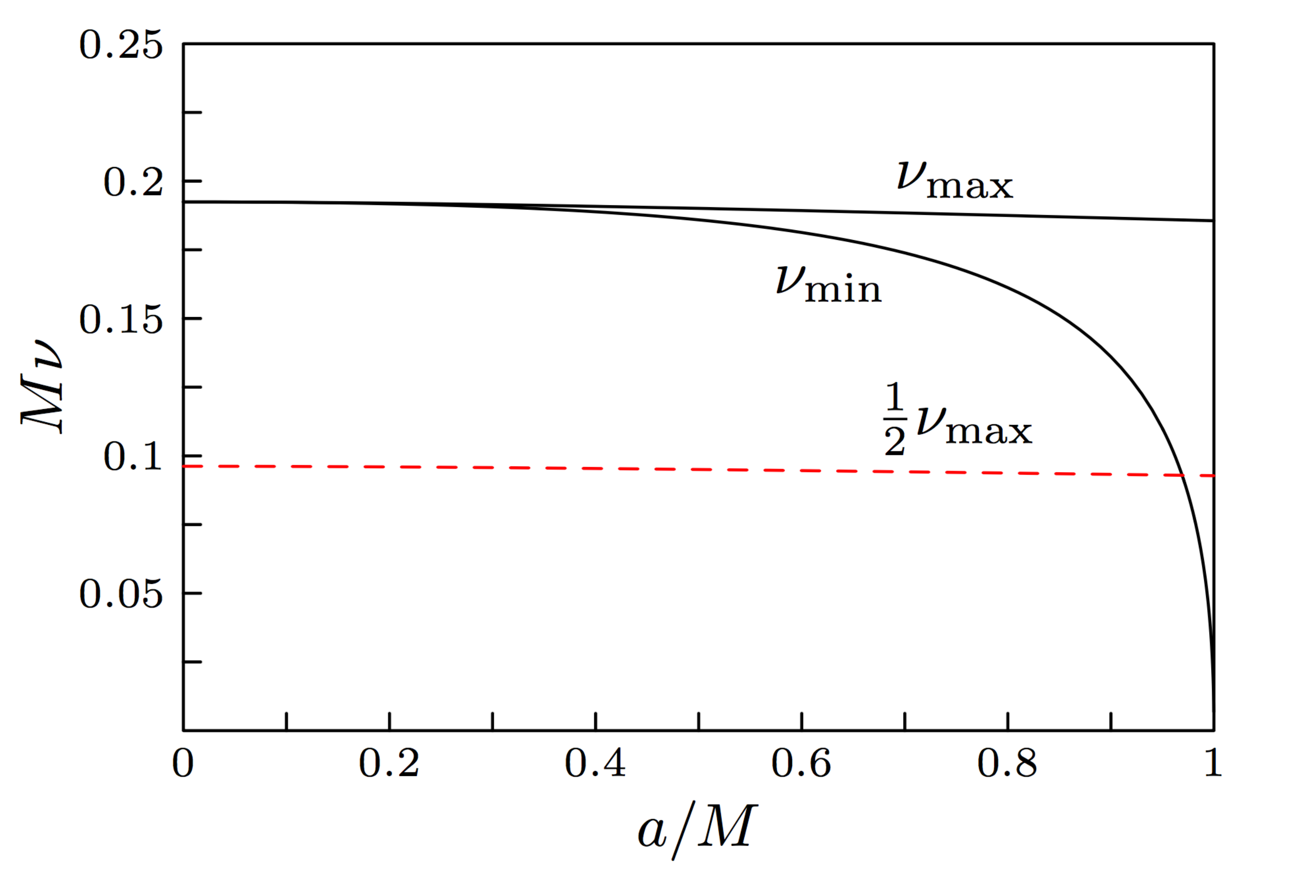}
\caption{Left: numerically computed admissible range of
parameters for the subextremal Kerr--de Sitter black hole
(light shaded) and the range to which the 
resonance expansions of \cite{dya} apply. QNM are defined and
discrete for parameters below the dashed line,
$ ( 1 - \Lambda a^2/3 )^3 = 9 \Lambda M^2$, see \cite[\S 3.2]{dya}.
Right: The dependence of $ \nu_{\max} $ and $ \nu_{\min} $ on the
  parameters $ M $ and $ a $ in the case of $ \Lambda = 0 $.  The
  dashed line indicates the range of validity of the pinching condition needed 
  for the Weyl law \cite{dy} recalled in Figure \ref{f:iv}.}
\label{f:dyz2}
\end{figure}

The key fact is that $ \Delta_r $ has simple zeros at $ r = r_+ $ and $ r = r_C $
which gives 
\begin{equation}
\label{eq:Plag1} P_g ( \lambda ) := e^{ - i \lambda t } \Box_g e^{ i \lambda t } 
: \CI ( X ) \to \CI ( X) , \ \ X =  ( r_+ , r_C ) \times \mathbb S^2 , 
\end{equation} 
structure similar to that of $ P ( \lambda) $
in \eqref{eq:Plag}, with $ r=r_+ $ or $ r= r_C $ corresponding to $ x_1 = 0 $
(when $ a= 0 $, $ r \in ( r_+, r_C ) $ corresponds to $ x_1 > 0 $; 
for rotating black holes the situation is more complicated but microlocally
the structure is similar).
  
Since we are interested in solving $ \Box_g w = 0 $ we consider the null
geodesic flow 
$ \varphi_t := \exp t H_G $ (see \eqref{eq:Hamp}) on the positive 
light cone $ \mathcal C_+ \subset \{ G = 0 \} $. (Null stands for $ G = 0 $
which is preserved by the flow.)
The trapped set, $ \mathcal K $, consists of null geodesics that stay away from $r=r_+,r=r_C$ for all times. We refer to Dyatlov \cite[Prop.3.2]{dya} for the description of this
trapped set and the constraints on the parameters $ a$, $M $ and $ \Lambda $.
Here we remark that it has a particularly simple form when 
$ a = 0$: 
\[ \mathcal K = \left\{ ( t , 3M, \omega ; \xi_t , 0 , \omega^* ) : 
( \omega, \omega^* )  \in T^* \SP^2 , \ \  
\frac{ 1 - 9 \Lambda M^2 }{ 27 M^2 } 
| \omega^* |_\omega^2 = \xi_t^2 \right\}, \]
where $ | \bullet |_\omega $ is the standard metric on $ T^*_\omega \SP^2 $.

In general $ \mathcal K $  
is symplectic
in the sense that the  symplectic form on $ T^* X $ 
is nondegenerate on the surfaces $K = \mathcal K\cap\{t=\const\}$. 
The flow is $ r$-normally hyperbolic for any $ r $. That means that
there exists a
$C^r$ splitting
$ T_\mathcal K \mathcal C_+=T \mathcal K\oplus \mathcal V_+\oplus\mathcal V_-$,
invariant under the flow and such that for some constants $\nu>0,C>0$,
\begin{gather}
\label{eq:NHrel}
\begin{gathered}
\sup_{(x,\xi)\in K}
|d\varphi^{\mp t}
|_{\mathcal V_\pm}|\leq Ce^{-\nu t},\ 
\  \
\sup_{(x,\xi)\in K}
|d\varphi^{\pm t}
|_{TK}|\leq Ce^{\nu|t|/r},\ 
t\geq 0. 
\end{gathered}
\end{gather}
In other words, the maximal expansion rates (Lyapunov exponents) on
the trapped set are $r$-fold dominated by the expansion and
contraction rates in the directions transversal to the trapped set. 
As shown in Hirsch--Pugh--Schub \cite{HPS} (see also \cite[\S 5.2]{dya})  $r$-normal
hyperbolicity is stable under perturbations: when $  G_\epsilon $ is a
time independent (that is, stationary) Hamiltonian such that $ G_\epsilon  $ is close to $ G $
in $ C^r $ near $ \mathcal K $, the flow for $ G_\epsilon $ is $
r$-normally  hyperbolic in the sense that the trapped set $ \mathcal K_\epsilon
$ has $ C^r $ regularity, is symplectic and \eqref{eq:NHrel} holds.
For Kerr(--de Sitter) metrics the flow is $r$-normally 
hyperbolic for all $ r $ as shown by Wunsch--Zworski \cite{WuZw2},
Vasy \cite{vasy1} and Dyatlov \cite{dya}, essentially 
because the flow on $\mathcal K$ is completely integrable. The dependence
of maximal and minimal expansion rates on the parameters is shown in 
Figure \ref{f:dyz2}.

Vasy's method \cite{vasy1} which was described in \S \ref{vasy} 
applies in this setting. Quasi-normal modes (QNM) for 
Kerr--de Sitter metrics are defined as in 
\eqref{eq:charker}: they are complex frequencies $ \lambda $ such that
there exist $ u = u ( r , \theta, \phi ) $ such that 
\begin{equation}
\label{eq:charrr}  \text{ $ \Box_g ( e^{ - i \lambda t } u ) = 0 $
and $ u $ continues smoothly across the two horizons.} 
\end{equation}
The connection 
with \eqref{eq:charker} comes from \eqref{eq:Plag1} and the fact that
 $ P_g ( \lambda ) u = 0$. The 
same definition works for perturbations. 

For Schwarzschild--de Sitter black holes, QNM were described by 
S\'a Barreto--Zworski \cite{SaBZb}: at high frequencies they 
lie on a pseudo-lattice as in the case of one hyperbolic closed
orbit \cite{GeSj} but with multiplicities coming from the spherical 
degeneracy. Dyatlov \cite{zeeman},\cite{dy} went much further
by describing QNM for Kerr--de Sitter black holes: for small
values of rotations he showed a Zeeman-like splitting of 
multiplicities predicted in the physics literature, and for
perturbations of Kerr--de Sitter black holes, he obtained
a counting law (see Figures~\ref{f:iv},\ref{f:dyz2}).
In particular, for small values of $ a $ the results of \cite{zeeman}
show that there are no modes with $ \Im \lambda \geq 0 $.
For Kerr black holes (\eqref{eq:Gr} with $ \Lambda = 0 $) Shlapentokh-Rothman \cite{shlapa} showed that this is the case in the full range $ |a| < M $.
This suggests the following
\begin{conj}
\label{c:rel}
For $ a $, $ \Lambda $ and $ M $ satisfying 
$ ( 1 - \Lambda a^2/3 )^3 >  9 \Lambda M^2 $ and in the shaded region 
of Figure \ref{f:dyz2} (or in some other ``large" range of values),
any $ \lambda \neq 0 $ for which \eqref{eq:charrr} holds satisfies
$ \Im \lambda < 0 $ and $ \lambda = 0 $ is a simple resonance.
\end{conj}

The study of QNM is motivated by the expectation
that forward solutions to $ \Box_g u = f \in \CIc ( \RR \times X ) $ have 
expansions in terms of  $ u_k $'s and $ \lambda_k$'s satisfying \eqref{eq:charrr},
\[  u ( t , r , \theta, \varphi ) \sim \sum_{ k } e^{ - i \lambda_k t } 
u_k ( r, \theta, \varphi ) , \]
but the convergence and errors can be subtle.

The first expansion involving infinitely many QNM lying on 
horizontal strings of the pseudo-lattice of \cite{SaBZb} was 
obtained by Bony--H\"afner \cite{boha}. Using his precise results on 
the distribution of QNM for Kerr--de Sitter metrics with 
small values of  $ a $, Dyatlov \cite{zeeman} obtained similar
expansions for rotating black holes. In \cite{dy} he formulated
an expansion of waves for perturbations of Kerr--de Sitter metrics
in terms of a microlocal projector, see also \cite{dya}.

Here we will only state a simpler result which is an almost
immediate consequence of Theorem \ref{t:7}, Theorem \ref{t:9} and 
of gluing results of Datchev--Vasy \cite{dv1}. For 
tensor-valued wave equations on perturbations of Schwarzschild--de Sitter spaces (including Kerr--de Sitter spaces with small values of $ a $) 
and in any space-time dimension $ n \geq 4 $  this result was obtained 
by Hintz \cite{hi}. A more precise formulation valid across the 
horizons can be found there. 

\begin{thm}
\label{t:rel}
Suppose that $ ( X, g ) $ is a stationary perturbation of the 
Kerr--de Sitter metric in the sense of \cite{hi},\cite{HiV2} and
that $ u $ is the forward solution  $ \Box_g u = f \in \CIc ( \RR \times X ) $.
Let $ \nu_{\min } $ be the minimal expansion rate at the trapped
set given by \eqref{eq:Lyap} (see Figure \ref{f:dyz2}). Then for 
any $ \epsilon > 0 $, 
\[  u ( t , x ) = \sum_{ \Im \lambda_j > - \frac12 \nu_{\min} + \epsilon  } \,\,
\sum_{m=0}^{m_j} \sum_{ \ell = 0 }^{d_j} t^\ell e^{ - i \lambda_j t } 
u_{jm\ell} ( x )  + \mathcal O_K  ( e^{ - ( \nu_{\min}/2 - \epsilon) t }) , \ \
x \in K \Subset X . \]
When $ d_j = 0 $, $ u_{j,m,0} $ and $ \lambda_j $ satisfy \eqref{eq:charrr};
otherwise the resonance expansion is computed from the residue of
$ \lambda \mapsto P_g ( \lambda )^{-1}  e^{ - i \lambda t } $ -- 
see \cite[\S 5.1.1]{HiV2} (and \cite[Theorem 2.7]{res} for a simple example).
\end{thm}

\subsection{Upper bounds on the number of resonances: fractal Weyl laws}
\label{Weyl}
The standard Weyl law for the density of quantum states was 
already recalled in \eqref{eq:Weyll} as motivation for the 
upper bound on the number of resonances in discs \eqref{eq:NVb}.
In this section we will describe finer upper bounds which 
take into account the geometry of the trapped set. 
They were first proved  by Sj\"ostrand \cite{SjDuke} for 
operators with analytic coefficients but in greater geometric generality
than presented here.

Throughout this section we 
assume that 
\begin{equation}
\label{eq:sectionass}
\text{ $ P ( h ) $ given by \eqref{eq:defPofh} has {\em a compact hyperbolic 
trapped set $ K_E $} }
\end{equation}
in the sense of Definition \ref{d:hyp} and that
$ M$ is a manifold for which we have an effective meromorphic continuation of
$ ( P ( h ) -z)^{-1}  $ (see \S\S\ref{rae},\ref{vasy}).

When $ M $ is a {\em compact} manifold then $ P ( h )$ has discrete
spectrum and the counting laws have a long tradition -- see Ivrii \cite{Ivre}.
The basic principle can be described as follows: if we are interested in 
\[ N_{P(h)} ( [ a , b ] ) := | \Spec ( P ( h ) ) \cap [ a, b ] |, \ \
P ( h ) = -h^2 \Delta_g + V ,  \]
then the relevant region in phase space is given by 
\begin{equation}
\label{eq:phaseab} 
p^{-1} ( [ a , b ] ) =  \{ ( x, \xi ) :  a \leq p ( x, \xi ) \leq b \} , \ \ 
p ( x, \xi ) := |\xi|_g^2 + V ( x ) . 
\end{equation}
The uncertainty principle says that a maximal (homogeneous) localization of a state
in phase space is to a box of sides of length $ \sqrt h $ (here centered at
$ ( 0 , 0 )$):
\[  \left( \int_{\RR^n} | x_j  u ( x ) |^2 dx \right)^{\frac12}
\times \left( \int_{\RR^n } | \xi_j  \mathcal F_h u ( \xi ) |^2 d\xi \right)^{\frac12}
\geq \frac{ h}{ 2}  \| u \|^2_{L^2} , \]
where $ \mathcal F_h u ( \xi ) := ( 2 \pi h)^{-n/2} \int_{\RR^n }
u ( x ) e^{ - i \langle x , \xi \rangle/h } dx $ is the (unitary) semiclassical
Fourier transform. Hence the maximal number of quantum states ``fitting" into
an $h$-independent open set $ U \subset T^* M $  is proportional to $  
\vol( U  ) h^{-n} $. For eigenvalue counting we need a factor of
$ (2 \pi )^{-n} $ which gives
\begin{equation}
\label{eq:NhW}   N_{P(h)} ( [ a , b ] ) = ( 2 \pi h)^{-n} \left( \vol ( p^{-1} ( [ a, b ] )) + 
o ( 1 ) \right) .
\end{equation}
By rescaling and taking $ [ a, b ] = [ 0 , 1 ] $ we obtain \eqref{eq:Weyll}.

\begin{figure}[ht]
\begin{center}
\includegraphics[width=11.5cm]{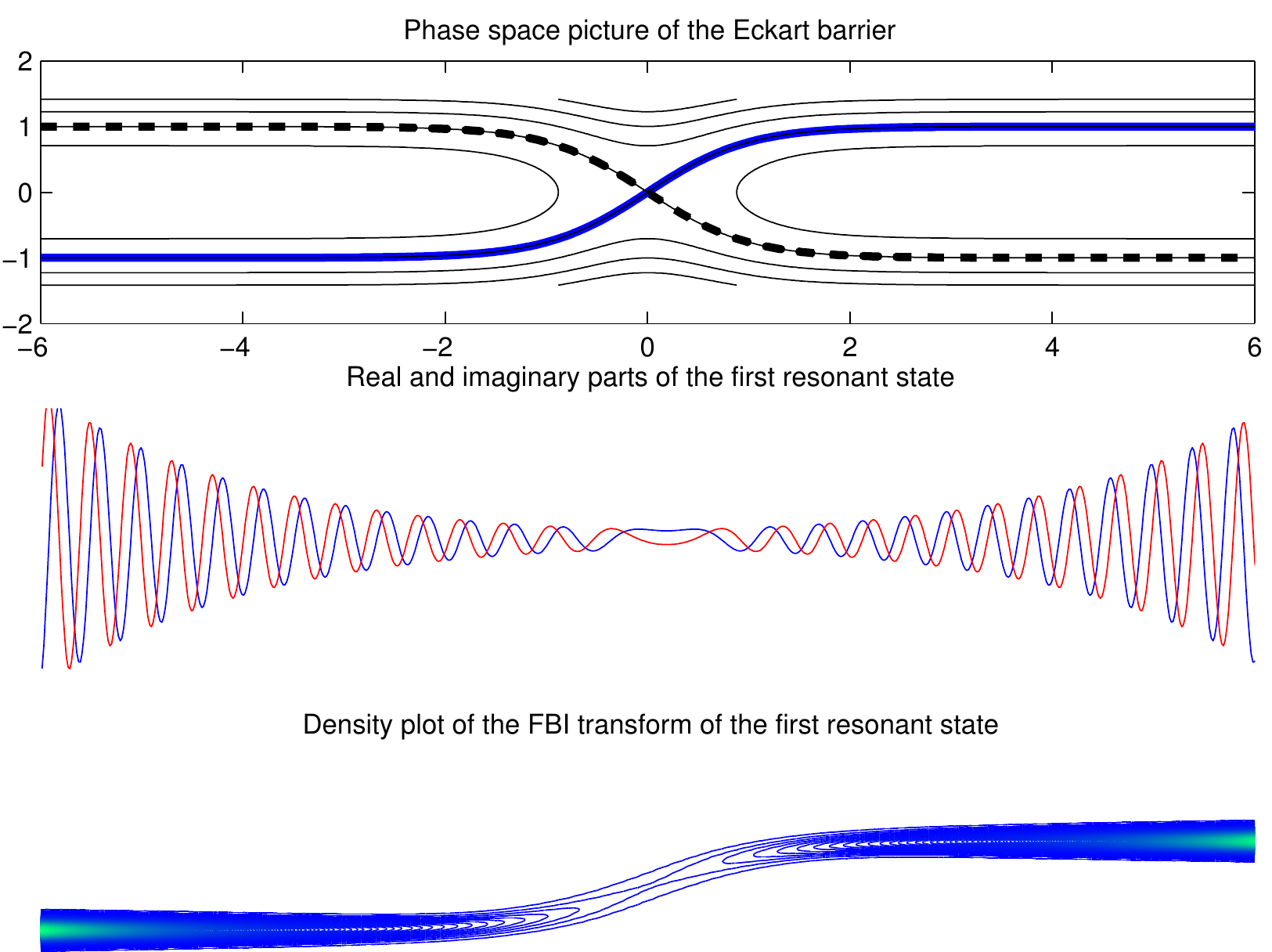}
\end{center}
\caption{The top figure shows the phase portrait for the Hamiltonian $ p ( x , \xi \
) =                                                                                 
\xi^2 + \cosh^{-2} ( x ) $, with $ \Gamma_1^\pm $ highlighted. The middle
plot shows the resonant state corresponding to the resonance
closest to the real axis at $ h = 1/16 $ (compare this to the
resonances in Figure~\ref{f:1} generated by the local maxima), 
and the bottom plot shows
 the squared modulus of its FBI tranform (which describes a wave function 
 in phase space \cite[Chapter 13]{EZ}). 
The resonance states were computed by D.~Bindel \cite{BiZ}
 and the FBI transform
was provided by L.~Demanet. We see that
the mass of the FBI transform is concentrated on $ \Gamma_1^+ $,
with the exponential growth in the outgoing direction. See also 
\url{https://www.youtube.com/watch?v=R7eUpEVpckk&feature=youtu.be&list=PL32oQ9j1OWc0RIlNwnY6BtDawepnBnQyx}}
\label{f:FBI}
\end{figure}

Suppose now that we consider a scattering problem, in which case $ M $ is non-compact,
say, equal to $ \RR^n $ outside a compact set. In that case for 
a subset $\Omega = \Omega ( h ) $ we can define the number of
resonances of $ P ( h ) $ in $ \Omega $:
\begin{equation}
\label{eq:NhP}
N_{ P ( h ) } ( \Omega ) = | \Res ( P ( h ))  \cap \Omega | ,
\end{equation}
which generalizes the definition above. 
The basic heuristic principle is that 
the trapped set $ K_{ \Omega \cap \RR } $ should
replace $ p^{-1} ( [ a, b ] ) $ as everything else escapes. 
To guarantee that escape we assume 
that $ K \subset \{ \Im z > - \gamma h \} $. This 
restriction on the imaginary parts is 
explained by the fact that the corresponding resonant states have
an $h$-independent rate of decay: $ \Re e^{ - i t z/h } = e^{  t \Im z /h} 
\geq e^{ - t \gamma } $.

We should
then look at the maximal number of localized quantum states that can cover
$ K_{\Omega \cap \RR } $. But that number is now related to the
upper box or upper Minkowski dimension of $ K_{ \Omega \cap \RR } $:
\[  \dim^+_M ( U ) := \limsup_{ h \to 0 } \frac{ n ( \sqrt {h}, U  ) }{ \log ( 1/\sqrt{h})} ,\]
where $ n ( \epsilon , U ) $ is the minimal number of boxes of sides $ \epsilon $ needed to 
cover $ U $. This suggests that we should  have an upper bound
\begin{equation}
\label{eq:UBtr}
N_{ P ( h ) } ( [ a, b ] - i h [ 0 , \gamma] ) \leq C h^{ - \frac12 \dim^+_M (K_{[a,b]}) } . 
\end{equation}

This intuition is put to test by the fact that resonant states with $ 
\Im z > - \gamma h $ are {\em not}
localized to the trapped set but rather to the outgoing tails $ \Gamma^+_E $ given in 
\eqref{eq:defGama} -- see Figure~\ref{f:FBI} and \cite[Theorem 4]{NZ1}. 
The method of complex scaling described in \S \ref{rae} changes the exponential
growth we see in Figure \ref{f:FBI} to exponential decay and hence localizes
the resonant state to the interaction region. By a microlocal complex
scaling constructed using suitable low regularity {\em microlocal weights} 
Sj\"ostrand \cite{SjDuke} similarly localized resonant states to 
an $ \sqrt h $ neighbourhood of the trapped set and hence was able to 
prove bounds of the form \eqref{eq:UBtr}. Although we concentrate
here on hyperbolic trapped sets we remark that Dyatlov's asymptotic
formula \cite{dy} in Figure \ref{f:iv} is consistent with \eqref{eq:UBtr}.

Put 
\begin{equation}
\label{eq:dimM}
d_E := \frac{ \dim_M^+ K_E - 1 }{2} , \ \ \ \ 
\delta_E  = \frac{ \dim_H K_E -1 } 2 ,
\end{equation}
where $ \dim_H $ is the Hausdorff dimension (hence $ \delta_E \leq d_E $).
Then the following theorem was proved for 
Euclidean infinities by Sj\"ostrand--Zworski \cite{SZ10}, 
for even asymptotically hyperbolic infinities (see \S \ref{vasy})
by Datchev--Dyatlov \cite{fwl} (generalizing earlier results 
\cite{Z99},\cite{glz}) and for several convex obstacles by 
Nonnenmacher--Sj\"ostrand--Zworski \cite{NSZ}:

\begin{thm}
\label{t:fwl}
Under the assumption \eqref{eq:sectionass} and using the notation 
\eqref{eq:NhP},\eqref{eq:dimM}, for every $ \gamma $ and $ \epsilon > 0 $ there exists
$ C$ such that
\begin{equation}
\label{eq:fwl}
N_{P ( h ) } ( [ E - h , E+h ] - i h [ 0 , \gamma ] ) \leq
C h^{ - d_E - \epsilon } .
\end{equation}
When $ V = 0 $ in \eqref{eq:defPofh} and $ \dim M = 2 $ we can 
replace $ d_E + \epsilon $ by $ \delta_E $ in \eqref{eq:fwl}.
\end{thm}

For outlines and ideas of the proofs we refer the reader to general 
reviews \cite[\S\S 4,7]{Nonn},\cite{Naud3} and to 
\cite[\S 2]{SZ10},\cite[\S 1]{fwl},\cite[\S 1]{NSZ}.

We remark that a fractal upper bound for resonances should be valid
for more general quotients  $ \Gamma \backslash \HH^n $ than
the ones considered in \cite{fwl} -- see
Guillarmou--Mazzeo \cite{guma} for results on meromorphic continuation 
in that case. Even the simplest case of surfaces with both funnels and
cusps is not understood (the only results
treat the non-trapping case -- see Datchev \cite{Dat3} and
Datchev--Kang--Kessler \cite{dkk}).

\begin{figure}
\includegraphics[width=13cm]{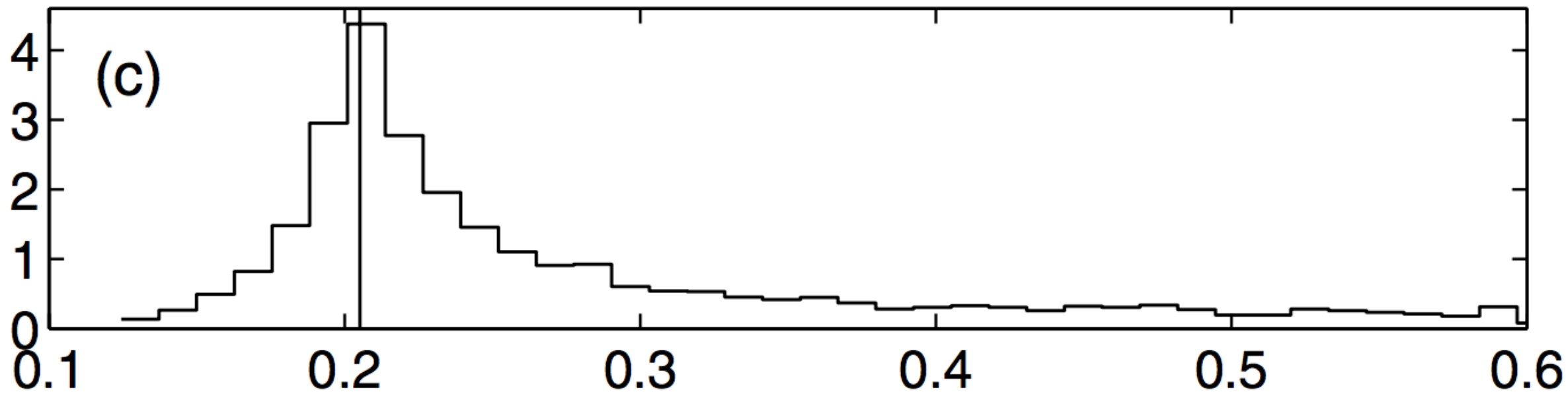}
\includegraphics[width=12.5cm]{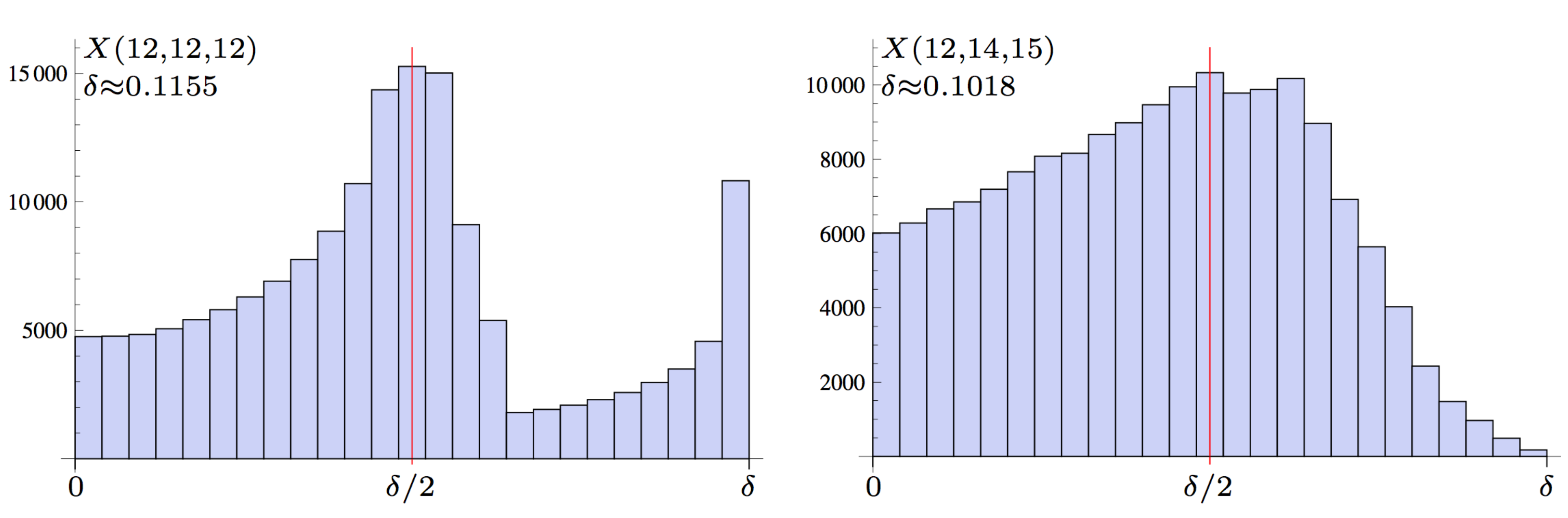}
\caption{Top: histogram of imaginary parts of resonances of a three
discs configuration (see Figure~\ref{f:kuhl}) with real parts
with $ | \Re \lambda | \leq 2 \times 10^3$ computed using a semiclassical zeta
function \cite{LSZ}. The imaginary parts concentrate at 
$ \frac 12\mathcal  P_1 ( 1 ) $, that is at one half of the classical
escape rate. (We should remark here that 
the method of calculation based on the zeta function, although widely accepted
in the physics literature, does not have a rigorous justification
and may well be inaccurate. However, the phenomenon of 
concentration at $ \frac12 \mathcal P_1 ( 1 )  $ has been verified experimentally -- see 
\cite{Bark} and Figure~\ref{f:class}). Bottom: histograms 
from Borthwick \cite{borth2} showing the distribution of 
imaginary parts (the $x$ axis is $   \Im \lambda + \frac12 $) for
hyperbolic surfaces $ X ( \ell_1, \ell_2, \ell_3 )$ (see Figure~\ref{f:borth1})
with $ | \Re \lambda | \leq 2\times 10^4 $.
The concentration at $ \Im \lambda = \frac12 ( \delta - 1) $ is clearly visible.
Recalling \eqref{eq:FtoG} and \eqref{eq:defdel} that is the same as in the top figure.
\label{f:concent}
}
\end{figure}

Existence of fractal lower bounds or asymptotics has been studied numerically,
first for a three bump potential by Lin \cite{Lin} and then using 
semiclassical zeta function calculations \cite{GaRi} for three disc
systems (see Figures \ref{f:kuhl} and \ref{f:class})
by Lu--Sridhar--Zworski \cite{LSZ}. The results 
have been encouraging and led to experimental investigation by 
Potzuweit et al \cite{Potz}: the experimentally observed density 
is higher than linear but does not fit the upper bound \eqref{eq:fwl}.
There are many possible reasons for this, including the limited 
range of frequencies. (Related experiments by Barkhofen et al \cite{Bark}
confirmed the pressure gaps of Theorem \ref{t:8}, see Figures~\ref{f:class} and
\ref{f:concent}).
Fractal Weyl laws have been considered (and numerically
checked) for various open chaotic quantum maps -- see \cite{NZ0},\cite{dylo}
and references given there. 
The have also been proposed in other types of chaotic 
systems, ranging from dielectric cavities to communication and social
networks -- see recent review articles by Cao--Wiersig \cite{CaWi} and 
Ermann--Frahm--Shepelyansky \cite{EFS} respectively. Fractal 
Weyl laws for {\em mixed systems} and localization of resonant
states have been investigated by K\"orber et al \cite{KMBK},\cite{KBK}.

So far, the only rigorous lower bound which agrees with the fractal upper bound was
obtained by Nonnenmacher--Zworski  in a special toy model 
quantum map \cite[Theorem 1]{NZ1}. Nevertheless the evidence seems
encouraging enough to reiterate the following conjecture \cite{LSZ},\cite[(4.7)]{gz}:

\begin{conj}
\label{c:3}
Suppose that $ N ( \Omega ) := N_{P(h) } ( \Omega ) $ is defined in \eqref{eq:NhP}
and $ \delta_E $ is given by \eqref{eq:dimM}. If $ \dim M = 2 $, then
\begin{equation}
\label{eq:fwl2}  \lim_{ h \to 0 } h^{  \delta_E } N( [E-h,E+h] - i h [ 0 , \gamma]) 
= v_E( \gamma ),  \end{equation}
where $ v_E ( \gamma ) > 0 $ for $ \gamma $ large enough.
When $ \dim M  >2 $ then 
\begin{equation}
\label{eq:fwln}  
\limsup_{ h \to 0 } h^{ \delta_E }  N ( [E-h,E+h] - i h [ 0 , \gamma] ) > 0 , 
\end{equation}
when $ \gamma $ is large enough.
\end{conj}

Numerical investigation of fractal Weyl laws lead to the observation that
imaginary parts of resonances (that is resonance widths) concentrate at 
$ \Im z \approx \frac12 \mathcal P_E ( 1 ) h $ where $ \mathcal P_E ( s ) $ is defined
by \eqref{eq:pressure}. The value $ | \mathcal P_E ( 1 ) |$ gives the
classical escape rate on the energy surface $ p^{-1} ( E) $ -- see 
\cite[(17)]{Nonn}. The point made by Gaspard--Rice \cite{GaRi} was that
the gap is given by $ |\mathcal P_E ( \frac12 ) |h < \frac 12| \mathcal P_E ( 1 )|h$
(with the latter being a more obvious guess). However most of resonances
want to live near $ \Im z = \frac 12 \mathcal P_E ( 1 ) h $ -- see
also Figure~\ref{f:dy1}. 

The first rigorous result indicating lower density away from $ \Im \lambda = 
\frac12 \mathcal P_E ( 1 ) $ was obtained by Naud \cite{Naud2} who
improved on the fractal bound in Theorem \ref{t:fwl} for 
$ \gamma < |\mathcal P_E ( 1 ) |/2 $ in the case of convex co-compact
surfaces. Dyatlov \cite{dy6} provided an improved bound for 
convex co-compact hyperbolic quotients in all dimensions. We state 
his result in the case of dimension two and in the 
non-semiclassical setting (resonances as poles of $  ( - \Delta_g^2 
- \frac 14 - \lambda^2 )^{-1} $ -- see \eqref{eq:Res4Deltag},\eqref{eq:FtoG}
and \eqref{eq:defdel}):

\begin{thm}
\label{t:dy6}
Let $ M = \Gamma \backslash \HH^2 $ be a convex co-compact quotient
(see Figure \ref{f:borth1}) and define
$  \mathcal N ( r , \beta )  $ to be the number of poles of the continuation of
$ ( - \Delta_g - \frac12 -\lambda^2 )^{-1} $ in 
\[    r <  \Re \lambda  \leq r + 1 , \ \ \Im \lambda > - {\textstyle{\frac12}} + 
( 1 - \beta  ) \delta , \]
where $ \delta $ is defined in \eqref{eq:defdel}. Then for any $ \beta $ and
$ \epsilon > 0 $
there exists $ C $ such that
\begin{equation}
\label{eq:tdy6}
\mathcal N( r , \beta ) \leq C r^{ \min ( 2 \beta \delta + \epsilon, \delta )  } .
\end{equation}
\end{thm}

The bound $ r^\delta $ comes from \cite{glz} and the improvement 
is in strips $ \Im \lambda > - \gamma > - \frac12 ( 1 - \delta ) $. 
A comparison between \eqref{eq:tdy6} and numerically fitted exponents
was shown in Figure \ref{f:dy1} (for the more robust counting fuction
$ N ( R , \beta ) = \sum_{k=1}^{[R]} \mathcal N ( R - k, \beta ) $). 
These numerical experiments of Borthwick--Dyatlov--Weich \cite[Appendix]{dy6}
again suggest that Conjecture \ref{c:3} (perhaps in a weaker form) holds,
while the optimality of \eqref{eq:tdy6} is unclear. Some lower bounds
in strips have been obtained by Jakobson and Naud \cite{jn1} but they 
are far from the upper bounds. 

Dyatlov--Jin \cite{dylo} proved an analogue of Theorem \ref{t:dy6} for open
quantum maps of the form shown in Figure~\ref{f:dylo}. That paper can 
be consulted for more numerical experiments and interesting conjectures.

Here we make a weak (and hopefully accessible) conjecture which 
in the case of quotients follows from Theorem \ref{t:dy6}:

\begin{conj}
\label{c:vE}
In the notation of \eqref{eq:fwl2} and for general operators satisfying 
\eqref{eq:sectionass}
\[ \gamma <  {\textstyle{\frac12} }
| \mathcal P_E ( 1 ) | \ \Longrightarrow  \ 
v_E ( \gamma ) = 0 .  \]
\end{conj}

Jakobson and Naud \cite{jn2} made a bolder conjecture that for convex co-compact hyperbolic 
surfaces there are only finitely many resonances with $ \Im \lambda > - 
\frac12 ( 1- \delta) $ -- see Figure \ref{f:dy1}.  We make an equally bold 
but perhaps more realistic
\begin{conj}
\label{c:4}
Suppose that $ N ( \Omega ) := N_{P(h) } ( \Omega ) $ is defined in \eqref{eq:NhP}.
Then 
for every $ \gamma > \mathcal P_E ( 1 )/2 $, there exists $ \epsilon_0 >  0 $
such that for every $ \epsilon < \epsilon_0 $, 
\begin{equation}
\label{eq:crazy}
\lim_{ h \to 0 } \frac{ N (  [ E - h , E + h ]  - {\textstyle{\frac 1 2}} i h
[ \mathcal P_E  ( 1 ) - \epsilon , \mathcal P_E ( 1) + \epsilon  ] )}
{N (  [ E - h , E + h ]  - i h [ 0 , \gamma] )}
=  1,\end{equation}
where $ \mathcal P_E ( s )  $ is the topological pressure defined in \eqref{eq:pressure}.
\end{conj}

In the case of convex co-compact hyperbolic quotients Theorem~\ref{t:dy6}
and Conjecture~\ref{c:3} imply that, for $ \gamma $ large enough,
\[ \lim_{ h \to 0 } \frac{ N (  [ E - h , E + h ]  - {\textstyle{\frac 1 2}} i h
[ 0,  \mathcal P_E  ( 1 ) - \epsilon , \gamma  ] )}
{N (  [ E - h , E + h ]  - i h [ 0 , \gamma] )}
=  1 . \] 
That could be a weaker but already very interesting substitute for \eqref{eq:crazy}.

\section{Pollicott--Ruelle resonances from a scattering theory viewpoint}
\label{dsPR}

Suppose that $ M $ is a manifold and $ X $ a smooth vector field on  $ M$.
The flow $ \varphi_t := \exp t X $ gives a group of diffeomorphisms
$ \varphi_t : M \to M $ and a large branch of the theory of dynamical
systems is concerned with {\em long time} properties of the evolution 
under $ \varphi_t $. One way to measure this evolution is by considering
{\em correlations}: for $ f, g \in \CIc ( M ) $ we define
\begin{equation}
\label{eq:corrfg} 
\rho_{f,g} ( t ) := \int_M f ( \varphi_{-t } ( x ) ) g ( x ) dm ( x) , 
\end{equation}
with respect to a measure on $ M $. (Sometimes $ ( \int_M f ) \times ( \int_M g ) $
is subtracted from $ \rho_{f,g} $ in the definition; it also interesting to 
consider less regular $ f$ and $ g$). What measure 
one should take is of great interest and invariant measures are the most
natural ones. The particularly interesting {\em Sinai--Ruelle--Bowen} (SRB) measures
(see \cite{kahal})
 have the property that
\[ \rho_{f,g} ( t ) = \lim_{ T\to \infty } \frac{1}{T} \int_0^T 
f ( \varphi_{s-t} ( x ) ) g ( \varphi_s ( x ) ) ds , 
\] 
for almost every $ x\in M $ with respect to a smooth (Lebesgue) measure.
This allows a computation of $ \rho_{f,g} ( t ) $ using one ``randomly chosen"
orbit of the flow.

\begin{figure}
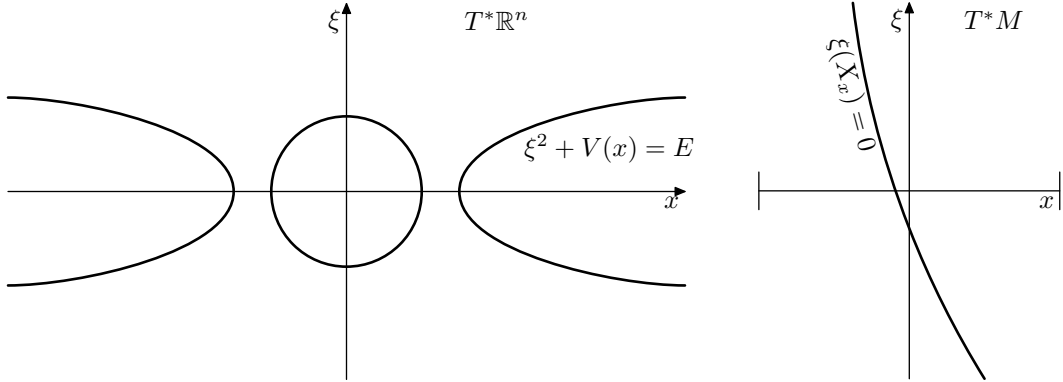

\includegraphics{revres.5} \hspace{0.2in} \includegraphics{revres.6}
\label{f:scat}
\caption{Scattering in $ \RR^n $ and Anosov flows on a compact manifold $M$:
the classical Hamiltonians in phase space (now a ``double phase space"
in case of Anosov flow, $ T^*M $)
are $ p ( x, \xi ) = |\xi|^2 + V ( x ) $, 
$ ( x, \xi ) \in T^* \RR^n $ and $ p( x, \xi ) = \xi ( X_x )  $, 
$ ( x, \xi ) \in T^* M $ (where $ X $ generates the
Anosov flow). In both cases, the energy surfaces $ p = E $ are {\em non-compact}:
in scattering $ x \to \infty $, in the Anosov case, $ \xi \to \infty $. 
Figure \ref{f:Epm} below should be compared to \cite[Figure 2]{mel}: it is not
surprising that Melrose's ``radial estimates" play a role in the analysis
of both cases. For flows the situation is simpler since the flow is
linear in the $ \xi $ variables (see \eqref{eq:Ptop}) and there is no
additional infinity (in the scattering case there is also the $ \xi $ infinity).
When $ M$ is non-compact the situation becomes more complicated as 
both infinities play a role: see \cite[Figure 2]{rnc}.}
\end{figure} 

With $ U ( t ) f := f \circ \varphi_{-t} $ this is the set-up
discussed in Figure~\ref{f:core} and we expect that for chaotic flows
we have a picture shown there. That is, we expect an expansion of
correlations in terms of suitably defined {\em resonances of the 
dynamical systems,} 
\begin{equation}
\label{eq:expcor} \rho_{f,g} ( t ) \sim \sum_{ \Im \lambda_j > - \gamma } 
e^{ - i \lambda_j t } u_j (f) v_j ( g ) + \mathcal O_{f,g} ( e^{-\gamma t } ), 
\end{equation}
where $ u_j , v_j \in \mathcal D' ( M ) $ and $ \lambda_j $ are independent
of the test functions $ f $ and $ g $. If we know that the number of $\lambda_j $'s with 
$ \Im \lambda_j > - \gamma $ is finite and $ \lambda_0 = 0 $ (with 
$ u_0 = v_0 = dm $) is the only real resonance then we have an
{\em exponential decay of correlations}:
\begin{equation}
\label{eq:expdcor}  \rho_{f,g} ( t ) -  \int_M f(x) dm(x )   \int_M g(x) dm( x)   = \mathcal O ( e^{ - \alpha t } ) . 
\end{equation}

Hence we see that the questions encountered here are similar to the questions
asked in the study of scattering resonances: existence of resonance expansions \S\S\ref{expan},\ref{resrel} and of 
 resonance free regions (gaps) \S\S\ref{resfree3},\ref{resfree}. One can also ask general questions about the distribution of $ \lambda_j $'s.

As signatures of a chaotic systems the {\em dynamical resonances} $ \lambda_j $'s were introduced
by Pollicott \cite{Pol} and Ruelle \cite{rue} -- see Figure~\ref{f:PNAS} for 
an example in modeling of physical phenomena. They can also be
studied in the simpler setting of maps, that is for systems with discrete time -- 
see Baladi--Eckmann--Ruelle \cite{bear} for an early study and  
Baladi \cite{ballade},\cite{ballade1} for later developments.

The explicit analogy with scattering theory was emphasized by 
Faure--Sj\"ostrand \cite{fa-sj} following  earlier works by Faure--Roy--Sj\"ostrand
\cite{faros} and Faure--Roy \cite{faro} in the case of maps. Semiclassical methods for the 
study of decay of correlations were also introduced by Tsujii \cite{Ts} who
applied FBI transform techniques to obtain precise bounds on the 
asymptotic gaps for certain flows (see \S\ref{resfreeA} below).
This led to rapid progress some of which is described below -- to see
the extent of this progress one can compare the
current state of affairs  to that in the early review by Eckmann \cite{eck}.
For some other recent developments see also Faure--Tsujii \cite{fatsss}.

In this section we first give a precise definition of chaotic dynamical 
systems and then define {\em Pollicott--Ruelle} resonances. We 
concentrate to the case of $ M$ compact but  give indications what
happens in the non-compact case.
We  then explain the connection
to dynamical zeta functions
and survey results on 
resonance free strips and on ounting of resonances.

\subsection{Anosov dynamical systems}

Let $ M $ be a compact manifold and $ \varphi_t = \exp t X :  M \to M  $ 
a $ C^\infty $ flow  generated by  $ X \in
C^\infty ( M; T M) $. 

In this article the precise meaning of being {\em chaotic} is that
the flow is an {\em Anosov flow}. That means that the tangent space
to $ M $ has a continuous decomposition $ T_x M = E_0 ( x ) 
\oplus E_s (x) \oplus E_u ( x)$ which is invariant, 
$ d \varphi_t ( x ) E_\bullet ( x ) = E_\bullet ( \varphi_t ( x ) ) $,
$E_0(x)=\mathbb R X_x $,
and for some $ C$ and $ \theta > 0 $ fixed 
\begin{equation}
\label{e:anosov}
\begin{split} 
&  | d \varphi_t ( x ) v |_{\varphi_t ( x )}  \leq C e^{ - \theta
  |t| } | v |_x  , \ \ v
\in E_u ( x ) , \ \ 
t < 0 , \\
& | d \varphi_t ( x ) v |_{\varphi_t ( x ) }  \leq C e^{ - \theta |t| } |  v |_x , \ \
v \in E_s ( x ) , \ \ 
t > 0 .
\end{split} 
\end{equation}
where $ | \bullet |_y $ is given by a smooth Riemannian metric on $
X $. This should be compared to Definition~\ref{d:hyp}: in the case
when $ M $ is compact all of $ M$ is ``trapped" and Anosov flows
are hyperbolic everywhere.

Following Faure--Sj\"ostrand \cite{fa-sj} we exploit the analogy between
dynamical systems and quantum scattering, with the fiber $\xi$-infinity 
playing the role of $x$-infinity in scattering theory. The 
pullback map can be written analogously to the Schr\"odinger propagator:
\begin{equation}
\label{eq:phtoP} 
\begin{gathered} \varphi_{-t}^* = e^{ - i t P } : \CI ( M ) \to \CI ( M) ,  \\ 
\varphi_{-t}^* f ( x) := f ( \varphi_{-t} ( x )) , \ \ 
P := {\textstyle{\frac 1 i }} X . 
\end{gathered} \end{equation}
Sometimes it is convenient to make the seemingly trivial semiclassical
modification and write $ P_h = \frac h i X $, $ \varphi_{-t}^* = 
e^{ - itP_h/h} $. In either case 
the symbol of $ P $ and its Hamiltonian flow 
\begin{equation}
\label{eq:Ptop} \sigma ( P ) =: p ( x , \xi ) = \xi( X_x ) , \ \ e^{ t H_{p} }(x,\xi) = 
 ( \varphi_t ( x ) , ( {}^T                          
d\varphi_t ( x) )^{-1} \xi )  .\end{equation}
Here $ H_p $ denotes the Hamilton vector field of $ p$: 
$ \omega ( \bullet , H_p ) = dp $, where $ \omega = d ( \xi dx ) $ 
is the symplectic form on $ T^* M $. This should be compared
to \eqref{eq:Hamf} and \eqref{eq:Hamp}, noting that $ \varphi_t $
mean a different thing in \S \ref{srr}. For us the classical
flow $ \varphi_t $ on $M $ is the ``quantum propagation" $ e^{ - i t P } $ and the 
corresponding Hamiltonian flow $ \exp t H_p $ is the symplectic lift of $ \varphi_t $
to $ T^*X $.

In the study of $ P $ we need the dual decomposition of 
the cotangent space:
\begin{equation}
\label{eq:Estar} T_x^*X=E_0^*(x)\oplus E_s^*(x)\oplus E_u^*(x),
\end{equation}
where
$E_0^*(x),E_s^*(x),E_u^*(x)$ are annhilators of
$ E_s ( x ) \oplus E_u ( x ) $, $ E_0 ( x ) \oplus E_s ( x ) $,
and $ E_0 ( x ) \oplus E_u ( x )$. Hence they are dual to 
to $E_0(x),E_u(x),E_s(x)$, respectively.

An important special class of Anosov flows is given by contact Anosov flows.
In that case $ M $ is a contact manifold, that is a manifold equipped
with a contact $1$-form $ \alpha $: that means that if the dimension of $M$ is 
$ 2k+1 $ then $ (d \alpha)^{\wedge k } \wedge \alpha $ is non-degenerate.
A contact flow is the flow generated by the Reeb vector field $ X$:
\begin{equation}
\label{eq:contact}      \alpha ( X ) = 1 , \ \ d \alpha ( X , \bullet ) = 0  . 
\end{equation}
Natural examples of Anosov contact flows are obtained from
negatively curved Riemannian manifolds $ ( \Sigma , g ) $: 
\begin{equation}
\label{eq:Sigma} M = S^*\Sigma := \{ ( z, \zeta ) \in T^* \Sigma \, : \, |\zeta|_g = 1 \},  \ \ \ \alpha = 
\zeta dz |_{S^* \Sigma } . 
\end{equation}

The scattering theory/microlocal point of view to Anosov flows
has already had applications outside the field: 
by Dang--Rivi\'ere \cite{DR16} 
to the analysis of Morse--Smale gradient flows,
by Guillarmou 
\cite{gu15},\cite{gu16}, Guillarmou--Monard \cite{gumo}, 
Guillarmou--Paternain--Salo--Uhlmann \cite{gypsy} to inverse problems
and by Faure--Tsujii to the study of semiclassical zeta functions
\cite{fatss}.

\subsection{Definition of Pollicott--Ruelle resonances}
\label{defPRres}

The intuitive definition should follow Figure \ref{f:core}. We should 
consider the {\em power spectrum} of correlations:
\begin{equation}
\label{eq:rhohfg} \hat \rho_{f,g} ( \lambda ) := \int_0^\infty e^{ i \lambda t } \rho_{f,g} ( t ) 
dt , \end{equation}
which is well defined for $ \Im \lambda > 0 $. If we show that
\begin{equation}
\label{eq:powersp}
\text{ $ \lambda \longmapsto \hat \rho_{f,g} ( \lambda ) $ continues meromorphically to $ \Im \lambda > - A $,}
\end{equation}
then the poles of the continuation of $ \hat \rho_{f,g} (  \lambda ) $ 
are the complex frequencies we expect to appear in expansions of correlations \eqref{eq:expcor}. 

Let us assume now that the correlations \eqref{eq:corrfg} are 
defined using a smooth (not necessarily invariant) density $ dm ( x ) $.
We denote by  $ \langle \bullet, \bullet \rangle $ the distributional pairing
using this density.
 In the notation of \eqref{eq:phtoP} and for $ \Im \lambda > 0 $, 
\begin{equation}
\label{eq:rhotoP} \hat \rho_{f,g} ( \lambda ) =
\int_0^\infty \langle e^{ - i t P } f ,  g \rangle 
e^{ i \lambda t } dt = 
{\frac 1 i } \langle ( P - \lambda )^{-1} f , g \rangle 
\end{equation}
which means that \eqref{eq:powersp} follows from 
\begin{equation}
\label{eq:contiP}
 \text{ $ ( P - \lambda )^{-1} : \CI ( M ) \to \mathcal D' ( M )  $ continues meromorphically to $ \Im \lambda > - A $,}
\end{equation}
To obtain this meromorphic continuation one wants to find suitable 
spaces on which $ P - \lambda $ is a Fredholm operator. 
A microlocal construction of such Hilbert spaces was provided by 
Faure--Sj\"ostrand \cite{fa-sj} but the origins of the method
lie
in the works on anisotropic Banach spaces by
Baladi--Tsujii~\cite{bats},
Blank--Keller--Liverani~\cite{b-k-l},
Butterley--Liverani~\cite{but-liv},
Gou\"ezel--Liverani~\cite{go-liv}, and
Liverani~\cite{liver,liver2} -- see \cite{ballade1} for a recent
account in the setting of maps. 

Here we will follow a modified approach of \cite[\S\S3.1,3.2]{zeta}
where the spaces are defined using microlocal weights. To describe the
spaces we denote
by $ \Psi^{0+} ( M ) $ the space of pseudodifferential operators
of order $ \epsilon $ for any $ \epsilon > 0 $, with $ \sigma : 
\Psi^{m} ( M ) \to S^m ( T^* M ) / S^{m-1} ( T^*M ) $ denoting 
the {\em symbol map} (see \eqref{eq:Hamf} where we used the 
{\em semiclassical symbol}, $ p = \sigma_h ( P ) $ and 
\cite[Appendix B]{zeta} for definitions and references). 
We then put
\begin{gather}
\label{eq:weighted}
\begin{gathered}
H_{rG } (M ) := \exp ( - rG ( x,  D) ) L^2 ( M ) , \ \  G \in \Psi^{0+} (  M) , \\
\sigma ( G ) = ( 1 - \psi_0 ( x, \xi) ) m_G ( x , \xi ) \log | \xi|_g , 
\end{gathered}
\end{gather}
where $ \psi_0 \in C^\infty_{\rm{c}} ( T^*M, [ 0 , 1 ] ) $ is $ 1 $ near $\{ \xi = 0 \}$, 
$ m_G ( x, \xi ) \in C^\infty ( T^* M \setminus 0 , [ -1, 1 ] ) $ is homogeneous of
degree $ 0 $ and satisfies (see \eqref{eq:Estar})
\begin{equation}
\label{eq:defmG}  m_G ( x , \xi ) = \left\{ \begin{array}{ll}  \ \ 1 & \text{near $ E_s^* $}\\
-1 & \text{near $ E_u^* $ } \end{array} \right. \ \ H_p m_G ( x, \xi ) \leq 0 , \ \
( x, \xi) \in T^* M \setminus 0 . \end{equation}
The existence of such $ m_G $ is shown in \cite[Lemma C.1]{zeta}. 
Intuitively, it is clear that it can be done 
if we look at Figure \ref{f:Epm}: the flow lines
of $ H_p $ go from $ E_s^* $ to $ E_u^* $ or towards the zero section. Hence
we can have a function which is $ 1 $ near $ E_s^* $ and $ -1 $ near 
$ E_u^* $ and decreases along the flow. 

\begin{figure}
\includegraphics[scale=1.2]{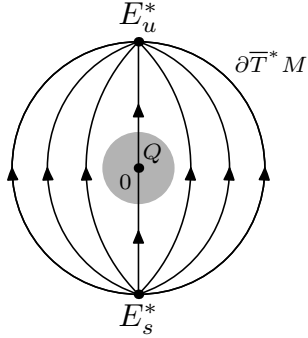}
\caption{A schematic representation of the flow on the compactification of
$ T^* M $, ${\overline T}^* M$ -- by compactification we mean replacing each fiber 
$ T_x^* M $ by a ball $ {\overline T}^*_x M $ with the boundary corresponding to the sphere at infinity. At zero energy, $ p^{-1} ( 0 ) = \{ ( x, \xi) : \xi (X_x ) = 0 \} $
the trapped set of the flow is given by the zero section and the operator
$ X/i - i Q ( x, h D )/h - \lambda $ is invertible on spaces $ H_{ r G  } $ 
(with uniform bounds once the norms are semiclassical modified by 
changing $ G ( x, D )$ to $ G ( x, h D) $) for $ Q \geq 0 $ microlocalized 
to the shaded region and elliptic in a neighbourhood of the trapped set and
$ \Im \lambda > -r/C_0 $.  Since $ Q ( x , h D )$
is a compact operator that, and estimates on the adjoint on dual spaces, 
shows the Fredholm property of $ X/i - \lambda $. }
\label{f:Epm}
\end{figure}

The properties of $ m_G $ and hence of the operator $ G ( x, D ) $ 
show that for $ r\geq 0 $, 
\[  H^r ( M ) \subset H_{rG } ( M ) \subset H^{-r} ( M ) . \]
Moreover, 
{\em microlocally} near $ E_s^* $ and $ E_u^* $ the
space $ H_{rG} $ is equivalent to $ H^{r} $ and $ H^{-r} $ respectively. 
To explain this rigorously we recall the notion of {\em wave front set}, $
\WF ( u ) $,
of a distribution $ u \in \mathcal D' ( \RR^n ) $ (since the notion is
local, or rather {\em microlocal}, the definition works for 
manifolds as well): for $ ( x , \xi ) \in T^* \RR^n \setminus 0 $, 
\begin{equation}
\label{eq:WF} 
\begin{gathered}
 ( x, \xi ) \notin 
\WF ( u )  \\ \Updownarrow \\
\exists \, \varphi \in \CIc ( \RR^n ) ,\, \, \varphi ( x) \neq0 , \, \epsilon > 0  
\ \forall \, N  \ \exists \, C_N \\
 |\widehat{ \varphi u } | \leq C_N  ( 1 + |\eta | )^{-N} , \ \ 
 \eta \in \RR^n \setminus 0 , \, 
\left|{\textstyle{\frac{\eta}{|\eta|} - \frac{\xi}{|\xi|}}} \right| < \epsilon ,
\end{gathered}
\end{equation}
see \cite[Definition 8.1.2]{H1} and \cite[Theorem 18.1.27]{H3} for a useful
characterization.
The definition means that the Fourier transform of a localization of 
$ u $ to a neighbourhood of $ x $ decays rapidly in a conic neighbourhood
of $ \xi $. Rapid decay in all directions would mean that $ u $ is
smooth at $ x $ and hence $ \WF ( u ) $ provides a phase space localized
description of smoothness. We also recall \cite[Definition 8.2.2]{H1}:
for a closed conic subset, $ \Gamma $, of $ T^* M \setminus 0 $ ($0$ denotes the zero section $ \xi = 0 $; conic means the invariance under the multiplicative
action of $ \RR_+ $ on the fibers),
\begin{equation}
\label{eq:defDG}
\mathcal D'_\Gamma ( M ) := \{ u \in \mathcal D'_\Gamma ( M ) \; : \;
\WF ( u ) \subset \Gamma \} .
\end{equation}
With this notation we can say that there exist conic neihbourhoods, 
$ \Gamma_s $ and $ \Gamma_u $ if $ E_u^* $ and $ E_s^* $ respectively 
such that for $ r \geq 0 $,
\begin{equation}
\label{eq:mzation}
\begin{split} H_{rG} ( M ) \cap \mathcal D'_{ \Gamma_{u} } ( M) & = 
H^{-r } ( M ) \cap \mathcal D'_{ \Gamma_u } ( M)  , \\
H_{rG} ( M ) \cap \mathcal D'_{ \Gamma_s } ( M) & = 
H^{r } ( M ) \cap \mathcal D'_{ \Gamma_s } ( M) . 
\end{split} \end{equation}
Hence, our spaces improve regularity along the flow (see Figure \ref{f:Epm})
and that leads to a Fredholm property (roughly, the
inclusion $ H^{r} \hookrightarrow H^{-r} $, $ r > 0 $,  is a compact operator). 
Such spaces appeared in scattering theory long ago: in the
work of Melrose \cite{Mel1} on the Poisson formula for resonances 
and, with an explicit use of microlocal weights (in the analytic setting
via the FBI transform), in the work of Helffer--Sj\"ostrand \cite{He-Sj0} 
on a microlocal version of complex scaling.

The following theorem  was first proved
by Faure--Sj\"ostrand \cite{fa-sj} 
for more specific weights $ G $
and 
by Dyatlov--Zworski \cite{zeta}.
The characterization of resonant
states using a wave front set condition is implicit in 
\cite{fa-sj} and is stated in Dyatlov--Faure--Guillarmou
\cite[Lemma 5.1]{DGF} and \cite[Lemma 2.2]{zazi}. For non-compact
$ M $ with hyperbolic trapped sets for the flow $ \varphi_t $ 
a much more complicated analogue was proved by 
Dyatlov--Guillarmou \cite[Theorem 2]{rnc}.

\begin{thm}
\label{t:Fred}
Suppose that $ H_{rG} ( M )  $ is defined in \eqref{eq:weighted},
$ P = \frac 1 i X $, where $ X $ is the generator of the (Anosov)
flow. Define
$ D_{rG} ( M ) := \{ u \in H_{rG} ( M ) : P u \in H_{rG} (M ) \} $.
There exist constants, $ C_0 ,C_1  $ such that
\begin{equation}
\label{eq:Fred}
P - \lambda : D_{ r G } ( M ) \to H_{ r G }(M)  , \ \ \
 r > C_0 ( \Im \lambda )_- +C_0 , 
\end{equation}
is a Fredholm family of operators, invertible for $ \Im \lambda > C_1 $.
Hence, $ \lambda \mapsto ( P - \lambda )^{-1} :  H_{ r G } \to 
H_{ r G } $ is a meromorphic family of operators in $ \Im \lambda > 
- r/C_0 $ and in particular \eqref{eq:contiP} holds for any $ A $.
In addition,
\begin{equation}
\label{eq:PtoWF}
\frac{1}{ 2\pi i} \tr_{ H_{ r G } ( M ) } \oint_\lambda 
( \zeta - P )^{-1} d \zeta = 
\dim \{ u \in \mathcal D'_{ E_u^* } ( M ) : 
\exists \, \ell \ ( P - \lambda)^\ell u = 0 \} . 
\end{equation}
Here the integral is over a sufficiently small positively oriented
circle centered at $ \lambda $.
\end{thm}

This result remains valid for operators
$  \mathbf P : \CI ( M , \mathcal E) \to \CI ( M , \mathcal E ) $ 
where $ \mathcal  E $ is any smooth complex vector bundle and $ \mathbf P $
is a first order differential system with the principal part given by 
$ P $. That will be important in \S \ref{zeta}. 

Theorem \ref{t:Fred} provides a simple definition of Pollicott--Ruelle
resonances formulated in terms of the wave front set condition and
action on distributions. We state it only in the scalar case:

\begin{defi}
\label{d:PR}
Suppose $ M $ is a compact manifold, $ X \in \CI ( M ; TM ) $
generates an Anosov flow in the sense of \eqref{e:anosov} and 
$ E_u^* $ is defined in \eqref{eq:Estar}.

We say that $ \lambda \in \CC $ is a {\em Pollicott--Ruelle resonance}
if
\begin{equation}
\label{eq:defPR}
\exists \, u \in \mathcal D'_{E_u^* } ( M ) \ \ 
( P - \lambda ) u = 0 , \ \  P:= {\textstyle{ \frac 1 i } } X ,
\end{equation}
where $ \mathcal D'_{E_u^* } ( M ) $ is defined in \eqref{eq:defDG}.
The multiplicity of the resonance is defined by the right hand side of
\eqref{eq:PtoWF}.
\end{defi} 
Our insistence of $ \frac 1 i $ in $ P$ comes from quantum mechanician's
attachment to self-adjoint operators: $ P $ is self-adjoint on 
$ L^2 ( M , dm ) $ if the flow admits a smooth invariant measure $ dm $.
In some conventions, e.g. in \cite{DGF},\cite{g-w}, resonances are given by 
$ s = - i \lambda$. 

From the physical point of view this definition should be stable
when the flow is randomly perturbed, that is if we change
\begin{equation}
\label{eq:brown} \dot x ( t ) = - X_{ x( t ) }, \quad x(0)=x \ \ 
\longrightarrow \ \ \dot x_\epsilon ( t ) = 
- X_{ x_\epsilon ( t ) } + \sqrt{ 2 \epsilon } \dot B( t )  , \quad x_\epsilon 
(0)=x, 
\end{equation}
where where $ B ( t ) $ is the Brownian motion 
corresponding a Riemannian metric $ g$ on $M$ -- see for instance \cite{brown}.
This corresponds to changing the evolution by $ P = X/i $:
\begin{equation}
\label{eq:brown1}
e^{ - i t P } f ( x ) = f ( x ( t ) ) \ \ \longrightarrow 
\ \ e^{ - i t P_\epsilon  } f ( x ) = \mathbb E \left[ f (x_\epsilon (  t)  ) \right] ,
\ \ P_\epsilon := P + i \epsilon \Delta_g . 
\end{equation}
The operator $ P_\epsilon $ is elliptic for $ \epsilon > 0 $ and hence
has a discrete spectrum with $ L^2 $ eigenfunctions. This spectrum 
converges to the Pollicott--Ruelle resonances as was shown in \cite{damped}.
Stability of Pollicott--Ruelle resonances for Anosov maps
has been established by Blank--Keller--Liverani \cite{b-k-l} and Liverani \cite{liver2},
 following a very general argument of Keller--Liverani \cite{kl}. 

\begin{thm}
\label{t:damped}
Suppose that $ \lambda_0 $ is a Pollicott--Ruelle resonance of
multiplicity $ m$ (see Definition \ref{d:PR}). Then there exists
$ r_0 > 0  , \epsilon_0 > 0 $ such that for $ 0 < \epsilon < \epsilon_0 $,
$ P_\epsilon $ (defined in \eqref{eq:brown1}) has exactly 
$ m$ eigenvalues $ \{ \lambda_j ( \epsilon ) \}_{j=1}^m $ in 
$ D ( \lambda_0 , r_0  ) $, and 
\[  \lambda_j ( \epsilon ) \to \lambda_0 , \]
with the convergence uniform for $ \lambda_0 $ in a compact set.

Moreover,  if $ ( M , \alpha ) $ is 
a contact manifold and $ X $ is its Reeb vector field (see \eqref{eq:contact}),
then for any $ \delta > 0 $ there exist $ R > 0 $ such that for all
$ \epsilon > 0 $, 
\begin{equation}
\label{eq:tdamped}   
\begin{gathered}
\Spec ( P_\epsilon ) \cap 
\left( [ R , \infty ) - i [ 0 , \textstyle{\frac12} ( \gamma_0 - \delta ) ] \right)
= \emptyset ,
\end{gathered}
\end{equation}
where 
\begin{equation}
\label{eq:defga0}
\gamma_0 = \liminf_{ t \to \infty } \frac 1 t \inf_{ x \in M  } 
\log \det \left( d\gamma_t |_{ E_u ( x ) } \right) .
\end{equation}
\end{thm}
The last statement means that an asymptotic resonance free region for
contact flows (see \S \ref{resfreeA}) is stable under random perturbations.
It also comes with polynomial resolvent bounds
$ ( P_\epsilon  - \lambda )^{-1} = \mathcal ( |\lambda|^{N_0} )_{ H^{s_0} \to 
H^{-s_0} } $. Stability of the gap for certain maps 
was established by Nakano--Wittsten \cite{NaWi}. 

\begin{figure}
\includegraphics[scale=1]{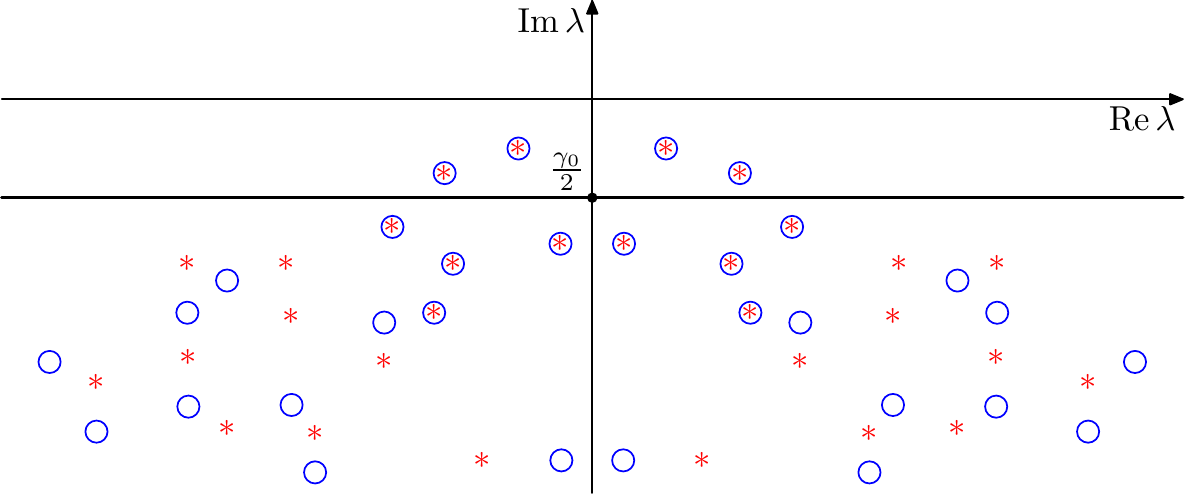}
\caption{A schematic presentation of the results in Theorems~\ref{t:damped}. Pollicott--Ruelle resonances of the generator of
the flow $ X $ (denoted by red asterisks) are approximated by 
the eigenvalues of $ X/i + i \epsilon \Delta_g $ (denoted by blue circles)
uniformly on compact
sets. The asymptotic resonance free strip is uniform with respect to 
$ \epsilon $.}
\label{f:PR}
\end{figure}

An elliptic perturbation $ P_\epsilon = X/i + i \epsilon \Delta_g $
is natural when there is no additional structure. However, for Hamiltonian 
systems, or more specifically, for Anosov geodesic flows 
on Riemannian manifolds $ ( \Sigma, g ) $ -- see \eqref{eq:Sigma} --
 a more natural operator is given by 
\begin{equation}
\label{eq:drou} 
\begin{gathered}
 \widetilde P_\epsilon := {\textstyle{\frac1i}} X   + i \epsilon \widetilde \Delta_{g} , \\
\widetilde \Delta_{g} u ( z, \omega ) := 
 \sum_{ i,j}^n  g_{ij} ( z ) \partial^2_{\xi_i \xi_j} \left[ u ( z, \xi/|\xi|_g)\right] , \\ \omega = \xi/|\xi|_g , \ \ 
u ( z, \bullet ) \in \CI ( S^*_z \Sigma ) .
\end{gathered}
\end{equation}
The operator $ \widetilde P_\epsilon $ is now only {\em hypoelliptic}
and the Brownian motion ``kick" occurs only in the momentum variables
which is physically natural. The analogue of the first
part of Theorem \ref{t:damped} was recently proved  by 
Drouot \cite{Dr2} and that paper can be consulted for
background information and many references. 
The proof combined methods of \cite{damped}
and semiclassical hypoelliptic estimates inspired by the work 
of Lebeau \cite{lef}. An alternative approach to these estimates
was given by Smith \cite{s16} who adapted his earlier paper \cite{s94}
to the semiclassical setting.

We formulate the $ \widetilde P_\epsilon $ analogue of the second part of
Theorem \ref{t:damped} as
\begin{conj}
\label{c:8}
Suppose $ M = S^* \Sigma $, where $ ( \Sigma, g ) $ is a compact Riemannian
manifold whose geodesic flow is an Anosov flow. Then 
the second part of Theorem \ref{t:damped} holds with $ \widetilde P_\epsilon
$ in place of $ P_\epsilon$.
\end{conj}

\subsection{Connections to dynamical zeta functions}
\label{zeta}

In addition to expansions of correlations, the study of Pollicott--Ruelle
resonances is also motivated by their appearance as poles and zeros
of dynamical zeta functions. For motivation and history of that rich subject we 
refer to \cite[\S 1]{glp} and \cite{polly}. Here we will consider
the {\em Ruelle zeta function} \cite{Ru1} which is defined by analogy with 
the Riemann zeta function, $ \zeta (s  ) =  \prod_p ( 1 - p^{-s} )^{-1} $,  with primes $ p $ replaced by lengths of primitive
closed geodesics
\begin{equation}
  \label{eq:defRuz}
\zeta_R ( \lambda ) := \prod_{ \gamma \in \mathcal G } ( 1 - e^{ i \lambda  
\ell_\gamma } ) .
\end{equation}
(We switch to our convention $ \lambda = i s $.)
Here $ \mathcal G $ denotes the set of primitive closed trajectories 
(that is trajectories on which one ``goes around only once") of 
$ \varphi_t $ and $ \ell_\gamma $ is the length of the trajectory $ \gamma $.

One question, asked by Smale in \cite[\S II.4]{Sm} was 
if $ \zeta_R $ admits a meromorphic continuation to $ \CC $\footnote{I cannot
resist recalling that Smale referred
to it as a ``wild idea" and wrote
``I must admit a positive answer would be a little shocking!"} or to 
large strips $ \Im \lambda > - A $. When 
$ M = S^* \Sigma $ (see \eqref{eq:Sigma}) where $ ( \Sigma, g ) $ is
a compact Riemannian surface of negative curvature, that had already 
been known thanks to the Selberg trace formula and the meromorphy 
of the Selberg zeta function:
\begin{equation}
\label{eq:Selberg}  \zeta_S ( s ) := \prod_{ \gamma \in \mathcal G } \prod_{ m=0}^\infty
( 1 - e^{- ( m + s ) \ell_\gamma } ) , \quad
\zeta_R ( is ) = \frac{ \zeta_S ( s  ) }{ \zeta_S ( s + 1  ) } ,
\end{equation}
see for instance \cite[Theorem~5]{Mark} for a self-contained presentation. 

The first zero of $ \zeta_R $ is related to topological entropy 
(the value of the pressure at $ 0 $ in the notation of \eqref{eq:pressure})
and the continuation to a small strip past that first zero was achieved by 
Parry--Pollicott \cite{Papo}. To obtain larger strips turns out to be
as difficult as obtaining global meromorphic continuation (which 
proceeds strip by strip; in the case of $ C^k $ flows the meromorphy
only holds in a strip of size depending on $ k $ -- see \cite{glp}; 
our microlocal arguments give that as well but with less precision).
When the manifold $ M$ and the flow are {\em real analytic} that
continuation was obtained by Rugh \cite{Ru} and Fried \cite{Fr2}.
For Anosov flows on smooth compact manifolds it was first established 
by Giulietti--Liverani--Pollicott~\cite{glp} and then by Dyatlov--Zworski~\cite{zeta}.
Dyatlov--Guillarmou \cite{rnc} considered the more complicated non-compact case and essentially settled the original conjecture of Smale. 

We will now explain the proof \cite{zeta} of the meromorphic continuation of
$ \zeta_R $. The first step \cite{Ru1} is a factorization of the zeta function 
valid in the case when the stable and unstable bundles, $ x \mapsto E_s(x)$, 
$ x \mapsto E_u (x) $, respectively, are orientable\footnote{See \cite[Appendix B]{glp} for the modifications needed in the non-orientable
case.}:
\begin{equation}
\label{eq:factze}
\zeta_R ( \lambda ) = \prod_{ j=0}^{n-1} \zeta_j ( \lambda )^{(-1)^{j+ q } } , \ \ \
\zeta_j ( \lambda ) := 
\exp \left( - \sum_{ k=1}^\infty \sum_{\gamma \in \mathcal G}    
\frac{   e^{ i \lambda k \ell_\gamma } 
   \tr \wedge^{ j} \mathcal P_\gamma^k }
 { k |  \det(I - \mathcal P_\gamma^k)  | } \right) , 
 \end{equation}
where $  \dim M = n $, $ q = \dim E_s $ and the Poincar\'e map is defined by 
\begin{equation}
\label{eq:Poinc} \mathcal P_\gamma := d\varphi_{-\ell_\gamma} ( x_\gamma ) |_{ E_u ( x_\gamma ) 
 \oplus E_s ( x_\gamma ) } , \ \ \  x_\gamma = \varphi_{ \ell_\gamma } ( x_\gamma ) \in \gamma. \end{equation}
Since we are taking
determinants and traces, $ \zeta_j $'s do not depend on the choices of 
$ x_\gamma $. The factorization essentially follows from 
$  \det ( I - A ) = \sum_{ j=0}^{ n-1} ( - 1 )^j \tr \wedge^j A $ -- see
\cite[\S 2.2]{zeta}.

The relation of the zeta functions $ \zeta_j $ 
with the flow
by the Atiyah--Bott--Guillemin trace formula (see \cite[Appendix B]{zeta} 
for a proof):
\begin{equation}
\label{eq:ABG}
\begin{gathered}
\tr^\flat e^{ - i t {\mathbf P}_j  } = 
\sum_{ k=1}^\infty \sum_{\gamma \in \mathcal G} 
\frac{ \ell_\gamma \tr \wedge^j \mathcal P_\gamma^k \delta ( t - k \ell_\gamma ) }
{ | \det ( I - \mathcal P_\gamma^k ) |} , \ \ t > 0 , \\ 
{\mathbf P}_j = {\textstyle{\frac 1 i } } 
\mathcal L_X : \CI ( M; \mathcal E^j_0 ) \to \CI ( M; \mathcal E^j_0 ) , 
\end{gathered}
\end{equation}
where
$\mathcal E^j_0$ be the smooth invariant verctor bundle of
all differential $j$-forms $\mathbf u$ satisfying $\iota_V\mathbf u=0$,
where $\iota$ denotes the contraction operator by a vector field.

The flat trace, $ \tr^\flat $, is defined using operations on 
distributions which generalize integration of the Schwartz kernel 
over the diagonal: let $ \jmath : \RR_+ \times M \to 
\RR_+ \times M \times M $ be given by 
$ \jmath  (t, x) = (t, x , x) $ and $ \pi : \RR_+ \times M \to \RR_+ $, by 
$ \pi ( t, x ) = t $. If $ K \in \mathcal D' (\RR_+ \times M \times M ) 
$ is the Schwartz kernel of the operator $ e^{ - i tP } = \varphi_{-t}^* $
then 
\begin{equation}
\label{eq:flatr1}   \tr^\flat e^{ - i t {\mathbf P}_j  } := 
\pi_* \jmath^* K \in \mathcal D' ( \RR_+ ) . \end{equation}
The push forward of a distribution $ \pi_* $ is always well defined
in case compact manifolds but {\em that is not the case for a pullback} by an 
inclusion
$ \jmath^*$ which, if $ K$ were smooth would be 
\begin{equation}
\label{eq:pbj}  \jmath^* K ( t , x ) := K ( t, x, x )  
\end{equation}
 -- see \cite[Theorem 8.2.4, Corollary 8.2.7]{H1}. The condition
which justifies \eqref{eq:pbj} for distributional $ K$ is 
\begin{equation}
\label{eq:WFK}
\begin{gathered}
  \WF ( K ) \cap N^* ( \RR_+ \times \Delta ) = \emptyset, \\ 
  \RR_+ \times \Delta := \{ ( t , x, x ) : ( t, x ) \in\RR_+ \times
\Delta \}, 
\end{gathered}
\end{equation}
where the wave front set $ \WF $ was defined in \eqref{eq:WF} and $ N^* ( \RR_+ \times \Delta ) \subset T^* ( \RR_+ \times M \times M ) 
\setminus 0 $ is the conormal bundle of $ \RR_+ \times \Delta $, that
is, the annihilator of $ T (\RR_+ \times \Delta ) \subset T( 
\RR_+ \times M \times M ) $.

The condition \eqref{eq:WFK} is satisfied for $ K $ given by 
the Schwartz kernel of $ e^{-itP} = \varphi_{-t}^* $, $ t > 0 $, 
because for Anosov flows powers of the Poincar\'e map
\eqref{eq:Poinc}, $ \mathcal P_\gamma^k $, cannot have $ 1 $ as 
an eigenvalue. Hence
\eqref{eq:flatr1} makes sense and \eqref{eq:ABG} holds.

We now observe that for $ \Im \lambda \gg 1 $ and for $ 0 < t_0 < 
\min_{\mathcal G}  \ell_\gamma $,
\begin{equation}
\label{eq:zeta2ph}\frac 1 i \int_{t_0}^\infty \tr^\flat e^{ - i t ( \mathbf P_j - \lambda ) } dt 
= \zeta_j( \lambda)^{-1}  \frac{d}{d \lambda} \zeta_j ( \lambda )  .
\end{equation}
In view of \eqref{eq:factze}, to show that $ \zeta_R $ has a meromorphic
extension, it is enough to show that each $ \zeta_j $ has a holomorphic
extension. That in turn follows from a meromorphic extension of
$ \zeta'_j/\zeta_j$ with simple poles and integral residues. 

Hence we need to show that the left hand side of \eqref{eq:zeta2ph} 
has a meromorphic extension to $ \CC $ with poles of finite rank.
This is clearly related to \eqref{eq:contiP}. A formal manipulation
suggests that the left hand side of 
\eqref{eq:zeta2ph} is given by 
\begin{equation}
\label{eq:ph2P}
- e^{i t_0 \lambda}  \tr^\flat\left(  \varphi_{-t_0}^* ( \mathbf P_j - 
\lambda)^{-1}\right)  = - e^{i t_0 \lambda} \langle \,\jmath^* 
(\varphi_{-t_0} \otimes {\rm{id}})
^* \mathbf K_j  ( \lambda ) 
, 1 \rangle, 
\end{equation}
where $ \mathbf K_j  ( \lambda ) $ is the Schwartz kernel of 
$ ( \mathbf P_j - \lambda)^{-1} $ (which makes sense
for all $ \lambda $ in view of Theorem \ref{t:Fred}), 
$ \jmath ( x ) = ( x , x ) $ and $ \langle \bullet, \bullet \rangle $ is
the distributional pairing $ \mathcal D'( M ) \times \CI ( M) \to \CC $.
Just as in \eqref{eq:WFK} this is justified by showing that for $ 0 < t_0 
< \min_{\gamma \in \mathcal G} \ell_\gamma  $, 
\begin{equation}
\label{eq:WFK0}
\WF ( ( \varphi_{-t_0} \otimes {\rm{id}} )^* \mathbf K_{j} ( \lambda)  ) \cap 
N^* \Delta = \emptyset, \ \ \ \Delta := \{ ( x, x ) : x \in M \}. 
\end{equation}
We used the notation $ ( \varphi_{-t_0} \otimes {\rm{id}} ) ( x , y ) := 
( \varphi_{-t_0} ( x ) , y ) $ and the pullback by that diffeomorphism
of $ M \times M $ is always well defined \cite[Theorem 8.2.4]{H1}.
We have to be careful at the poles of $ \mathbf K_j (\lambda ) $
but a decomposition similar to \eqref{eq:Jordan} shows that we
can separate a holomorphic and singular parts. On the singular part
$ \tr^\flat = \tr_{ H_{rG}} $ which kills all the higher order
poles (nilpotency of Jordan blocks) and gives an integral residue.

Condition \eqref{eq:WFK0} is an immediate
 consequence of \cite[Proposition 3.3]{zeta} (though it takes a moment
 to verify it: the pullback by $ ( \varphi_{-t_0} \otimes {\rm{id}} )$ 
 shifts things away from the diagonal):
\begin{thm}
\label{t:zeta}
Let  $ \mathbf K_j ( \lambda ) \in 
\mathcal D' ( M ) $ be the Schwartz kernel of
$ ( \mathbf P_j - \lambda )^{-1} : \CI ( M , \mathcal E_0^j ) 
\to \mathcal D' ( M , \mathcal E_0^j )  $. Then away from the poles of
$ \lambda \mapsto \mathbf K_j ( \lambda ) $,
\begin{equation}
\label{eq:tzeta}
\begin{gathered}
( x, y ; \xi , -\eta ) \in \WF ( \mathbf K_j ( \lambda ) ) 
\ \Longrightarrow \
( x, \xi, y , \eta ) \in 
\Delta(T^*M)\cup \Omega_+\cup (E_u^*\times E_s^*) , \ \ 
\end{gathered} \end{equation}
where $ \Delta ( T^* M ) := \{ ( \rho, \rho ) : \rho \in T^* M \} \subset 
T^* M \times T^* M  $ and 
\[ \Omega_+:=\{(e^{tH_p}(x,\xi),x,\xi) \; : \; t\geq 0,\ p(x,\xi)=0\}, \ \
p ( x, \xi) := \xi ( X_x ) . \]
\end{thm}

Away from the conic sets $ E_u^* \times E_s^* $ this follows from 
a modification of results of Duistermaat--H\"ormander on 
propagation of singulaties \cite[Proposition 2.5]{zeta}. Near 
$ E_s^* $ we use a modification of Melrose's propagation result \cite{mel}
for radial sources \cite[Proposition 2.6]{zeta}, and near $ E_u^* $, 
his propagation result for radial sinks  \cite[Proposition 2.7]{zeta}. 
The property \eqref{eq:mzation} is crucial here and implies the relation 
between $ r $ and $ \Im \lambda $ in Theorem \ref{t:Fred}. 
Except for the fact that $ E_u^* $ and $ E^*_s $ are typically very irregular
as sets, these are the same estimates that Vasy \cite{vasy1} used
to prove Theorem \ref{t:7} presented \S \ref{vasy}.

Retracing our steps through \eqref{eq:WFK0},\eqref{eq:ph2P},\eqref{eq:zeta2ph}
and \eqref{eq:factze} we see that we proved the theorem of
Giulietti--Liverani--Pollicott \cite{glp} settling Smale's conjecture \cite[\S II.4]{Sm}
in the case of compact manifolds. For the full conjecture in the non-compact
case proved by a highly nontrivial elaboration of the above strategy, see Dyatlov--Guillarmou \cite{rnc}:
\begin{thm}
The Ruelle zeta function, $ \zeta_R ( \lambda ) $, defined for $ \Im \lambda
\gg 1 $ by \eqref{eq:defRuz}, continues meromorphically to $ \CC $.
\end{thm}

What the method does not recover is the order of $ \zeta_R ( \lambda ) $ in the
case when $M $ and $ X $ are real analytic \cite{Ru},\cite{Fr2}. That may be related to issues
around Conjecture \ref{c:hypo}.

Pollicott--Ruelle resonances of $ \mathbf P_j $'s 
are exactly the zeros of the entire functions $ \zeta_j ( \lambda ) $'s. 
The simple Definition \ref{d:PR} using wave front set characterization 
led to a short proof of the following fact \cite{zazi}: suppose
that $ ( \Sigma, g )  $ is an oriented negatively curved Riemannian surface. Then 
\begin{equation}
\label{eq:zazi}   \zeta_R ( \lambda ) = c \lambda^{| \chi ( \Sigma )|} 
( 1 + \mathcal O ( \lambda ) )  , \ \ c \neq 0 , \ \ \lambda \to 0 , \end{equation}
where $ \chi ( \Sigma ) $ is the Euler characteristic of $ \Sigma $. 
In particular this implies that lengths of closed geodesics determine 
the genus of the surface linking dynamics and topology. Previous to \cite{zazi}
that was known only for surfaces of constant curvature. 

Hence \eqref{eq:zazi}
provides evidence that results valid in rigid (that is, locally symmetric) 
geometries may be valid in greater generality.
Thus,  consider a compact oriented negatively curved Riemannian manifold $ ( \Sigma , g ) $. Following Fried \cite{Fr-Anator},\cite{Fr2}, for 
a representation $ \alpha : \pi_1 ( S^* \Sigma ) \to GL(N, \CC )  $, we define
\begin{equation}
\label{eq:azeta}
\zeta_\alpha ( \lambda ) := \prod_{ \gamma \in \mathcal G} 
\det ( I - \alpha ( \gamma ) e^{ i \lambda \ell_\gamma } ) . 
\end{equation}
The proof of meromorphy of $ \zeta_R $ given in \cite{zeta} and sketched above
gives also the meromorphic continuation of $ \zeta_\alpha$: we only need to 
change the vector bundles on whose sections $ \frac 1 i \mathcal L_X $ acts.
The following statement was proved by Fried \cite{Fr-Anator} in the case of 
manifolds of {\em constant} negative curvature and conjectured (in 
an even more precise form) in \cite[p.181]{Fr2} for more general manifolds.
We refer to his papers \cite{Fr-Anator},\cite{Fr2} for precise definitions 
of the objects involved in the statement:
\begin{conj}
\label{c:Fried}
Suppose that $ ( \Sigma, g ) $ is a negatively curved oriented compact 
Riemannian manifold of dimension $ n > 2 $ and that $ \alpha $ is
an acyclic unitary representation of $ \pi_1 ( \Sigma ) $. Then 
\begin{equation}
\label{eq:Fried}
| \zeta_\alpha ( 0 )|^{ (-1)^{n-1} } = | T_\alpha ( \Sigma ) |^2, 
\end{equation}
where $ \zeta_\alpha $ is defined by \eqref{eq:azeta} and 
$ T_\alpha ( \Sigma ) $ is the analytic torsion of $ \Sigma $.
\end{conj}
The analytic torsion $ T_\alpha ( \Sigma ) $ was defined by Ray and Singer 
using eigenvalues of an $ \alpha$-twisted Hodge Laplacian. 
Their conjecture that $ T_\alpha ( \Sigma ) $ is equal to the 
Reidemeister torsion, a topological invariant, was proved independently by 
Cheeger and M\"uller. Hence \eqref{eq:Fried} would link 
dynamical, spectral and topological quantities. 
In the case of locally symmetric manifolds a more precise version 
of the conjecture 
was recently proved by Shen 
\cite[Theorem 4.1]{Sh} following earlier contributions
by Bismut~\cite{Bismut} and Moskovici--Stanton~\cite{Moskovici}.

\subsection{Distribution of 
Pollicott--Ruelle resonances}
\label{resfreeA}

Exponential decay of correlations \eqref{eq:expdcor} 
was established  by Dolgopyat \cite{dole} 
 for systems which include geodesic flows on 
negatively curved compact surfaces and by Liverani  \cite{liver}
 for all contact Anosov flows, hence in particular for 
 geodesics flows on negatively curved compact manifolds -- see
 \eqref{eq:contact},\eqref{eq:Sigma}. These two papers
 can also be consulted for the history of the subject. 
 See also Baladi--Demers--Liverani \cite{badeli} for some very 
 recent progress in the case of hyperbolic billiards.

Comparing \eqref{eq:expdcor} and \eqref{eq:rhotoP} shows that
if we know a resonance free strip with a polynomial bound on 
$ ( P - \lambda)^{-1} $ in that strip then we obtain
exponential decay of correlations. Hence we define
the following two {\em spectral} (though really ``scattering")
quantities: the spectral gap (for simplicity we consider only the
scalar case):
\begin{equation}
\label{eq:nu1}    \nu_1 = \sup \{ \nu \, : \, \Res ( P ) \cap \{ \Im \lambda > - \nu \} = 
\{ 0 \} \} , \end{equation}
and the {\em essential} spectral gap:
\begin{equation}
\label{eq:nu0} \nu_0 = \sup \{ \nu \, : \,  | \Res ( P ) \cap \{ \Im \lambda > - \nu \} | 
< \infty \} . 
\end{equation}
Here $ \Res ( P ) $ denotes the set of resonances of $ P$. Obviously if $ \nu_0 > 0 $ and if we know that there are no non-zero real resonances, then $ \nu_1 > 0 $. 

We have the following result which implies exponential decay 
of correlations for contact Anosov flows:
\begin{thm}
\label{t:tsuji}
Suppose that $ ( M , \alpha )  $ is a compact contact manifold
and that the Reeb vector field \eqref{eq:contact} generates 
an Anosov flow. Then in the notation of \eqref{eq:nu0} and \eqref{eq:defga0},
\begin{equation}
\label{eq:tsuji}
{\textstyle{\frac12}} \gamma_0 \leq \nu_0 < \infty .
\end{equation}
Moreover, for some $ s $
\begin{equation}
\label{eq:tsuji1} 
\forall \, \delta > 0 \ \exists \, C, N \ \ 
\begin{array}{l} 
\| ( P - \lambda)^{-1} \|_{ H^{s} \to H^{-s} } \leq C ( 1 + |\lambda |)^{N}, \\ 
\ \ \text{for } \ | \Re \lambda | > C , \ \Im \lambda > - \nu_0 + \delta . 
\end{array} 
\end{equation}
\end{thm}
One could estimate $ s $ more precisely depending on the width of the strip
which would give better decay of correlations results.

The lower bound in \eqref{eq:tsuji} was proved by Tsujii \cite{Ts} who
improved the result of Liverani at high energies. It also follows from the
very general results of \cite{NZ3} which apply also in quantum scattering.
The method of \cite{NZ3} is crucial in obtaining the second part of Theorem \ref{t:damped}
-- see also Conjecture \ref{c:8}. 

The finiteness of 
$ \nu_0 $ follows from a stronger statement in Jin--Zworski \cite{loz}:
for {\em any} Anosov flow there exist strips with infinitely many Pollicott-Ruelle
resonances. 

One immediate consequence of Theorem \ref{t:tsuji} is the analogue of
Theorems \ref{t:3} and \ref{t:rel} in the setting of Anosov flows:  
suppose we enumerate the non-zero resonances so that
$  \Im \lambda_{j+1} \leq \Im \lambda_{j} < 0 $\footnote{There are no non-zero
real resonances for Anosov flows with smooth invariant measures and in particular
for contact Anosov flows -- for a direct microlocal 
argument for that see \cite[Lemma 2.3]{zazi}.}. There exist distributions 
$ u_{j,k}, v_{j,k} \in {\mathscr D}' ( M ) $, $ 0 \leq k
\leq K_j $, 
such that, for any $ \delta > 0 $, there exists $J(\delta)\in\NN$ such that  
for any $ f , g \in \CI ( M ) $
\begin{equation}
\label{eq:exprfg}
\begin{split}
 \rho_{f,g} ( t ) & = 
\int_M f   dm  \int_M g
  dm 
 + \sum_{ j=1}^{ J ( \delta ) }  \sum_{ k=1}^{K_j } t^{k}
e^{ - i t \mu_{j} } \langle u_{j,k},  f \rangle 
\langle v_{j,k} , g \rangle + \mathcal O_{f,g} ( e^{ - t(
  \nu_0 - \delta )} ), 
\end{split} 
\end{equation}
for $ t > 0 $. Here $ dm $ is the invariant density coming from the contact
form and normalized to have $ \int_M dm = 1 $ and $ \langle \bullet, \bullet
\rangle $ denotes distributional pairing. This refines \eqref{eq:expdcor} and provides a rigorous
version of \eqref{eq:expcor}.

We conclude with a review of counting results for resonances. We put
\[  N ( \Omega ) := |\Res ( P ) \cap \Omega |  , \ \  \Omega \Subset \CC , \]
and denote $ n:= \dim M $. 

The first bound was proved by Faure--Sj\"ostrand \cite{fa-sj} 
and it holds for general Anosov flows on compact manifolds: for any $ \gamma > 0 $, 
\begin{equation}
\label{eq:fas}
N \left( [ r -  r^{\frac12} , r + r^{\frac12} ] - i [ 0 , \gamma ] \right) = o ( r^{ n - \frac12} ) . 
\end{equation}
For Anosov {\em contact} flows a sharp bound was proved by Datchev--Dyatlov--Zworski
and it says that
\begin{equation}
\label{eq:ddz}
N \left( [ r - 1, r + 1 ] - i [ 0 , \gamma ] \right) = \mathcal O ( r^{\frac{n-1}2} ) , 
\end{equation}
which improves \eqref{eq:fas}, giving 
$ N ( [ r - r^{\frac12} , r +  r^{\frac12} ] - i [ 0 , \gamma ] ) = O ( r^{ \frac n 2 } ) $, 
in the contact case. These bounds are sharp in all dimensions as shown
by the very precise analysis of Pollicott--Ruelle resonances on 
compact hyperbolic quotients by Dyatlov--Faure--Guillarmou \cite{DGF}. 
Although these resonances have been known for a long time in the case
of surfaces, in higher dimensions some new structure was discovered in \cite{DGF},
see also Guillarmou--Hilgert--Weich \cite{g-w}.

Under a pinching condition on minimal and maximal expansion rates $ \nu_{\max} < 
2 \nu_{\min} $,  
Faure--Tsujii \cite{fats} proved a sharp lower bound in the contact case: 
for any $ \epsilon > 0 $ and any $ \delta > 0 $ there exists $ c > 0 $ such that
\begin{equation}
\label{eq:fats}
N \left( [ r - r^\epsilon , r + r^\epsilon ] - i {\textstyle{\frac12}} [ 0 , \nu_{\max} + \delta ]\right)
\geq c r^{ \frac {n-1} 2 + \epsilon } .
\end{equation}
This should be compared to Dyatlov's Weyl asymptotics shown in Figure \ref{f:iv}
-- see \cite[\S 1]{dy}. The only, and very weak, bound in the general Anosov
case follows from a local trace formula \cite[Theorem 1]{loz} (similar in spirit to 
Sj\"ostrand's local trace formula for resonances \cite{SjL}): for every 
$ 0 < \delta < 1 $ there exists $ A = A_\delta $ such that 
in the Hardy--Littlewood notation\footnote{$ \Omega ( f ) = g $ if it is 
{\em not} true that $ f = o(|g|) $},
\begin{equation}
\label{eq:loz}
N \left( [ 0 , r ] - i [ 0 , A ] \right) = \Omega ( r^\delta ) . 
\end{equation}
However, Naud observed \cite[Appendix B]{loz} that for suspensions of 
certain Anosov maps,   the bound $ \mathcal O (r ) $ holds, 
in which case \eqref{eq:loz} is not too far off.

This state of affairs shows that many problems remain and the non-compact
case seems to be completely open.


\medskip

\noindent
{\sc Acknowledgements.} I would like thank my collaborator on 
\cite{res}, Semyon Dyatlov, for allowing me to use some of the
expository material from our book, including figures, in the preparation of this survey
and for his constructive criticism of earlier versions. Many thanks go also to 
David Borthwick, Alexis Drouot, Jeff Galkowski, Peter Hintz, Jian Wang and Steve Zelditch for their comments and corrections. 
I am also grateful to Hari Manoharan, Ulrich Kuhl, Plamen Stefanov, David Bindel, Mickael Chekroun, Eric Heller and David Borthwick for allowing me to use Figures \ref{f:hm},\ref{f:kuhl},\ref{f:sphere},\ref{f:3},\ref{f:PNAS},\ref{f:galk},\ref{f:borth1},\ref{f:dy1} and \ref{f:concent}, respectively.
The work on this survey was supported in part by the National Science Foundation 
grant DMS-1500852. 

\def\arXiv#1{\href{http://arxiv.org/abs/#1}{arXiv:#1}}

\end{document}